\documentclass[12pt]{article}
\usepackage{amsmath,amsthm,amsfonts,amssymb,latexsym,enumerate,graphicx,mathrsfs,mathabx,cancel}
\usepackage[utf8]{inputenc}
\usepackage[dvipsnames]{xcolor}
\usepackage{hyperref}
\hypersetup{colorlinks, citecolor=blue, linkcolor=red}
\usepackage{stmaryrd}

\usepackage{color}

\usepackage[normalem]{ulem}

\usepackage{tikz}
\usetikzlibrary{arrows,decorations.pathmorphing,decorations.pathreplacing}
\usetikzlibrary{shapes.geometric,positioning}
\usepackage{tikz-cd}
\tikzset{
  math to/.tip={Glyph[glyph math command=rightarrow]},
  math double/.tip={Glyph[glyph math command=twoheadrightarrow]},
  loop/.tip={Glyph[glyph math command=looparrowleft, swap]},
  weird/.tip={Glyph[glyph math command=Rrightarrow, glyph length=1.5ex]},
  pi/.tip={Glyph[glyph math command=pi, glyph length=1.5ex, glyph axis=0pt]},
  pool/.tip={Glyph[glyph math command=looparrowleft]},
}

\newtheorem{thm}{Theorem}[section]
\newtheorem{cor}[thm]{Corollary}
\newtheorem{prop}[thm]{Proposition} 
\newtheorem{lem}[thm]{Lemma}
\newtheorem{conv}[thm]{Convention}

\newtheorem{mainthm}{Theorem}

\newtheorem{conj}[mainthm]{Conjecture}

\theoremstyle{definition}
\newtheorem{defn}[thm]{Definition}

\theoremstyle{remark}

\newcommand{\define}{\textit}
\renewcommand{\vec}{\underline}

\makeatletter
\@tfor\next:=abcdefghijklmnopqrstuvwxyzABCDEFGHIJKLMNOPQRSTUVWXYZ\do{%
	\def\command@factory#1{%
		\expandafter\def\csname cal#1\endcsname{\mathcal{#1}}
		\expandafter\def\csname frak#1\endcsname{\mathfrak{#1}}
		\expandafter\def\csname scr#1\endcsname{\mathscr{#1}}
		\expandafter\def\csname bb#1\endcsname{\mathbb{#1}}
		\expandafter\def\csname rm#1\endcsname{\mathrm{#1}}
		\expandafter\def\csname bf#1\endcsname{\mathbf{#1}}
	}
	\expandafter\command@factory\next
}
\makeatother

\makeatletter
\DeclareFontEncoding{LS1}{}{}
\DeclareFontSubstitution{LS1}{stix}{m}{n}
\DeclareMathAlphabet{\mathscr}{LS1}{stixscr}{m}{n}
\makeatother

\newcommand{\onto}{\ensuremath{\twoheadrightarrow}}
\newcommand{\bk}[1]{{\left\langle #1 \right\rangle}}
\newcommand{\into}{\ensuremath{\hookrightarrow}}

\newcommand{\Stab}{{\rm Stab}}

\newcommand{\freegp}[1]{{\bbF(#1)}}
\newcommand{\nsgp}{\triangleleft}
\newcommand{\charsgp}{\nsgp_{\mathrm c}}
\newcommand{\fidx}[1]{\stackrel{\mathrm{f.i.}}{#1}}
\newcommand{\semidirect}[1]{\rtimes_{#1}}
\newcommand{\torus}[1]
{{{\bbT_{#1}}}}
\newcommand{\verts}[1]{{{\textbf{V}(#1)}}}
\newcommand{\edges}[1]{{{\textbf{E}(#1)}}}

\newcommand{\aut}[1]{{\mathrm{Aut}\left(#1\right)}}
\newcommand{\Aut}[1]{{\mathrm{Aut}\left(#1\right)}}
\newcommand{\Autfo}[1]{{\mathrm{Aut_{fo}}\left(#1\right)}}
\newcommand{\autfo}[1]{{\mathrm{Aut_{fo}}\left(#1\right)}}

\newcommand{\out}[1]{{\mathrm{Out}\left(#1\right)}}
\newcommand{\Outfo}[1]{{\mathrm{Out_{fo}}\left(#1\right)}}
\newcommand{\outfo}[1]{{\mathrm{Out_{fo}}\left(#1\right)}}
\newcommand{\Bass}{\mathbf{Bass}} 

\newcommand{\Inn}[1]{{\mathrm{Inn}\left(#1\right)}}

\newcommand{\StalGraph}[1]{\mathcal{G}\left(#1\right)}
\newcommand{\BasedStalGraph}[2]{{\left(\mathcal{B}\left(#1\right),#2\right)}}
\newcommand{\ad}[1]{{\mathrm{ad}_{#1}}}

\renewcommand{\hat}{\widehat}

\newcommand{\Hom}[2]{{\mathrm{Hom}\left(#1,#2\right)}}
\newcommand{\elev}[2]{\mathrm{Elev}_{#2}\left(#1\right)}
\newcommand{\G}{{\torus \alpha}}
\newcommand{\link}[1]{{\mathscr{l}}_{#1}}

\title{Unipotent linear suspensions of free groups}

\author{Fran\c{c}ois Dahmani, Nicholas Touikan}
\begin{document}
\maketitle
\begin{abstract}
    Motivated by the study of the conjugacy problem for outer automorphisms of free groups,  we develop the algorithmic theory of the free-by-cyclic groups  produced by unipotent linearly growing automorphisms of finitely generated free groups. 
    
    We compute canonical splittings of these suspensions as well as their subgroups. We compute their automorphism groups. We show that this class of suspensions is effectively coherent. We solve the mixed Whitehead problem in these suspensions and show that their subgroups all satisfy the Minkowski property, i.e. that torsion in their outer automorphism group is faithfully represented in some computable finite quotients.
    
    An application of our results is a solution to  the conjugacy problem for exponentially growing outer automorphisms of  free groups whose polynomially growing part is unipotent linear.

  \end{abstract}

\tableofcontents

    \section{Introduction}
    \subsection{Reducibility and the failure of semisimplicity in $\out{F_n}$}
        
        A single automorphism $\alpha$ of a group $G$ will typically have several interesting dynamical, algebraic, and geometric features associated to  its conjugacy class $[\alpha]$ in the ambient automorphism group $\aut G$.
        When one attempts to extract and  analyze  these features, one often encounters a situation of  reduction to  subgroups that encapsulate a specific feature.

        A classical example is $GL_n(\bbC)$, the automorphism group of the vector space $\bbC^n$. The dynamics of $A \in GL_n(\bbC)$ are fully captured by examining the action of $A$ on generalized eigenspaces. The minimal Jordan blocks in a Jordan  form decompose $A$ into a sum of irreducible automorphisms, which are easily understood.
        
        A more sophisticated example is the (extended) Mapping Class Group of a closed orientable surface $\Sigma$, which classically plays a role of exotic analogue of the lattices $GL_n(\bbZ)$.   It is the group of isotopy classes of homeomorphisms of the surface and the Dehn-Nielsen theorem identifies this group with $\out{\pi_1(\Sigma)}$. Here the outer automorphism group $\out{G}$ is the quotient of $\Aut{G}$ by the normal subgroup of the inner automorphisms (or conjugations). This construction makes $\out{\pi_1(\Sigma)}$ well-defined even without a choice of basepoint for $\Sigma$. In particular outer automorphisms tend to capture the dynamical, algebraic, and geometric features that are of the most interest.
        
        The Nielsen-Thurston classification of the elements of Mapping Class Groups divides mapping classes into irreducible mapping classes (i.e. pseudo-Anosov) and reducible mapping classes: those have a power that preserves a decomposition of $\Sigma$ into a union of invariant sub-surfaces. One then has a reduction to the mapping class groups of the subsurfaces:  a power of reducible mapping class can be described as disjoint support mapping classes with some twist parameter for gluing subsurfaces together.
        
        Thus both $GL_n(\bbC)$ or mapping class groups have the following \emph{semisimplicity} property: every element has a power that can be described as  
        a collection of irreducible automorphisms on smaller objects, with gluing parameters. For  $GL_n(\bbC)$ (unlike the Mapping Class groups), these parameters are always trivial, and the restrictions to subspaces is sufficient. 
        
        This decomposition facilitates the conjugacy problem, that aims to describe sufficient invariants to  determine whether two automorphisms are conjugate. Unfortunately, the outer automorphism group of free groups does not satisfy such a semisimplicity property.  
        
        Following \cite{bestvina_train_1992}, $\out{F}$, the outer automorphism group of a finitely generated free group $F$, also contains irreducible elements, and reducible ones: those have a power that preserves a conjugacy class of a free factor.
        If a reducible outer automorphism $\Phi$ happens to preserve a free factorization of $F$ then indeed $\Phi$ can be well-understood much like reducible mapping classes, by its restrictions to invariant free factors. Unlike for mapping class groups or linear groups, however, reducibility in $\out{F}$ does not imply the preservation of a nice decomposition. 
        
        A powerful perspective to make sense of reducibility in $\out{F}$ is via the study of suspensions of automorphisms. If $\phi$ is an automorphism of $F$, its \emph{suspension} is  the  semi-direct product $F \rtimes_\phi \bk t$ in which the conjugation on $F$ by the generator $t$  $\phi$. Proposition \ref{prop:conj-crit} gives an equivalence of categories
        between $\out{F}$ and suspensions of $F$ where conjugacy corresponds to fibre and orientation preserving isomorphisms.

        It was proved in \cite{dahmani_li_relative_2022, ghosh_relative_2018} that, when $F$ is a finitely generated free group,  $F\hookrightarrow F\rtimes \bbZ$ possesses a geometry that detects possibilities of reduction: it is relatively hyperbolic with respect to some sub-suspensions of some subgroups of $F$ (its parabolic subgroups), whose conjugacy classes are preserved by $\phi$, on which it induces polynomial growth of conjugacy classes. This leads to the following \emph{suspension classification of elements of $\out{F}$}: their suspensions either
        \begin{enumerate}
            \item do not admit non-trivial relatively hyperbolic structure, in which case they are called \emph{NRH groups}, and the corresponding automorphism is polynomially growing, or
            \item are hyperbolic groups, in which case the automorphism is atoroidal and exponentially growing, or
            \item are hyperbolic relative to a non-empty family of NRH subgroups, in which case the automorphism is exponentially growing, but there are subgroups of $F$ on which the restriction is polynomially growing.
        \end{enumerate}
        This classification is not compatible with standard reducible/irreducible classification that has proven itself so useful in the study of $\out{F}$ and the dynamics of actions on outer space, but it does provide a concept of reducibility that allows us to recast semisimplicity in a way that is amenable to the conjugacy problem in $\out{F}$.
        
        Recall that the free group analogue of a decomposition into direct summands for a vector space, or a decomposition into subsurfaces for a surface, is a free product decomposition. These all correspond to the semisimple regime, where reducible automorphisms are just direct sums of simpler irreducible automorphisms.
        
        A free product $\Gamma = K_1*\ldots*K_n$ provides a way to embed the groups $K_i$ into a larger group $\Gamma$ in such a way that $\Gamma$ can be fully understood in terms of these free factors. A generalization of a free product of $K_1,\ldots,K_n$ is a group $\Delta = (\Delta;K_1,\ldots,K_n)$ that is hyperbolic relative to $K_1,\ldots,K_n$.
        
        Our version of reducibility in $\out{F}$ therefore occurs when the suspension of an automorphism is relatively hyperbolic and the irreducible components are isomorphic to suspensions of polynomially growing automorphisms of subgroups of $F$. The failure of semisimplicity comes from the fact that in the general relatively hyperbolic case, the whole group cannot be fully constituted from these irreducible components. What this means is that to solve the conjugacy problem for reducible automorphisms it is not enough to simply to solve the conjugacy problem for the polynomially growing restrictions, we also need to take into account the substantial interactions of our irreducible components with the ambient group.
        
        A recurrent theme in relative hyperbolicity is the successful reductions of global algorithmic problems to algorithmic problems in parabolic subgroups and in \cite{dahmani_touikan_reducing_2021} the authors give a complete set of extra algorithmic problems that need to be solved in the irreducible components in order to determine the conjugacy of two reducible automorphisms. In this paper we show how to solve these problems for the base case of suspensions of unipotent linearly growing automorphisms, thus overcoming the failure of semisimplicity in $\out{F}$, and as an immediate consequence of \cite{dahmani_touikan_reducing_2021}  and the main results of this paper we get
        
         \begin{thm}\label{thm:CP-rel-ulg}
             Let $F$ be a finitely generated free group. Dehn's Conjugacy Problem in $\out F$ among exponentially growing  outer automorphisms whose polynomial growth part is unipotent linear is algorithmically solvable. 
         \end{thm}
         The authors take this opportunity to express their hope that this first step will be successfully followed by the study of suspensions of automorphisms with higher degrees of polynomial growth with the  final goal of a complete solution to the conjugacy problem in $\out{F}$.
         
        \subsection{Growth of automorphisms}
        
        Let $G$ be a group acting by isometries on a metric space $(X,d)$.  An automorphism $\alpha \in \aut {G}$ is said to have polynomial growth (for $(X,d)$) if, for all $g\in G$,  there exists a polynomial $P\in \bbZ[X]$ such that for all $n\in \bbN$,  \[ \inf_{x\in X} d(x,\alpha^n(g)x) \leq P(n).\]  
        
        In many cases (and in all this paper), the space $(X,d)$  will be $G$ equipped with the word metric of a finite generating set.  
        For $g\in G$,  the integer $\|g\| =  \inf_{h\in G} d(h,gh)$ denotes the length of a shortest element in its conjugacy class. The quantity $\inf_{x\in X} d(x,\alpha^n(g)x)$ is then simply $\|\alpha^n(g)\|$. 
          An automorphism of $G$ is therefore \emph{linearly growing} (for a word metric $(G,d)$) if
        for any $g \in G$, there  are $\lambda, \mu \in \bbN$, such that, for all $n\geq 1$, 
        \[
          \|\alpha^n(g)\| \leq \lambda n +\mu.
        \]

         Linear growth is a property  shared by all
        automorphisms $\alpha'  \in \aut{G}$ in the same outer class as
        $\alpha$ in $\out{G}$, and that  persists under change of
        finite generating set  of $G$. Being linearly growing is  a property of conjugacy classes of outer automorphisms, in $\out{G}$.  
         
        Our main interest is for $G=F$ a  finitely generated free group,  (we will sometimes insist on its rank by calling it $F_m$), with the word metric given by a chosen basis.
        
        The most relevant elementary example is a  Dehn twist. Write $F$ as a free product $F= A*B$, and to define $\alpha$, apply to $A$ the identity, and to  $B$ a conjugation by an element of $A$. This is an automorphism that has linear growth.  Write $F$ as a free HNN extension $F=A*_{\{1\}}$, of stable letter $s$, and define another automorphism $\alpha$ to be the identity on $A$ and to send $s$ on $sa$ for some $a\in A$.  Again such $\alpha$ has linear growth. More generally a composition of compatible (in a certain sense) Dehn twists is a linearly growing automorphism. 
        
        Another example is the case of automorphisms of fundamental groups of closed surface groups: for a closed surface $\Sigma$,  any Dehn twist along a simple closed curve of  $\Sigma$ defines a linearly growing outer automorphism of $\pi_1(\Sigma)$, as well as a composition of Dehn twists along disjoint curves. 
        There is more: by classification of mapping classes, any automorphism of a closed surface group that induces a polynomially growing automorphism on an invariant finitely generated subgroup, induces actually a linearly growing automorphism on this subgroup.  If the subgroup is of infinite index, it is free, and the situation is covered by the former example, though more classical. 

       
       Most automorphisms of non-abelian free groups and of (non-virtually abelian) closed surface groups  are not polynomially growing (most  are exponentially growing), nevertheless the polynomially growing and linearly growing ones are important. 
       


       

        Among polynomially growing automorphisms, one may define those that are unipotent: consider $\alpha \in \aut{F_m}$ of polynomial growth, and its  image $\bar \alpha $ in $GL_m(\bbZ)$ through the abelianization map $F_m \to \bbZ^m$, we say that $\alpha$ is unipotent if $\bar \alpha$ is conjugate to an upper triangular matrix with ones on the diagonal (see \cite[p. 564]{bestvina_tits_2000}). This subclass is peculiarly more equipped with tools. On the other hand, it is not a rare condition: for all polynomially growing automorphism of  a free group of  rank $m$,  its power by a uniform exponent (specifically $|GL_m(\bbZ/3)|$)  
        is unipotent \cite[Corollary 5.7.6]{bestvina_tits_2000}.  
        Observe also that our earlier Dehn twists examples are easily seen to be unipotent.  In contrast, a non-trivial finite order outer automorphism is linearly growing but is not unipotent. 

        In    this paper we develop the structure and algorithmic theory of suspensions of unipotent linear automorphisms of free groups, the first instance of unipotent polynomial automorphisms. 
        
        \subsection{Mapping tori}
        
        Given any $\alpha \in \aut{F}$ it is possible to form its
        \emph{mapping torus} or
        \emph{suspension}: if we write the conjugation as $f^t= t^{-1} f t$, it is the semidirect product  \[ \torus \alpha=F \semidirect\alpha \bk t = \bk{F, t \mid
            f^t=\alpha(t); f \in F}.
        \]
        We call such an explicit decomposition as a semidirect product a
        \emph{fibration} and we call the normal semidirect factor $F$ the
        \emph{fibre}. We use this terminology since $F = \ker(\phi)$ for
        some $\phi \in \Hom{F \semidirect\alpha \bk t}{\bbZ}$. In the case
        where $\alpha \in \aut F$ is linearly growing,  we call $F \semidirect\alpha \bk t$ a
        \emph{linear suspension.}

        More generally, one may define the set of polynomially growing automorphisms and  outer automorphisms 
        and for such an automorphism, we
        would call the resulting  suspension a \emph{polynomial suspension}. 
        

        There are many results (e.g. \cite{brinkmann_hyperbolic_2000, dahmani_li_relative_2022})  indicating that the algebraic structure of the group $ \Gamma= F \semidirect\alpha \bk{t}$ gives  information about
        the outer class of $\alpha$. It is almost immediate to see that
        if $\alpha,\beta \in \aut{F}$ have conjugate images in $\out{F}$, then we have an
        isomorphism of suspensions\[ \torus\alpha \simeq \torus\beta.
        \]
        
        The converse, however, is not true. For certain  $\alpha$ in
        $\aut{F}$,
        there may be infinitely many other
        fibrations of $\torus \alpha = F\semidirect\alpha\bk t$ (see for
        example \cite{mecham_hyperbolic_2009,dowdall_dynamics_2015,button_mapping_2007})
        corresponding to kernels of elements of $\Hom{\torus \alpha} \bbZ$ with finitely
        generated kernel.  The ranks of the fibres that arise in
        this case can be unbounded, though the dynamics of twisting
        automorphisms maintain similarities. We leave it to a reader who wants
        to grasp conjugacy in $\out{F}$ to prove the following criterion for
        themselves. An isomorphism between two semidirect products $F\rtimes_\alpha \langle t\rangle$, and $F\rtimes_\beta \langle s\rangle$  (marked by their fibres and the elements $s,t$) is fibre and orientation preserving if it sends $F$ to $F$ and $t$ in $sF$. 
        
        \begin{prop}[Conjugacy criterion]\label{prop:conj-crit}
          Let $\alpha$ and $\beta$ be automorphisms of a group $F$.    Then $\alpha$ and $\beta$ have  conjugate images in $\out{F}$ if and only if          there is a fibre and orientation preserving isomorphism\[
            F\semidirect\alpha\bk t \stackrel{\sim}{\to}
            F\semidirect\beta\bk s.\]
            
            Moreover,  $\alpha$ and $\beta$ are conjugate in $\aut{F}$ if and only if 
            there is a fibre-preserving isomorphism   $F\semidirect\alpha\bk t \stackrel{\sim}{\to}
            F\semidirect\beta\bk s$   that sends $t$ on $s$. 
        \end{prop}
        
      
        For the case of an arbitrary automorphism $\alpha$ of a group $G$ that is either free (finitely generated) or fundamental group of  a closed surface group, we know that the mapping torus $G\rtimes_\alpha \bbZ$ is  hyperbolic relative to mapping tori of finitely generated subgroups of $G$ on which $\alpha$ (taken to a suitable power, post-composed by a suitable inner automorphism) induces a polynomially growing automorphism (and even always linearly growing, in the case of a surface group) \cite{ghosh_relative_2018, dahmani_li_relative_2022}. In both situations, mapping tori of finitely generated free groups by linearly growing automorphisms play an important role in the structure of an arbitrary automorphism.

        \subsection{Objectives and results}
       
        The following conjectures illustrate the goal of this paper. The first is whether the suspension as a group (rather than as a group equipped with a normal subgroup) retains the information about the automorphism.
        
        \begin{conj}[Unipotent polynomial monofibration]\label{conj:poly-mono}
        
          Let $\Gamma = F_n \semidirect\alpha \bk t$ be a unipotent polynomial
          suspension of a free group of rank $n$. If $\Gamma = F_m\semidirect\beta \bk s$ is another unipotent polynomial fibration of $\Gamma$, for a free group of rank $m$,  then $m=n$ and $[\alpha]$ and $[\beta]^{\pm 1}$ are conjugate in $\out{F_n}$.
        \end{conj}
        
        The second is whether the suspension's finite order automorphisms are visible in the profinite completion, or specifically in a computable finite quotient. 
        
        \begin{conj}[Minkowski Property]\label{conj:mink-poly} 
        
            Let $\Gamma = F \semidirect\alpha \bk t$ be a unipotent polynomial
          suspension of a finitely generated free group. Then, there is a computable characteristic finite quotient $\Gamma \onto Q$ so that the congurence quotient $\out \Gamma \to \out Q$ has torsion free kernel.
        \end{conj}
        
        The third, more algorithmic, is a general orbit problem for the automorphism group of suspensions of unipotent linear automorphisms of free groups. We follow Bogopolski and Ventura \cite{Bogopolski_Ventura} in its denomination.  
        
        \begin{conj}[Mixed Whitehead Problem]\label{conj:whit-poly} 
            Let $\Gamma = F \semidirect\alpha \bk t$ be a unipotent polynomial
          suspension of a finitely generated free group. Then, the orbit problem for tuples of conjugacy classes of tuples of elements under the action of $\aut {\Gamma, F}$, is solvable. 
        \end{conj}

        A positive answer to  Conjecture \ref{conj:poly-mono} would relate the conjugacy 
        problem for these automorphisms 
        to the isomorphism problem for 
        suspensions. 
        We note that the conjugacy problem for unipotent polynomially growing automorphisms 
        was recently solved by Feighn and Handel \cite{feighn_conjugacy_2019}. It should be noted that one cannot expect that all isomorphisms preserve the fibre.
        
        A positive answer to Conjectures \ref{conj:mink-poly} and \ref{conj:whit-poly} would be a significant advance to the conjugacy problem for all outer automorphisms of free groups, through  the main result of \cite{dahmani_touikan_reducing_2021} and of   \cite{feighn_conjugacy_2019} (the statement of Conjecture \ref{conj:mink-poly} would have to be strengthened to cover finitely generated subgroups of the suspensions as well).

    

        We prove that all three   Conjectures are true if 
         $\alpha$ is an unipotent linearly growing automorphism of free groups. 
         
        
        \begin{thm}[Unipotent linear monofibration, see Corollary \ref{cor;ULMF}]\label{thm:UL-mono}
        Let $\phi$ and $\psi$  be two  unipotent linearly growing automorphisms  of a finitely generated free group $F$,  the semidirect products $F\rtimes_\phi \bbZ$ and $F\rtimes_\psi \bbZ$ are isomorphic if and only if  $\phi$ is conjugate to $\psi^{\pm 1}$ in $\out F$.
        
        \end{thm}

        \begin{thm}[Mixed Whitehead problem, see Theorem \ref{thm;MWP_fo}]\label{thm;MWP_intro}
        
            Let $\phi$ be a unipotent linearly growing automorphism of  a finitely generated free group $F$, and  $\Gamma= F\rtimes_\phi \bbZ$. 
            
            The orbit problem for tuples of conjugacy classes of tuples of elements of $\Gamma$ under the action of $\aut{\Gamma, F}$ is solvable.   
        \end{thm}
        
        \begin{thm}[Minkowski property, see Theorem \ref{theo;SPTS_M}]
        
            Let $\phi$ be a unipotent linearly growing automorphism of  a finitely generated free group $F$, and $H$ be a finitely generated subgroup of $F\rtimes_\phi \bbZ$. There exists a computable characteristic finite quotient $H\onto Q$ such that the induced congruence map
            $\out H \to \out Q$ has torsion free kernel.
        \end{thm}

        
        

        
       We also  establish a set of properties, as the following items, for this  class of  suspensions of free groups.   
       
       We establish the algorithmic tractability of their class of subgroups (Section \ref{sec;algo_trac}, Theorem \ref{thm;algo_trac}), which is a collection of solutions to general algorithmic problems in the class of groups (and subgroups), such as computing presentations of subgroups, deciding the conjugacy problem, and the generation problem.   In particular, we show (Propositions \ref{prop;eff_coh_ump_for_GDT}, \ref{prop;eff_coh}) that these suspensions are \emph{effectively
          coherent}, strengthening the result of Feighn and Handel \cite{feighn_mapping_1999}
        for this class of suspensions.
        
        We recover  the solution to the  fibre and orientation preserving isomorphism problem due to Cohen and Lustig  \cite{cohen_lustig_conjugacy_1999}  (Section \ref{sec;UL_IP_aut}, Corollary \ref{cor;ULMF}, Proposition \ref{prop:iso-problem-for-suspensions}), which by the monofibration theorem, provides a solution to the isomorphism problem too.
       
       A direct consequence of the results of this paper, and of the main reduction result of  \cite{dahmani_touikan_reducing_2021} is Theorem \ref{thm:CP-rel-ulg}. It is rewarding to notice that Theorem \ref{thm;MWP_intro}, thanks to the second part of the criterion of Proposition \ref{prop:conj-crit}, allows to proves the following.
       
       \begin{thm}\label{thm:cp_aut}
           Let $F$ be a finitely generated free group. Dehn's Conjugacy Problem in $\Aut F$ among unipotent linearly growing automorphisms is algorithmically solvable.
       \end{thm}
       
       \begin{proof} Let $\alpha, \beta$ be two such automorphisms. We can first decide whether they are  conjugate in $\out{F}$, by  \cite{cohen_lustig_conjugacy_1999}  (or the results mentioned above). If they are not, we are done, and if they are, we can compute an explicit fibre-and-orientation preserving isomorphism $f:F\semidirect\alpha \bk t \to F\semidirect\beta \bk s$. The criterion of Proposition \ref{prop:conj-crit} reduces the problem to determining whether there is a fibre and orientation preserving automorphism of $F\semidirect \beta \bk s$ that sends $f(t)$ to $s$. By  Theorem \ref{thm;MWP_intro}, we can decide whether there is such an automorphism that sends conjugacy class $[f(t)]$ to the conjugacy class $[s]$. Since $F$ is normal in the group, after composing by the appropriate inner automorphism, we are done. \end{proof}

        \subsection{Overview of the arguments and methods}

        We now discuss methods and arguments. One of the main ingredients of our approach is JSJ theory, which was first developed by Rips and Sela \cite{rips_cyclic_1997}, then through the works \cite{fujiwara_papasoglu,dunwoody_sageev} and is presented in an accomplished form in the monograph  \cite{guirardel_jsj_2017}.
        
        The first step is to consider  the train-track description of elements of unipotent polynomially growing outer automorphisms. The first construction was  
        given in \cite{bestvina_tits_2005}, and, was
        observed in \cite{macura_detour_2002} (see also \cite{hagen_cubulating_2016}) to give  a cyclic
        hierarchical decomposition of $F\semidirect\alpha\bk{t}$. 
        Using a more sophisticated type of train tracks, namely the completely split train tracks of Feighn and Handel \cite{feighn_recognition_2011}, we prove the folklore theorem that unipotent linearly growing automorphisms coincide with generalized Dehn twists and give a description of the induced graph of groups decomposition of the underlying free group as relative JSJ decompositions. 
        
        This decomposition of the free group $F$ gives rise to a decomposition of corresponding suspension $F\semidirect\alpha\bk{t}$ and we are able to show that this composition is canonical. This thus produces a preferred Macura decomposition of $F\semidirect\alpha\bk{t}$. One novelty of our work is that these suspensions are not so-called CSA groups nor do they resemble any other groups to which JSJ theory has effectively applied, nonetheless  the setting in \cite{guirardel_jsj_2017} is sufficiently rich to yield a canonical splitting in this new context.

        
        
        An application of Bass-Serre theory allows us to go further and produce canonical splittings for one-ended finitely generated  subgroups of $F\semidirect\alpha\bk{t}$ as well, where vertex groups are either free groups or direct products of free groups with $\bbZ$. Of course, we keep track of the computability of these canonical splittings.
        
        As a consequence of the canonicity of these splittings, we are able to describe the automorphisms groups of suspensions and of subgroups of suspensions in terms of automorphisms of graphs of groups. We note that recent works of Andrew and Martino involving similar methods give related results \cite{andrew_free-by-cyclic_2022,andrew_centralisers_2022}. This understanding of  automorphism groups is a fundamental component of our proofs of the main results of this paper. In working out preliminaries for solving the mixed Whitehead problem, we immediately obtain the monofibration of Theorem \ref{thm:UL-mono} and an alternative solution to the conjugacy problem for unipotent linearly growing outer automorphisms. The computability of these canonical splittings is also a crucial component in our computability results.
        
        The next task is to prove that the class of suspensions of free groups by unipotent linearly growing automorphisms is an algorithmically tractable class of groups. In particular, using the properties of our splittings of being benign, in the sense of Kapovich, Weidmann and Miasnikov \cite{kapovich_foldings_2005}, we prove that the groups are effectively coherent. 
        
        The Mixed Whitehead problem for suspensions $F\semidirect\alpha\bk{t}$ is also approached  through the obtained canonical splitting. Given an element, one considers a normal form in the decomposition. It is not unique, and one retains the sequence of double cosets of edge groups in the encountered vertex groups. This is the information shared by the different normal forms. Unfortunately, the orbit problem for collections of double cosets in vertex groups is delicate. We associate a free subgroup (a linkage) to a double coset, generated by conjugates of the two involved subgroups, and keep track of the configuration of the linkages appearing in each vertex group.  The orbit problem for subgroups, in the vertex groups,  is solved by a result of Gersten \cite{gersten_whiteheads_1984}, and  will be used here as well, but to have all the requisite information preserved, a dual construction must be carried simultaneously, which makes the argument more delicate. As mentioned earlier, the argument relies on Gersten's solution of the orbit problem for the action of automorphisms of free groups on conjugacy classes of subgroups. It has to be modified though, in order to treat tuples of conjugacy classes of tuples of subgroups, which we propose as an application of Gersten's result.

        Finally, we study the Minkowski property, which asks, given a subgroup $H$ of $F\semidirect\alpha\bk{t}$,  for a finite quotient of $H$  in which all finite order outer automorphism of $H$  survive. We first produce a characteristic clean covering of the canonical graph of groups decomposing $H$ that we obtained earlier.  It is done through an argument borrowed from \cite{cotton-barratt_conjugacy_2012},  
        involving coverings of graphs of spaces, Marshall Hall's theorem, and the omnipotence theorem of Wise \cite{wise_subgroup_2000}. Then, once in clean position, we consider $V$ the associated finite index subgroup of $H$, and we may produce a Dehn filling quotient of the graph of groups of $V$, that quotient vertex groups to free groups, and thus produces a virtually free group. We take care to keep the separation of certain witnesses of finite order automorphisms of $H$ (these witnesses are built from elements of $H$ moved by the automorphisms). Then, by a classical conjugacy separability result, we quotient the virtually free group $V$ to a finite group $Q_0$ in which all conjugacy classes of finite order elements are kept distinct. Since it concerns all elements in vertex groups of $V$, the witnesses of the finite order automorphisms of $H$ are kept distinct in the finite quotient $Q_0$. The care with which the witnesses were defined allows us finally to show that the finite order outer automorphisms of $H$ descend to automorphisms of $Q$, a finite group containing $Q_0$, in a non-trivial way.

\vskip .5cm

{\it Acknowledgements.} F.D. acknowledges support of the ANR-22-CE40-0004 GoFR, the Labex Carmin, ANR-10-LABX-59-01,   and the  CNRS
IRL 3457
(CRM, Montr\'eal).  N.T. acknowledges the support of the Natural Sciences and Engineering Research Council of Canada (NSERC). Both authors are grateful for  the support and facilities of the program `Recherche en Bin\^ome' in CIRM, Luminy,   no. 2943.  
It is a pleasure to thank N. Macura, A. Martino, S. Francaviglia, S. Hugues, M. Kudlinska,  for discussions related to this work. We thank the anonymous referees for pointing out an error in the initial version.

    \section{Background on trees and graphs of groups}
     
        \subsection{Bass-Serre theory}\label{sec;vocab}
     
            While we assume the reader familiar with Bass-Serre theory, we still must set notation and terminology conventions.
            
             A \emph{graph} $X= (\verts X, \edges X, i, \tau, -)$ consists of a set of vertices $\verts X$, a set of oriented (or directed) edges $\edges X$, two maps $i, \tau :\edges X \to \verts X $, and a fixed-point free involution $-: \edges X\to \edges X$ satisfying ${i(\bar e)} = \tau(e)$.

             The cell 1-complex whose vertex set is $\verts{X}$, and one-cells are the pairs $\{e,\bar e\}$, $e\in\edges{X}$,  is the \emph{geometric realization of X.}

             A \emph{graph of groups} $\bbX$ consists of an underlying graph $X$, a \emph{vertex group}
             $\bbX_v$ for each vertex $v \in \verts X$, an \emph{edge group} $\bbX_e$   for each edge $e \in \edges X$, with a pair of injective homomorphisms   $i_e:\bbX_e \into \bbX_{i(e)}, \tau_e:\bbX_e \into \bbX_{\tau(e)}$,      such that for all edge $e$,  
             $\bbX_e = \bbX_{\bar e}$, and   $i_e = \tau_{\bar e}$.

             Given a graph of groups $\bbX$, one may consider the free product of all the vertex groups and the free group over the unoriented edges:     
                \[ \mathbb{F}(\bbX) =    
             		\left( \underset{ v \in \verts X}\Asterisk \bbX_v \right) * \left(F_{\edges{X}}/_{\langle\!\langle \{  e\bar e, \, e\in
             		\edges X \}  \rangle\!\rangle} \right).
             	\]     
             The \emph{Bass group} of
             $\bbX$ is  given as the following quotient, which identifies attached edge groups     
                \begin{equation}\label{eqn:bass}
         	        \Bass(\bbX) = \mathbb{F}(\bbX)/_{\langle \!\langle    \{ i_e(g)e \tau_e(g)^{-1} e^{-1},\;   e\in \edges X, \, g \in \bbX_e\}    \rangle \!\rangle}.
                 \end{equation}
             
             In the two previous groups, the normal subgroups by which one quotients thus provide the (more readable)  relations in the Bass group:
             \[ \forall e\in \edges X,  \forall  g \in \bbX_e: \qquad  e^{-1}= \bar e, \qquad  i_e(g) e   = e \tau_e(g).    \]
             
             Given some $b \in \verts X$, the \emph{fundamental group of $\bbX$ based at $b$}, $\pi_1(\bbX,b)$, 
             is the subgroup of the Bass group generated by elements, called
             \emph{$\bbX$-loops based at $b$}, of the
             form
                \[ w=a_1e_1a_2e_2\cdots e_n a_{n+1}
                \] 
            which satisfy the following criteria:  $a_1, a_{n+1} \in \bbX_b$,
            and for all $j\in\{1,\ldots,n\}$,  $a_j \in \bbX_{i(e_j)}, a_{j+1} \in \bbX_{\tau(e_j)}$. In particular, the
            concatenation  $ \overline{w}_X = e_1\cdots e_n$ 
            is a closed loop based at $b$ in $X$. 

            Given a spanning tree $\tau \subset X$ the {\it fundamental group  of $\bbX$ at $\tau$}  is   
            $ \pi_1(\bbX,\tau) = \Bass(\bbX)/_{\langle \! \langle \edges \tau\rangle \! \rangle }$.

            The following is an important foundation of Bass-Serre theory.
        
            \begin{thm}[{\cite{serre_trees_2003}}]
         	    Let $\tau \subset X$ be a spanning tree,   and $F:\Bass(\bbX) \to \pi_1 (\bbX,\tau)$ the quotient map. 
         	
             	Then, for all vertex  $u \in \verts
             	X$,  the restriction  $F|_{\pi_1(\bbX,u)}: \pi_1(\bbX,u) \to \pi_1 (\bbX,\tau)$    is an isomorphism.  
            \end{thm}
         
            We  call an isomorphism $G \simeq \pi_1(\bbX, v_0)$ or  $G \simeq \pi_1(\bbX, \tau)$   a \emph{splitting 
            of $G$	as a graph of groups.}

            Let $G$ be a group.  A $G$-tree $T$ is a tree with an action of $G$
            without edge inversion. Given $v\in \verts T$ and $e \in \edges T$,
            we denote the stabilizers of $v$ and $e$ by $G_v$ and $G_e$
            respectively.  
         
            We will say $T$ is a \emph{minimal $G$-tree} if $T$ has no
            $G$-invariant subtrees.

            The Bass-Serre duality \cite{serre_trees_2003} famously provides a correspondence between graphs of groups $\bbX$ and their dual trees, that are $\pi_1(\bbX, \tau)$-trees $T$, for which  $X = \pi_1(\bbX, \tau)\backslash T$, and for which isotropy groups  are the conjugates of the corresponding vertex or edge groups  of $\bbX$. 

        \subsection{$G$-tree vocabulary and JSJ theories}\label{sec:G-trees}
         If $\calA$  and $\calH$  are two classes of subgroups of $G$, we say that a $G$-tree $T$ is \emph{an $(\calA, \calH)$-tree} if all edge stabilizers are in $\calA$, and all elements of $\calH$ are elliptic. If $\calH$ is empty, we say $T$ is a  $\calA$-tree.  
        
        The concept of JSJ decomposition is one of the most important tools in this situation. The theory, initiated by Rips and Sela, is comprehensively presented in \cite{guirardel_jsj_2017}.   We will now follow \cite{guirardel_jsj_2017} to recall certain notions. Certain terms will be streamlined to the contexts in which they will be used. \emph{Morphisms} of $G-$graphs (hence of trees) are simplicial, while \emph{maps} need only be continuous.
        
        Let $G,H$ be groups,   $\varphi:G \to H$ be a group homomorphism,  $S$ be a $G$-tree and $T$ be an $H$-tree. A map  $a:S \to T$,  is  \emph{$\varphi$-equivariant} if  for all $x \in S$ and $g \in G$, one has  $a(g\cdot x) = \varphi(g)\cdot a(x)$. We will sometimes leave the equivariance implicit. Also, $\phi$ will often be the identity map on $G$.  Equivariant surjective maps between trees (which may fail to be simplicial) are called \emph{dominations}. If $S,T$  are both $G$-trees that dominate each other, then we say they are in the same \emph{deformation space.} Being in the same deformation space amounts to saying both threes have the same vertex stabilizers (though the trees themselves may not be isomorphic). For a fixed group $G$ we say a $(\calA,\calH)$-tree is \emph{universally elliptic} if its edge stabilizers act elliptically on every other $(\calA,\calH)$-tree. An \emph{$(\calA,\calH)$-JSJ tree} $T$ is a universally elliptic tree that dominates every other universally elliptic $(\calA,\calH)$-tree. Let $k \geq 0$ be an integer. An action is \emph{$k$-acylindrical} if, for every pair of vertices  that are at distance at least $k+1$,   the stabilizer of the pair     is trivial. Thus  
        an action on a tree with trivial edge stabilizers is
        thus $0$-acylindrical.

        We will typically use $\calA$ to denote the class of of allowable edge groups of our splittings and denote by $\calA_\infty$ the infinite subgroups in $\calA$. In this paper, $\calA$ will either be the class of cyclic groups or the class of free abelian groups of rank at most 2. Following \cite[Definition 7.1]{guirardel_jsj_2017}, let $\calA$ be class of subgroup of $G$ such that $\calA_\infty$ is \emph{sandwich closed}, which is to say if $H,K \in \calA_\infty$ and $H\leq I \leq K$ then $I \in \calA_\infty$. Then an equivalence relation $\sim$ on $\calA_\infty$ is said to be \emph{admissible} (relative to $\calH$) if it satisfies the following axioms:
        \begin{enumerate}
            \item If $A\sim B$ and $g\in G$, then $gAg^{-1}\sim gBg^{-1}$.
            \item If $A \leq B$, then $A \sim B$.
            \item Let $T$ be a $G$-tree in with edge stabilizers in $\calA_\infty$ (in which subgroups in $\calH$ are elliptic). If $A\sim B$ and $A,B$ fix $a,b \in T$ respectively then each edge $e$ in the path from $a$ to $b$ we have $G_e \sim A \sim B.$
        \end{enumerate}
        The equivalence class of $A\in \calA_\infty$ is denoted $[A]$ and the stabilizer of $[A]$ under the action of $G$ on $\calA_\infty/\sim$ induced by conjugation is denoted $G_{[A]}$. Still following \cite[\S 7.1]{guirardel_jsj_2017}, given an $(\calA,\calH)$ tree $T$, we define two edges $e,f$ to be equivalent if $G_e \sim G_f$ define a \emph{cylinder} to be the union of edges in an equivalence class. Cylinders are subtrees and distinct cylinders intersect at at most a point. We define the \emph{tree of cylinders} $T_c$ of $T$ as follows: it is a bipartite tree whose set of black vertices is the set of cylinders and whose white vertices is the set of vertices of $T$ that lie in at least two distinct cylinders. A white vertex $w$ is adjacent to a black vertex $b$ in $T_c$ if, in $T$, $w$ lies in the cylinder $b$.\footnote{\emph{Collapsed} trees of cylinders are also important in \cite{guirardel_jsj_2017} but will not be needed in this paper.}  The existence of JSJ trees and the properties of trees of cylinders depend in particular on the nature of $\calA$ and $\calH$.  

        This leads us to an informal concept that we call a \emph{JSJ theory}; which basically corresponds to a class of groups, a class $\calA$ of allowable edge groups, some peripheral structure (or a class of peripheral groups) $\calH$, and an admissible equivalence relation $\sim$ on $\calA_\infty$. Examples of JSJ theories include: the class of one ended hyperbolic groups where $\calA$ is the class of virtually cyclic groups and $A\sim B$ if and only if $A$ and $B$ are commensurable, or the class of finitely generated CSA groups (maximal abelian subgroups are malnormal) where $\calA$ is the class of abelian groups and $A\sim B$ if and only if $\bk{A,B}$ is abelian. We make this concept explicit since we will encounter different JSJ theories throughout this paper. 
        
        We'll say an $(\calA,\calH)$-tree $T$ is \emph{canonical} for $G$ if for any \emph{$(\calA,\calH)$-automorphism of $G$}, i.e. an an automorphism that preserves the classes $(\calA,\calH)$ extends naturally to a $G$-equivariant automorphism of $T$. This implies, for example, $(\calA,\calH)$-automorphisms map vertex groups to vertex groups. The existence of non-trivial canonical trees is a feature of a JSJ-theory. 

        In most ``classical'' JSJ theories, a canonical $(\calA,\calH)$-tree arises from taking the tree of cylinders of an $(\calA,\calH)$-JSJ tree, see \cite[Corollary 9.1]{guirardel_jsj_2017}. It is worth noting, however, that some JSJ theories do not admit canonical trees. One such negative example would be the case of a free group $F$ with $\calA$ being the set of infinite cyclic groups. On the other hand, if we take a class of subgroups $\calH$ such that $F$ is one-ended relative to $\calH$, then $F$ will admit a canonical $(\calA,\calH)$-tree $T_c$ that is the tree of cylinders of an $(\calA,\calH)$-JSJ tree $T$, see \cite[Theorem 9.5]{guirardel_jsj_2017}.

        \subsection{Automorphisms of graphs of groups}\label{sec;autom_gog}
 
        By the duality theorem, the notion of an automorphism of $G$-trees allows us to define automorphisms of a graph of groups. We will now give a more  combinatorial account of this concept. Although the treatment here may appear excessively formal, the notation introduced will be used throughout the paper.
 
        An automorphism $\Phi$ of a graph of groups $\bbX$ is a tuple
        consisting of an automorphism $\phi_X:X\to X$ of the underlying
         graph $X$, an isomorphism $\phi_v : \bbX_v\to \bbX_{\phi_X(v)}$ for each
         $v \in \verts X$, an isomorphism $\phi_e : \bbX_e\to \bbX_{\phi_X(e)}$ for
         each $e \in \edges X$, with $\phi_{\bar e} = \phi_e$, and elements
         $\gamma_e\in \bbX_{\phi_X(\tau(e))}$ for each $e \in \edges X$, that satisfy
         the Bass commutative diagram equations:
            \begin{equation}\label{eq:Bass_Diagram}  
             	\begin{tikzpicture}
             		\node[
             		regular polygon,
             		regular polygon sides=5,
             		minimum width=50mm,
             		yscale=.65
             		] (PG) {}
             		(PG.corner 1) node (PG1) {$\bbX_{\phi_X(e)}$} 
             		(PG.corner 2) node (PG2) {$\bbX_{e}$}
             		(PG.corner 3) node (PG3) {$\bbX_{\tau(e)}$}
             		(PG.corner 4) node (PG4) {$\bbX_{\phi_X(\tau(e))}$}
             		(PG.corner 5) node (PG5) {$\bbX_{\phi_X(\tau(e))}$}
             		;
             		\draw[->] (PG2) -- (PG1) node [midway,sloped,above] {$\phi_{e} $};
             		\draw[->] (PG1) -- (PG5)  node [midway,sloped,above] {$\tau_{\phi_X(e)} $} ;
             		\draw[->] (PG5) -- (PG4)  node [midway,sloped,below] {$\ad{\gamma_e} $} ;
             		\draw[->] (PG3) -- (PG4)  node [midway,above] {$ \phi_{\tau(e)} $};
             		\draw[->] (PG2) -- (PG3)   node  [midway,sloped,below] {$\tau_e$};
             	\end{tikzpicture}
            \end{equation}

        In other words: $  \phi_{\tau(e)} \circ \tau_e =
        \ad{\gamma_{e}}\circ \tau_{\phi_X(e)} \circ \phi_{ e}   $. An automorphism of the graph of groups $\bbX$ extends naturally to an automorphism of the Bass group in a unique way by sending the generator $e$ as follows 
            \[ e \mapsto \gamma_{\bar e}^{-1} \phi_X(e) \gamma_e. \]
         
        While vertices or spanning trees are not necessarily preserved,  an automorphism of a graph of groups induces an isomorphism between
        $\pi_1(\bbX, v)$ and $\pi_1(\bbX, \phi_X(v))$.  See 
        \cite[\S 2.3]{bass_covering_1993}, \cite[Lemmas 2.20-2.22]{dahmani_isomorphism_2010}.  The
        composition rule is as follows
        \begin{eqnarray*}
 	        && (\phi_X, (\phi_v)_v, (\phi_e)_e, (\gamma_e)_e) \circ (\phi_X',(\phi'_v)_v, (\phi'_e)_e, (\gamma'_e)_e)\\
 	        && = \left( \phi_X\circ \phi_X', \, (\phi_{\phi_X'(v)} \circ \phi'_v)_v, \, (\phi_{\phi_X'(e)} \circ \phi'_e )_e, \, \gamma_{\phi_X'(e)} \phi_{\phi_X'(\tau(e))}(\gamma'_e) \right).
        \end{eqnarray*}
  
        The group of all automorphisms of the form
            \[(\phi_X, (\phi_v)_v, (\phi_e)_e, (\gamma_e)_e)\] 
        satisfying \eqref{eq:Bass_Diagram} is denoted $\delta \aut\bbX$, and maps naturally to a subgroup of $\out{ \pi_1(\bbX, v)}$.

         The subgroup $\delta_0\aut\bbX \leq \delta\aut\bbX$ of \emph{pure automorphisms} is the kernel of the natural map $\delta \aut\bbX \to \aut X$ and is clearly of finite index in $\delta \aut\bbX$. Observe that $\delta_0\aut\bbX$ maps naturally to (a subgroup of) $\aut{\pi_1(\bbX, v)}$ for an arbitrary vertex $v$, and to $\aut{\pi_1(\bbX, \tau)}$, for arbitrary spanning tree $\tau$.

        It is worth observing that, for elements of $\delta_0\aut{\bbX}$, the composition rule simplifies to 
            \begin{equation}\label{eq:delta_0-composition} \begin{array}{ccl} \big( Id, (\phi_v)_v, (\phi_e)_e, (\gamma_e)_e \big) & \circ & \big( Id,
             (\phi'_v)_v, (\phi'_e)_e, (\gamma'_e)_e \big) 
              \\ &=& 
             \left( Id,
             (\phi_v \circ \phi'_v)_v, (\phi_e \circ \phi'_e )_e, ( \gamma_{(e)}
            \phi_{\tau(e)}(\gamma'_e))_e \right) \end{array}
            \end{equation} 
        and the Bass diagram also simplifies to:

            \begin{equation}\label{eq:Bass_Diagram_delta_0}  
 	        \begin{tikzpicture}
         		\node[
         		regular polygon,
         		regular polygon sides=5,
         		minimum width=30mm,
         		] (PG) {}
         		(PG.corner 1) node (PG1) {$\bbX_{e}$} 
         		(PG.corner 2) node (PG2) {$\bbX_{e}$}
         		(PG.corner 3) node (PG3) {$\bbX_{\tau(e)}$}
         		(PG.corner 4) node (PG4) {$\bbX_{\tau(e)}$}
         		(PG.corner 5) node (PG5) {$\bbX_{\tau(e)}$}
         		;
         		\draw[->] (PG2) -- (PG1) node [midway,sloped,above] {$\phi_{e} $};
         		\draw[->] (PG1) -- (PG5)  node [midway,sloped,above] {$\tau_{e} $} ;
         		\draw[->] (PG5) -- (PG4)  node [midway,sloped,below] {$\ad{\gamma_e} $} ;
         		\draw[->] (PG3) -- (PG4)  node [midway,below] {$ \phi_{\tau(e)} $};
         		\draw[->] (PG2) -- (PG3) node   [midway,sloped,below] {$\tau_e$};
 	        \end{tikzpicture}
            \end{equation}
 
        Thus, $\delta_0\aut\bbX$  naturally maps to $\left(\prod_{v\in \verts{\bbX}}  \aut{\bbX_{v}} \right) 
        \times \left( \prod_{e\in \edges{\bbX}}  \aut{\bbX_{e}} \right) $. 
        
        We define the {\it small modular group of $\bbX$} to be the kernel.
        
        Let us also define $\delta_1\aut{\bbX} \leq \delta_0\aut\bbX$ to be the kernel of the map to the second factor  $ \left( \prod_{e\in \edges{\bbX}}  \aut{\bbX_{e}} \right) $.
        
        
        
        
        Note a slight variation with the denomination of small modular group in \cite[\S 1.2]{D_TAMS}:  $\aut{\bbX_{v}}$ was there replaced by $\out{\bbX_{v}}$. The proposed convention seems more appropriate. 
        The Bass diagram and the specification of the automorphism on the Bass group give the following. 
 
        \begin{prop}\label{prop;formula_for_GDT} 
            The elements of the small modular group are the tuples 
            \[( Id_X, (Id_v)_v,  (Id_e)_e, (\gamma_e)_e ), \; \hbox{such that } \;  \forall  e,\, \gamma_e \in Z_{\bbX_{\tau(e)}} (i_e(\bbX_e)) \] 

            The composition rule is 
            \[ \begin{array}{ccl} ( Id_X, (Id_v)_v,  (Id_e)_e, (\gamma_e)_e ) & \circ & ( Id_X, (Id_v)_v,  (Id_e)_e, (\gamma'_e)_e ) \\ &= & ( Id_X, (Id_v)_v,  (Id_e)_e, (\gamma_e\gamma'_e)_e ).\end{array}  \]

            The image by $( Id_X, (Id_v)_v,  (Id_e)_e, (\gamma_e)_e )$ of a loop $e_1e_2\dots e_n $  of $\pi_1(\bbX, v_0)$  is  
            \[\eta = \gamma_{\bar e_1}^{-1} e_1 \gamma_{e_1} \gamma_{\bar e_2}^{-1} e_2\gamma_{e_2} \dots \gamma_{\bar e_n}^{-1} e_n \gamma_{e_n}.  \]

            The image by $( Id_X, (Id_v)_v,  (Id_e)_e, (\gamma_e)_e )$ of an element $e_1e_2\dots e_m \gamma_v \bar e_m \dots \bar e_2 \bar e_1$  of $\pi_1(\bbX, v_0)$  is $\eta   \gamma_v  \eta^{-1}$ for \[\eta = \gamma_{\bar e_1}^{-1} e_1 \gamma_{e_1} \gamma_{\bar e_2}^{-1} e_2\gamma_{e_2} \dots \gamma_{\bar e_m}^{-1} e_m \gamma_{e_m}.  \]

        \end{prop}

        We will call the elements of the small modular group, the generalized Dehn twists of $\bbX$.

        \subsection{Substitutions, generalized Dehn twists and linear growth} 
 
         We present now a formal substitution rule that will describe generalized Dehn twists. An orientation of a graph $X$ is a choice of a subset $\edges{X}^+$ of $\edges X$ that intersects each pair $\{e, \bar e\}$ on one element exactly. If the graph is bipartite we will choose the orientation such that  $e\in \edges{X}^+$ if and only if $i(e)$ is a black vertex.  Recall (see Section \ref{sec;vocab}) that if $\bbX$ is a graph of groups with underlying graph $X$, the notation $\bbF(\bbX)$ denotes the free product of all vertex groups and of the free group with basis $\edges{X}^+$.
         
         \begin{defn}[Full substitutions]
            Let $\bbX$ be a bipartite graph of groups in which black vertex groups are abelian and coincide with incident edge groups.
         
            An homomorphism $\Psi: \bbF (\bbX) \to \bbF(\bbX)$
            is  a $2$-sided  $\edges{X}$-substitution if it is the identity on  $(\bigcup_{v\in \verts X} \bbX_v)$ and sends each $e\in  \edges X$ to an element $\gamma_{\bar e}^{-1} e \gamma_e$ with $\gamma_e \in \tau_{e}(\bbX_{\tau(e)})$ and  $\gamma_{\bar e} \in \tau_{\bar e}(\bbX_{\tau(\bar e)})$).  
 
            If $\gamma_e=1$ for all $e\notin \edges{X}^+$, and for all edges whose  initial vertex is white, then we say that $\Psi$ is a \define{one-sided substitution}, or simply a $\edges{\bbX}$-substitution.

            We say that a one-sided substitution is \define{full} if,  
            for any length 2 path $e_i^{-1}e_j$ going through a black vertex $b$,  
            applying the substitution, we have in $\Bass(\bbX)$: \[
                e_i^{-1}e_j \neq \gamma_{e_i}^{-1}e_i^{-1}e_j \gamma_{e_j}. \]
                
                Observe that $\gamma_{e_i}^{-1}e_i^{-1}e_j \gamma_{e_j} =  e_i^{-1}e_j (c'\gamma_{e_j}) $  for some $c' \in \tau_{e_j}(\bbX_{e_j})$. We have then required that $c'\gamma_{e_j} \neq 1$. 
        \end{defn}
        
        We note that the requirement that a substitution is full ensures that no $\bbX$-loops with reduced form $g_1e_i^{-1}e_jg_2e_j^{-1}e_i$ will be left invariant. Thus, it is a weak form of the requirement that an $\edges \bbX$-substitution doesn't fix any hyperbolic conjugacy classes.

        Observing that  $\bar e$ is sent to the inverse of the image of $e$, and  $\gamma_e$ is in the centralizer in $\bbX_{\tau(e)}$ of the image of $\bbX_e$, which  contains $\tau_e(g)$, by considering the presentations of the groups  $\mathbb{F}(\bbX)$,  $\Bass(\bbX)$, and the definition of the subgroup  $\pi_1 (\bbX,v_0)$, we consign the following.
        
        \begin{lem} \label{lem;subst_to_gdt}
            Let $v_0$ be a vertex in $X$. Any 2-sided $\edges{\bbX}$-substitution extends   as  an automorphism of the Bass group, that preserves the subgroups  $\pi_1 (\bbX,v_0)$, and that is a generalized Dehn twist $(Id_X, (Id_v), (Id_e), (\gamma_e))$.
        \end{lem}

        The formula of Proposition \ref{prop;formula_for_GDT} indeed allows us to recognize the generalized Dehn twist. For terminology purposes, we will talk about a substitution Dehn twist for a Dehn twist induced by a one-sided substitution.

        Recall that an outer automorphism $\Phi$ of a finitely generated  group $G$ (equipped with a proper word metric) is said to have (at most) linear growth if,  for all element $g\in G$,  there exists $a,b \in \bbN$ such that for all $n\in \mathbb{Z}$  the shortest word length in the conjugacy class $\Phi^n ([g])$ is at most $(a|n| +b)\times \|g\|$. 
 
        \begin{prop} 
            If $\bbX$ is a graph-of-groups, and  $\Phi= (Id_X, (Id_v),  (Id_e), (\gamma_e) )$ is a generalized Dehn twist of $\bbX$, then, given $v_0\in \verts{\bbX}$,  the outer automorphism of $\pi_1(\bbX,v_0)$ defined by $\Phi$ has at most linear growth.
        \end{prop}
 
        \begin{proof} 
            Consider a generating set adapted to the free product $\bbF(\bbX)$, and the associated word length $|\cdot |$. Let $\gamma$ be an element of $\pi_1(\bbX,v_0)$, and $\ell_E(\gamma)$ the number of letters in $\edges X$ of a shortest word defining $\gamma$.  Choosing $a$ to be the maximal word length of the elements $\gamma_e$, $e\in \edges X$, one controls the word length:  $|\Phi^n(\gamma)| \leq 2an\times \ell_E(\gamma) + |\gamma|$, which is at most $(2an+1)|\gamma|$.
        \end{proof}

        Recall that an (outer)-automorphism of a free group $F$ is unipotent if it induces a unipotent automorphism of the abelianization $F^{ab}$ (that is, given a basis, an element $A$ of $GL_n(\bbZ)$  such that $ A-I_n$ is nilpotent).
 
        \begin{prop} 
            If $\bbX$ is a graph-of-groups whose fundamental group is free, and with trivial edge groups,  and  $\Phi= (Id_X, (Id_v)_v,  (Id_e)_e, (\gamma_e)_e )$ is a generalized Dehn twist of $\bbX$, then, given $v_0\in \verts{\bbX}$,  the outer automorphism of $\pi_1(\bbX,v_0)$ defined by $\Phi$ is unipotent.
        \end{prop}
 
        \begin{proof} 
            There exists a basis of the free group  $\pi_1(\bbX,v_0)$,  adapted to the free decomposition $\bbX$. after possible  reordering, it is  of the form $g_1, \dots, g_r, e_1, \dots, e_k$, with $g_i$ in vertex groups, and $e_i$ some non-separating edges of $X$. This basis gives a basis of the abelianization $\bbZ^{k+r}$, and the matrix of the induced automorphism from $\Phi$ in this basis is upper triangular with $1$s on the diagonal, hence unipotent. 
        \end{proof}
        
        Let $\Phi \in \out F$. A subgroup $H\leq F$ is said to be \define{$\Phi$-non-growing} for any $ \phi \in \Phi \subset \aut F$ there is some $f_H \in F$ such that $(\ad{f_F}\circ \phi)|_{H} = Id_H$.

        Our next result, Theorem \ref{thm;unipotent-linear-gdt}, establishes (with Lemma \ref{lem;subst_to_gdt}) that a unipotent automorphism of free group of linear growth has a description as a generalized Dehn twists on some graph of groups.

        \begin{thm}\label{thm;unipotent-linear-gdt}
            Let $F$ be a finitely generated free group, and let $\Phi$ be a unipotent outer automorphism with linear growth. Let $\calE$ denote the class of $\Phi$-non-growing subgroups of $F$. Then $F$ is one-ended relative to $\calE$. 
            
            Let $\calA$ denote the class of cyclic groups and let $T_c$ be the tree of cylinders of an $(\calA,\calE)$-JSJ tree $T$. The non-abelian vertex stabilizers of $T_c$ are precisely the maximal subgroups in $\calE$ and if $\bbX_c$ is the cyclic splitting of $F$ dual to $T_c$ then $\Phi$ induces a full $\edges{\bbX_c}$-substitution. Furthermore, all edge group monomorphisms to cylinder groups are surjective.
        \end{thm}

        Before going to the proof, some comment is perhaps appropriate. This statement is somewhat folklore (see \cite[Remark 2.4.7]{andrew_free-by-cyclic_2022} and the statement of \cite[Theorem 2.4.6]{andrew_free-by-cyclic_2022}). Yet, our proof is quite long, and it is a rare place in our work where we have to use advanced train tracks, such as CTs, and not only irreducible train tracks. This deserves an explanation. In order to ``see'' the generalized Dehn twist in a given unipotent linear growth automorphism $\Phi$, one needs to find the correct tree (or splitting) on which the twist occurs. JSJ theory is useful for this: the correct tree will be a canonical tree relative to the collection $\calE$ of the  subgroups  of $F$ on which $\Phi$ restricts to an inner automorphism. This canonical tree exists  because the free group $F$  is one-ended relative to $\calE$. Using this canonical tree, it then becomes easy to characterize $\Phi$ as a generalized Dehn twist. So the key claim is that the free group $F$ is  one-ended relative to the said collection $\calE$, as stated. We see no straightforward proof of this fact.

        Our concern is when $\Phi$ preserves the conjugacy classes of two free factors $A, B$, giving $F=A*B$, and  
        induces non-trivial generalized Dehn twists on  $A$ and on $B$ 
        with a conjugation of $B$ somehow overlapping with one of the twists in $A$. The case is similar to an example  of unipotent linear growth automorphisms with ``seemingly overlapping'' twists, that was communicated to us by N. Macura.   This difficulty explains our detour through CT maps. Here is Macura's example.
        
        Let $F_3=\langle a,b,c\rangle$ and let $\alpha: a\mapsto a, b\mapsto ba, c\mapsto cbab^{-1}a^{-2}$. The automorphism $\alpha$ has a structure that is typical for polynomial growth: each generator is sent to itself times a product of generators of lower strata. Since $b$ itself grows linearly, we know that $c$ grows at most quadratically. By \cite[Lemma 2.16]{macura_detour_2002} such estimates are sharp for quadratic growth and above, however here the automorphism $\alpha$ is still only linearly  growing due to cancellations in $bab^{-1}$ after applying $\alpha$. This is an example of an indecomposable Nielsen path. The construction in the proof below adds cylinders to a graph that enables us to homotope indecomposable Nielsen paths to paths consisting of non-growing edges.


 
 
 
        We may now go in the proof of the theorem. We will first recall and use facts about train tracks, then construct from them a splitting of $F$ on which the automorphism is described as a substitution, and then check that this is a JSJ tree, and that, on the  canonical tree in its deformation space, the automorphism is still described as a full substitution.
    
    \begin{proof} 
            We will use the completely split train tracks theorem of Feighn and Handel \cite[Theorem 2.19]{feighn_recognition_2011} (see the exposition of the case of unipotent polynomial growth in \cite[Section 3.5]{feighn_conjugacy_2019}). 
            
            There is a graph $Y$ with an isomorphism from $F$ to $\pi_1(Y, y_0)$, a representative $f=Y\to Y$ of $\Phi$, a $f$-invariant  filtration $\emptyset= Y_{-1} \subset Y_0 \subset Y_1 \subset \dots \subset Y_r\subset Y_{r+1} = Y$    such that for all $i\geq 0$,  $Y_i\setminus Y_{i-1}$ is a single edge $\{\epsilon_i, \bar \epsilon_i\}$, with a preferred orientation $\epsilon_i$, and for which there exists a reduced immersed loop (circuit) $\ell_{\epsilon_i}$ in $Y_{i-1}$ such that    $f(\epsilon_i) = \epsilon_i \ell_{\epsilon_i}$.
            

            The loop $\ell_{\epsilon_0}$ is necessarily empty, so $Y_0$ is a fixed edge for $f$.  
                We may alter this stratification as  $Y'_0 \subset Y'_1 \subset \dots \subset Y'_s\subset Y'_{s+1} = Y$, such that $Y'_0$ is the set of all fixed edges by $f$, and for all $i\geq 1$,  $Y'_i\setminus Y'_{i-1}$ is a single edge $\{e_i, \bar e_i\}$, with a preferred orientation $e_i$, for which  $f(e_i) = e_i \ell_{e_i}$ for a reduced loop $\ell_{e_i}$ in $Y'_{i-1}$. 
                Consequently, all loops $\ell_{e_i}$ are non-trivial, and the concatenation  $e_i \ell_{e_i}$ is reduced. 
            
            We adopt the notation that $\sigma_1\cdot \sigma_2$ is the concatenation of the two paths $\sigma_1, \sigma_2$ with no reduction, and $f_\flat (\sigma_1\cdot \sigma_2)$ is the reduction\footnote{On notations: it seemed to us that $\flat$ was more appropriate than $\sharp$ used in \cite{feighn_conjugacy_2019} since it corresponds to a flattening of the path.} 
            of the image by $f$ of $\sigma_1\cdot \sigma_2$. In particular, it is interesting to know whether it is $f_\flat (\sigma_1)\cdot f_\flat (\sigma_2)$ or not -- a question for which train tracks are relevant.

            The theory of completely split train tracks puts in evidence two types of paths: 
            \begin{itemize} 
            \item the \emph{indecomposable Nielsen paths}, which are of the form $e \ell_e^k \bar e$, for $e$ such that $\ell_e$ is a loop fixed by $f_\flat$ 
            (and $k$ rational such that $\ell_e^k$ is a loop), 
            \item and the \emph{exceptional paths} $e\ell_e^k \bar e'$ for $e$ as before, but where $e'\neq e$ and such that $\ell_{e'}$ is a rational power of $\ell_e$ (both fixed by $f_\flat$), furthermore $\ell_{e'}$ and $\ell_e$ are not the same power so that $e\ell_e^k \bar e'$ is forced to grow linearly (see \cite[Section 3.5]{feighn_conjugacy_2019}). 
            \end{itemize} 
            
            We should comment that Nielsen paths are paths preserved by $f_\flat$ (or in principle by a power, but for completely split train track maps, the exponent is one).  All edges in $Y_0$ are Nielsen paths (as well as all their subarcs), nevertheless, the paths that are called indecomposable Nielsen paths are of the special form  $e \ell_e^k \bar e$ indicated above. 
            
            An important result (\cite[Lemma 4.26]{feighn_recognition_2011})  is that, for all circuit $\sigma$ in $Y$, there exists $k_\sigma$ such that $f_\flat^{k_\sigma}(\sigma)$ is a concatenation \[\sigma_1 \cdot \sigma_2\cdot \dots \cdot \sigma_r\] of $\sigma_i$ being indecomposable Nielsen paths, exceptional paths, and edges, for which for all $m\geq 1$,  $f^m_\flat(\sigma_i \cdot \sigma_{i+1}) = f^m_\flat(\sigma_i)\cdot f^m_\flat(\sigma_{i+1})$.
 
            \begin{lem}\label{lem:linear_little_growth}
            If $\Phi$ is of linear growth,  for every edge $e$, the loop $\ell_{e}$ is a Nielsen path. In particular it is a  concatenation of indecomposable Nielsen paths and edges fixed by $f_\flat$.
            
            \end{lem}
            The edges fixed by $f_\flat$ are those in $Y'_0$, those whose loop $\ell_e$ is trivial. 
                        
            \begin{proof}
            Assume that for an edge $\varepsilon$, the loop  $\ell_{\varepsilon}$ is not preserved by $f_\flat$.  Since all periodic paths under $f_\flat$ have period $1$ (hence are Nielsen in our notations), the iterates of $f_\flat$ on $\varepsilon$  grow at least quadratically, and the eigenray \[ \varepsilon \cdot \ell_{\varepsilon} \cdot f_\flat (\ell_{\varepsilon})\cdot f^2_\flat (\ell_{\varepsilon}) \cdots \] is reduced. 
            Consider $\sigma$ a circuit defining a conjugacy class of $F$ that is not in any proper free factor. Then all the images $f_\flat^k (\sigma)$ must contain all edges of $Y$, hence
            $\varepsilon$, in particular. Choose $k= k_\sigma$ as provided by the completely split train track theorem, and decompose $f_\flat^k (\sigma)$ as above 
             \[f_\flat^k (\sigma)=\sigma_1 \cdot \sigma_2\cdot \dots \cdot \sigma_r.\]
            Let $i_0$ be such that $\varepsilon \in \sigma_{i_0}$.

            Since $\varepsilon$ is of quadratic growth, by \cite[Lemma 2.16]{macura_detour_2002}, it cannot be in an indecomposable Nielsen path, nor in an exceptional path (if it was in $e\ell_e^k\bar e'$, with $e'$ linear, it is different from $e'$, hence is in $e\ell_e^k\bar e$ as well, and the former remark applies).

        
            
            Therefore the piece $ \sigma_{i_0}$ cannot be Nielsen, nor exceptional. Therefore, it is a single edge, $\sigma_{i_0}= \varepsilon$. 
            
            It follows that for all $m$, $f^{mk}_\flat(\sigma)$ contains as a subpath $f^m_\flat(\varepsilon)$. However, the latter grows at least quadratically, which contradicts the linear growth of $\sigma$ under $f$.
            
            We thus have proved that  all loops $\ell_{e}$ are invariant by $f_\flat$.
            \end{proof}
            
            We will now construct the graph of groups decomposition $\bbX$ of $F$ that that will be a cyclic JSJ splitting for $F$-trees where the subgroups in $\calE$ are elliptic. $\bbX$ will be constructed from $Y$ in such a way that the loops $\ell_{\epsilon_i}$ correspond to elliptic elements in the splitting. To do so we will have to define several auxiliary objects simultaneously.
            
            Let  $Y^{''}_0 = Y'_0 \cup \verts Y$.  Its set of edges is $E_0 = \edges{Y'_0}$.
            Let $E_1 = \edges Y \setminus E_0$. This is the set of strictly growing edges through the map $f$ (those whose loop $\ell_e$ is non-trivial). Let us number the edges of $E_1$ as follows, according to our filtration $Y'_1 \subset \dots \subset Y'_{s+1}$: for $i\geq 1$, $\epsilon_i$ denotes the preferred orientation of the only edge in $Y'_{i}\setminus Y'_{i-1}$. 

            The vertex set $\verts X$ of the graph $X$ underlying $\bbX$ is defined to be the set of connected components of $Y^{''}_0$. The edge set is $\edges X = E_1$ and we define $\epsilon_i$ to join (in $X$) vertices $u,v \in \verts X$ if in $Y$ the endpoints of $\epsilon_i$ lie in the connected components $u$ and $v$ of $Y^{''}_0$. It remains to define the vertex groups, the edge groups, and the attachment maps.

            We will now construct a topological space $Z$ such that $\pi_1(Z)\simeq X$ will be a graph of spaces (see \cite{scott_topological_1979}) with underlying graph $X$. The vertex spaces will be graphs, the edge spaces will be circles and we will show that the corresponding graph of groups $\bbX$ has the required properties. Denote by $X^0_v \subset Y^{''}_0$ be the connected component $v$. Let $Z_0= Y^{''}_0$. We will construct $X^i_v, v\in \verts X$ and $Z_i$ inductively.

            Suppose that these have been constructed for $0\leq i <s$: $Z_i$ is a cellular space with the same set of vertices (or $0$-cells) as $Y$, and  so that the $X^i_v$ form a  
            collection subspaces of $Z_i$.
            
            
            Recall that 
            $Y'_{i+1}$ 
            is $Y'_i \cup \{\epsilon_{i+1},  \bar \epsilon_{i+1}\}$, and let $u_{i+1} \in \verts Y$ be the initial vertex and let $v_{i+1} \in \verts Y$ be the terminal vertex of  $\epsilon_{i+1}$. 
            Let $w_{i+1} \in \verts X$ be the connected component of $Y^{''}_0$ containing $u_{i+1}$. 
            
            We define $X^{i+1}_{w_{i+1}} = X^{i}_{w_{i+1}} \cup \mu_{\epsilon_{i+1}}$ where $\mu_{\epsilon_{i+1}}$ is a new loop, based at $u_{i+1} \in X^{i}_{w_{i+1}}$. For all other components $w\neq w_{i+1}$, then $X^{i+1}_w = X^{i}_w$. 
            We define $Z_{i+1}$ by first adding the loop $\mu_{\epsilon_{i+1}}$ based at $u_{i+1}$ that belongs to $X^{i+1}_{w_{i+1}}$,  second adding the edge $\epsilon_{i+1}$ attached at its natural vertices in $Y'_i$,  and third,   attaching a 2-cell as follows. Let $\mathring\ell_{i+1}$ be an immersed loop based at $v_{i+1} \in \verts Z$ such that $\ell_{i+1}$ is homotopic in $Z_i$ to a concatenation $\mathring\ell_{i+1}*\cdots*\mathring\ell_{i+1}$ but such that $\mathring\ell_{i+1}$ cannot be further decomposed as such a concatenation. Suppose furthermore that $\mathring\ell_{i+1}$ lies completely in $X^i_{c+1}$ where $X^i_{c+1} \subset Z_i$ is the subgraph that contains the terminal vertex $v_{i+1}\in \verts Z_i$ of $\epsilon_{i+1}$. Then we attach a 2-cell by immersing its boundary along the immersed loop \begin{equation}\label{eq:attaching_map}\bar\mu_{\epsilon_{i+1}}*\epsilon_{i+1}*\mathring\ell_{\epsilon_{i+1}}*\bar\epsilon_{i+1}.
            \end{equation} To be able to carry out this construction, we now need to prove the existence of a closed loop $\mathring\ell_{\epsilon_{i+1}}$ with the desired properties.

            \textbf{Claim:} For all $i$ there is a homotopy in $Z_{i-1}$relative to the basepoint $v_i$ from the loop $\ell_{\gamma_i}$ in $Y$ to some loop that lies completely in $X^i_v$ where $v$ in $\verts X$ is the connected component of $Y^{''}_0$ that contains the endpoint $v_i$ of  $\epsilon_i$. If $i=1$ then by our choice of ordering, $\ell_{\epsilon_1}$ must lie in $Y^{''}_0$ so the result holds. Suppose this was true for all $0\leq j \leq i$ and consider $\ell_{\epsilon_{i+1}}$. If $\ell_{\epsilon_{i+1}}$ doesn't contain any edges in $E_1$ then it must lie in $X^i_v$ and there is nothing to show. By Lemma \ref{lem:linear_little_growth}, $\ell_{\epsilon_{i+1}}$ must be a concatenation of indecomposable Nielsen paths. Such paths  either consist of individual edges in $Y^{''}_0$ or contain a ``growing'' edge $\epsilon_j$ for some $1\leq j<i+1$. In this case, the indecomposable Nielsen path is 
            of the special form $\epsilon_j \ell_{\epsilon_j}^k \bar \epsilon_j$. 
            
            Consider all maximal (with respect to containment) 
            indecomposable Nielsen subpaths of $\ell_{\epsilon_{i+1}}$. Note that because $\ell_{\epsilon_j}$ cannot contain any edges $\epsilon_k, k>j$ none 
            of these maximal subpaths can overlap. It follows that, $\ell_{\epsilon_{i+1}}$ is a concatenation of disjoint maximal thus special 
            subpaths and edges of $Y^{''}_0$ that lie in some component $c_{i+1} \in \verts X$, it follows that edges from of $Y^{''}_0$ lie in $X^i_{c_{i+1}}$ and the endpoints of any maximal subpaths $\epsilon_j \ell_{\epsilon_j}^k \bar\epsilon_j$ must also lie in $X^i_{c_{i+1}}$. By induction hypothesis, all the $\ell_{\epsilon_j}^k$ are homotopic in $Z_j$ to some path lying completely in $X^j_{c'}\subset X^j_{c'}$ and the 2-cell we added in passing from $Z_{j-1}$ to $Z_j$ lets us homotope the path  $\epsilon_j \ell_{\epsilon_j}^k \bar\epsilon_j$ to the path $\mu_{\epsilon_j}^{kn'}$ relative to endpoints, for some suitable power $kn'$. Thus $\ell_{\epsilon_{i+1}}$ is homotopic in $Z_{i}$ to some loop lying completely in $X_c^{i}\subset X^{i+1}_c$. The claim follows by induction.

            We may therefore construct the sequence of 2-complexes $Z_0\subset\cdots Z_s = Z$. We note that we have a natural inclusion $Y\subset Z$. Furthermore, for each $v \in \verts X$ we have subgraphs $X^{r}_v=X_v\subset Z$. We must now relate $\pi_1(Z)$ to $F\simeq \pi_1(Y)$.
            
            \textbf{Claim:}\textit{ For all $i$ there is a deformation retraction $Z_i \onto Z_{i-1}\cup\epsilon_i$}. Note that $Z_{i+1}$ is created by adding the edge $\epsilon_{i+1}$ the loop $\mu_{\epsilon_{i+1}}$ and a 2-cell to $Z_i$. Since the loop $\epsilon_{i+1}$ is only contained in that 2-cell, there is a deformation retraction, a free face collapse, that maps the 2-cell to path $\epsilon_{i+1}*\mathring\ell_{\epsilon_{i+1}}*\bar\epsilon_{i+1}$ that lies completely in $Z_i \cup \epsilon_{i+1}$. The claim again follows by induction.

            Iterating the claim above gives a deformation retraction $Z\onto Z_0\cup \epsilon_1 \cup \cdots \cup \epsilon_r = Y$ this gives the isomorphism $\pi_1(Z)\simeq F$. Furthermore we have shown that the inclusion $i:Y \subset Z$ induces an isomorphic inclusion $i_\sharp:\pi_1(Y) \stackrel{\simeq}{\into} \pi_1(Z)$. This allows us to to write $\pi_1(Y,y_0)=\pi_1(Z,y_0)$ for any point $y_0 \in Y \subset Z.$
            
            We will now define the graph of groups $\bbX$. The first step is to make a technical modification to $Y$ and $Z$ that will address any basepoint considerations and simplify notation. We could only make the modification at this point since we needed to be able to define the subgraphs $X_v, v\in \verts X$ and establish that $Z$ and $Y$ are homotopically equivalent.
            
            \textbf{Modification:} \textit{For each subgraph $X_v \subset Z$ we take a spanning tree $\tau_v$ and identify it to a point.} We note that $X_v$ either consists of vertices and edges from $Y^{''}_0$ or of loops $\mu_{\epsilon_i}$ that start and end at the same vertex. It follows that\[
                \zeta=\bigcup_{v \in \verts X} \tau_v \subset Y^{''}_0,
            \] is a disjoint union of trees and collapsing every connected component of $\zeta$  in $Y\subset Z$ to a point induces a homotopy equivalence, furthermore, the original homotopy equivalence $f:Y \to Y$ that carries $\Phi$ descends to this quotient of $Y$. We will therefore assume, without loss of generality, that $Y$ is this quotient graph. Since $\zeta \subset Z$ is also a forest, collapsing the components of $\zeta$ to points will also be a homotopy equivalence of $Z$.
            
            \textbf{Modification:} \textit{For each $v \in \verts X$ we now identify the vertices of $X_v$ (which are also the vertices of $\tau_v$, which we view as 0-cells of $Z$, with the vertex $v$ and set $v\in \verts X$ to be the basepoint of $X_v$. By collapsing the components of $\zeta$ we can identify the vertices of $Z$ with the vertices of $X$.} We also note that now whether viewing $\epsilon_i$ as an edge of $X$ or of $Z$, there is no longer any ambiguity as to what the initial and terminal vertices are.
            
            We can now at last fully specify the graph of groups $\bbX$. The vertex groups of $\bbX$ are $\bbX_v = \pi_1(X_v,v)$. For each edge $\epsilon_i \in \edges X$,  we note that the loop $\ell_{\epsilon_i}$ is a loop that is based at $v = \tau(\epsilon_i)$ and that it is homotopic (relative basepoint) in $Z$ to a power of $\mathring\ell_{\epsilon_1} \in \pi_1(X_v,v)$ we set $c_{\epsilon_i}=\mathring\ell_{\epsilon_i} \in \bbX_{\tau(\epsilon_i)}$. Similarly, if $i(\epsilon_i)=u$, then we define  $c_{\bar\epsilon_i} \in \bbX_{u}$ to be the loop $\mu_{\epsilon_i} \in \pi_1(X_u,u)$ that is directed so that it is freely homotopic to $\mathring\ell_{\epsilon_i}$ in $Z$. We denote the cyclic edge groups $\bbX_{\epsilon_i} = \bk{g_{\epsilon_i}}$ and we define the edge group monomorphisms via the generator mappings $i_{\epsilon_i}: g_{\epsilon_i} \mapsto c_{\bar \epsilon_i}$ and $\tau_{\epsilon_i}: g_{\epsilon_i} \mapsto c_{\epsilon_i}$. It is straightforward from \eqref{eq:attaching_map} to see that the graph of groups $\bbX$ defined in this way has fundamental group isomorphic to $\pi_1(Z)$.
            
             We now must show that the linearly growing outer automorphism $\Phi\in \out F$ is given by an $\edges\bbX$-substitution. We identify $F=\pi_1(Y,y_0)$ for some vertex $y_0 \in \verts Y=\verts X$. Let $g \in \pi_1(Y,y_0)=\pi_1(Z,y_0)$. $g$ is an edge path in $Y$  and we express it as follows
             \begin{equation}\label{eq:og-loop}
               g= y_0\epsilon_{i_1}^{\delta_1}\cdots \epsilon_{i_l}^{\delta_l} y_l,
            \end{equation} 
            where the $y_i$ are edges in $Y^{''}_0$ and where $\delta_i \in \pm 1$ and we adopt the convention that $\epsilon_j^1=\epsilon_j$ and $\epsilon_j^{-1} = \bar \epsilon_j$. We note that since all the components of $Y\setminus E_1$ only have one vertex, \eqref{eq:og-loop} is literally a $\bbX$-loop based at $y_0$. Applying $f$ to $g$ replaces $\epsilon_i$ with $\epsilon_i\ell_{\epsilon_i}$. After homotoping $\ell_{\epsilon_i}$ to $\mathring\ell_{\epsilon_i}^{k_i}$ in $Z$, which now sits in $X_{\tau(\epsilon_i)}$, we see that this $\Phi$ is precisely the one-sided $\edges \bbX$-substitution\[
                \epsilon_i \mapsto \epsilon_i c_{\epsilon_i}^{k_i}.
            \]
            
            Consider the elements of $F$ whose conjugacy classes are periodic under iteration of $\Phi$ and let $\calE$ denote the collection of maximal subgroups of $F$ consisting of these elements.
            
            \textbf{Claim:} \textit{$\calE$ consists precisely of the elliptic subgroups of $\bbX$.} Indeed suppose there is some $g \in \pi_1(Y,y_0) = \pi_1(\bbX,y_0)$ that is hyperbolic in $\bbX$ but whose conjugacy class is periodic. Forgetting basepoints and considering $g$ as an immersed loop in $Y$ we find that for some $k$, $f_\flat^k(g)=g$ so that $g$ is a Nielsen path. It follows that we have a splitting $f_\flat^k(g) = \sigma_1\cdot\cdots\cdot\sigma_r$  where each $\sigma_i$ either lies in $Y^{''}_0$ or is of the special form $\epsilon_i \ell_{\epsilon_i}^k\bar\epsilon_i$ of indecomposable Nielsen paths. By construction of $Z_i$, any subpath in special form can be homotoped into $X_{i(\epsilon_i)}$ and so it follows that $g$ is actually homotopic in $Z$ to some loop in $X_v$ for some $v\in \verts X$. It follows that $g$ is elliptic in $\bbX$. The claim is proved.
            
            It follows that we can take $\calE$ to be the set of vertex groups $\calE = \{\bbX_v| v\in \verts X\}$. Let $T$ be the Bass-Serre tree dual to $\bbX$. Since $\bbX$ has infinite cyclic edge groups and since vertex groups are one-ended relative to themselves, \cite[Corollary 1.5]{touikan_one-endedness_2015} implies 
            $F$ is one-ended relative to $\calE$. We also have that $T$ dominates any $(\calA,\calE)$-tree so it is in fact an $(\calA,\calE)$-JSJ tree (recall that $\calA$ in this context is the class of cyclic groups and recall the definitions in Section \ref{sec:G-trees}). 
             By \cite[Theorem 9.5]{guirardel_jsj_2017}, $F$ admits a unique canonical $(\calA,\calE)$-tree of cylinders $T_c$ relative to $\calE$ for the relation of commutation, and since all abelian groups are cyclic, $T$ and $T_c$ are in the same deformation space. Let $\bbX_c$ be the graph of groups dual to the action of $F$ on $T_c$.
            
            The process of passing from $T$ to $T_c$ may be complicated. Although, by \cite[Theorem 9.5]{guirardel_jsj_2017}, both trees have the same non-abelian vertex groups, there can be subtrees of $T$  all of whose vertex groups are cyclic. 
            Nonetheless, $\Phi$ preserves the subgroups $\calE$ so by the invariance properties of trees of cylinders it preserves the tree $T_c$. It follows that $\Phi$ can be represented by an element in $\delta\aut{\bbX_c}$. Furthermore, the non-abelian vertex groups of $T_c$ are $\Phi$-non-growing and since the cylinders subgroups of non-abelian vertex groups, these are also $\Phi$-non-growing. It follows that $\Phi=(Id_{X_c},(Id_v),(Id_e),(\gamma_e))$, i.e. 
            $\Phi$ induces an $\edges{\bbX_c}$-substitution. Furthermore, since edge groups are maximal cyclic, we can arrange this substitution to be one-sided. We may pick an edge orientation so that all edges point away from black vertices, which are the cylinder vertices. We also note that this implies that edge groups map surjectively onto adjacent cylinder groups.
            
            Finally, to see that this substitution is full, suppose towards a contradiction that in $\bbX_c$ there were two distinct edges $e_i,e_j$ pointing away from a common vertex black $b$ such that the substitution coefficients cancelled out i.e. $c_{e_i}^{-1}e_i^{-1}e_jc_{e_j} = e_{i}^{-1}e_j^{-1}$. Then letting $u = \tau(e_i)$ and $v = \tau(e_j)$ and noting that the tree of cylinders construction forces the vertex groups $(\bbX_c)_u$ and  $(\bbX_c)_v$ to be non-abelian, we can find $g_u \in (\bbX_c)_u$ and $g_v \in (\bbX_c)_v$ that do not commute with $c_{e_i}$ and $c_{e_j}$ respectively so that the element $g=g_ue_i^{-1}e_jg_ve_j^{-1}e_i$ is an element in $\pi_1(\bbX_c,u)$ written out in reduced form. By choice of $g_u$ and $g_v$ this element is hyperbolic in $\bbX_c$ but note that the substitution induced by $\Phi$ leaves $g$ unchanged, this contradicts that elements in $\calE$ are elliptic in $\bbX_c$. This completes the proof of the theorem.
            
        \end{proof}
 
        \begin{cor} \label{cor;ulg}
            An automorphism of a free group $F$ is unipotent of linear growth, if and only if it is an $\edges \bbX$-substitution for some cyclic splitting of $F$, if and only if it is an element of the small modular group for a free  product decomposition of $F$ (aka a generalized Dehn twist).
        \end{cor}

    \section{Piecewise trivial suspensions}

        \subsection{Piecewise trivial suspensions and suspensions of generalized Dehn twists}\label{sec;gen-Dehn-twists}

        We use the notation $\calA$ for the collection of abelian subgroup of a group.

        \begin{defn}\label{def;pts}
            A group $\Gamma$ acting on a $\calA$-bipartite tree $T$ is a piecewise trivial suspension over $T$  if: 
            \begin{itemize}
                \item black vertex stabilizers, and edge stabilizers are free abelian of rank $2$,
                \item for every white vertex $v$, its stabilizer $\Gamma_v$ is a direct product of a subgroup $G_v$ of $\Gamma$  with trivial centre by an infinite cyclic subgroup $\langle t_v \rangle$,
                \item for every white vertex $v$, the element $t_v$ and its non-trivial powers fix exactly the ball of radius $2$ around $v$.
            \end{itemize}
            
            We call $G_v$ the local direct factor, or the local fibre (with respect to the white vertex $v$). We note that the centre of every vertex group is characteristic and that the local direct factor is well-defined up to automorphism. When $\Gamma$ is a semi-direct product $G\rtimes \mathbb{Z}$, and when the structure of piecewise trivial suspension involves the groups $G_v = G\cap \Gamma_v$, we may say that it is globally fibered, with fibre $G$. 
            
        \end{defn}
        
        \begin{prop}\label{prop;PTS_are_acyl}
         If $\Gamma$ is a piecewise trivial suspension over $T$, 
         for any black vertex $b$ in $T$, and any different vertices $v,v'$ adjacent to $b$, the stabilizer $\Gamma_b$ of $b$ contains $\bk {t_v} \oplus \bk{t_{v'}}$ as a finite index subgroup. Moreover,  the action of $\Gamma$ on $T$ is $4$-acylindrical. Finally, for any pair of distinct edges adjacent to a white vertex, their stabilizers have cyclic intersection.
        \end{prop}
        \begin{proof}
        First, if a black vertex $b$ is between two white vertices $v, v'$, then $t_v$ and $t_{v'}$ are in $\Gamma_b$, the stabilizer of $b$, and therefore commute. We claim that they do not have a common non-trivial power. Assume otherwise, and let $b'$ another black neighbour of $v'$ and let $v''$ another white neighbour of $b'$. A power of $t_{v}$ then fixes $v$, $v'$ (by the third point of the definition) and $v''$ (because it equals a power of $t_{v'}$, and $t_{v'}$ fixes $v''$ by the third point of the definition). Therefore it fixes more than the ball of radius $2$ around $v$, contradicting the third point of the definition. 
        
        We have therefore established that $t_v, t_{v'}$ generate a subgroup that is abelian of rank $2$ in $\Gamma_b$, and therefore it is of finite index in $\Gamma_b$. Obviously, this subgroup fixes the star of $b$, because it lies at distance $\leq 2$ from $v$ and $v'$.   
        

        Let us now prove acylindricity. 
        Consider a segment of length $5$. Its ends have different colours, so assume that it starts at a white vertex $v_1$. Let $v_1, v_2, v_3$ the consecutive white vertices of this segment, and $b_{12}, b_{23}, b_{e}$ the consecutive black vertices ($b_{ij}$ is between $v_i$ and $v_j$, $b_e$ is the end of the segment). 
        
        Assume towards a contradiction,  that an element $g\neq 1$ fixes the segment. It is in the stabilizer of $b_{12}$, and of $b_{23}$, therefore a positive power of $g$ is in both $\langle t_{v_1}, t_{v_2} \rangle $ and $ \langle t_{v_2}, t_{v_3} \rangle$, because they are both of finite index in the stabilizers of $b_{12}$ and  $b_{23}$. 
        
        We thus have integers $a,b,c,d$ such that  $t_{v_1}^at_{v_2}^b =  t_{v_2}^ct_{v_3}^d$. However, since by the third point of the definition, $t_{v_1}^a \notin \Gamma_{b_{23}}$, but the other are, we have $a=0$. Symmetrically, $d=0$ as well, and $g$ is a power of $t_{v_2}$. But we already know that such a power only fixes a ball of radius $2$. It cannot fix ${b_e}$, contradicting our assumption. This proves the acylindricity. 
        
        Finally, consider two black vertices adjacent to a same white vertex. They are both fixed by a power of the generator of the centre of the white vertex stabilizer. Assuming the intersection of their stabilizers is non-cyclic, it would be abelian of rank $2$, and of finite index in both black vertex stabilizers. Consider a white vertex $v'$ one step further from one of the black vertices. The generator of its centre has some power that is in this intersection. It fixes only the ball of radius $2$ around $v'$, but the black vertex on the other side is at distance $3$, contradiction.
        
        \end{proof}
        
        It is perhaps worth noting that, in the context of Proposition \ref{prop;PTS_are_acyl}, if the local fibres are CSA (i.e. their maximal abelian subgroups are malnormal), the tree $T$ is its own tree of cylinders for the class $\calE$ of infinite abelian subgroups that have an infinite intersection with a local fibre of a vertex group, and the equivalence relation where that two groups are equivalent if the union of their intersections with the local fibres generates an abelian group is an admissible equivalence relation in the sense of Section \ref{sec:G-trees}.

        The next proposition, that says that generalized Dehn twists give rise to piecewise trivial suspensions, is easy for classical Dehn twists (defined for an amalgam or an HNN extension), but requires some care for generalized Dehn twists defined on a graph of groups.

        \begin{prop}\label{prop;gdt-PTS}
	        Let $\bbX$ be a bipartite graph-of-groups with  cyclic black vertex  groups, $v_0 \in \verts X$, and let $G= \pi_1(\bbX, v_0)$. Let $T$ be the Bass-Serre tree dual to $\bbX$. 
	        
	        Assume  that    every edge stabilizer in $T$ is equal to its adjacent black vertex stabilizer, and two edges adjacent to a same white vertex have stabilizers with trivial intersection.

	        Let  $\alpha$ be a generalized Dehn twist of $\bbX$ that is a full substitution.

	        Then $G\semidirect\alpha\bk t$ is a 
	        piecewise trivial suspension  over  $T$, which is globally fibered with fibre $G$. Furthermore, for $G\semidirect\alpha\bk t$, the stabilizers of edges coincide with the stabilizers of adjacent black vertices.

        \end{prop}

        \begin{proof}
            By assumption, $\alpha$ is a (1-sided) $\edges X$-substitution. Consider $\torus \alpha = G\rtimes_\alpha \bbZ$. 
            Clearly, $T$ is a $\calA$-bipartite $\torus \alpha$-tree. 
            Consider $v$ a vertex of $T$, its stabilizer   ${\rm Stab}_G(v)$ (in $G$) is the image of the subgroup $e_1\dots e_r \bbX_{\nu} \bar e_r\dots \bar e_1$ of the Bass group,  for some vertex group $\bbX_\nu$ and consecutive edges $e_1,\dots e_r$ in $X^{(1)}$ (with $i(e_1) =v_0$ a white vertex). 

            If $g'\in{\rm Stab}_G(v)$,  expressed as  $e_1\dots e_r g \bar e_r\dots \bar e_1$ for $g\in \bbX_\nu$.     Its image by $\alpha$ is $\eta^{-1} g \eta$ for $\eta^{-1} =  e_1c_{1}  e_2 c_{2} \dots  e_r c_{r}$, by the definition of the substitution (the coefficients $c_i$ are the products, in white vertex groups of the substitution coefficients $c_{e_i}$ of $e_i$ and $c_{\bar e_{i+1}}^{-1}$ of $\bar e_{i+1}$ ;   one could work with $2$-sided substitutions too, here). In other words, $\alpha(g') = g'^{e_1\dots e_r \eta}$.

            Therefore, if $t$ is the element of $\torus\alpha$ that induces $\alpha$ by conjugation on $G$,   $(g')^t = g'^{e_1\dots e_r \eta}$, and  the element $t_v = t \eta^{-1} \bar e_r\dots \bar e_1$ induces the identity on  ${\rm Stab}_G(v)$ by conjugation. This proves the two first points of the definition of piecewise trivial suspension (and that black vertex stabilizers are equal to their adjacent edges stabilizers). 

            If (as it is the case with the assumption on acylindricity) 
            ${\rm Stab}_G(v)$ fixes only one white vertex, then  $t_v$ has to fix it too.


            Moreover, if the substitution is full, consider a triple of adjacent edges $e_{r+1},e_{r+2}, e_{r+3}$, forming a reduced path from a white vertex $v$ to a black vertex $v'$.  We need to check that no non-trivial power of $t_v $ fixes $v'$. If one fixes $v'$, it must normalize its  stabilizer: the image (from the Bass group) of   \[e_1\dots e_re_{r+1}e_{r+2} \bbX_{\mu} \bar e_{r+2} \bar e_{r+2} \bar e_r\dots \bar e_1.\] 

            Let $g$ be non-trivial in $\bbX_\mu$ and \[g'=e_1\dots e_re_{r+1}e_{r+2}e_{r+3}\, \cdot  g\cdot  \,\bar e_{r+3} \bar e_{r+2} \bar e_{r+1} \bar e_r\dots \bar e_1 \in \Stab(v').\] 
            
             After applying the substitution, by our hypothesis that edge groups in $\bbX$ map surjectively onto cylinders groups, the relations in the Bass group allow us to completely migrate (by relations of the Bass group)     
            the substitution coefficients across edges:\[
            e_{r+1}e_{r+2} \mapsto c_{e_{r+1}}^{-1}e_{r+1}e_{r+2}c_{e_{r+2}}= e_{r+1}e_{r+2}c'.
            \] 
            Since the substitution is full, we have that $1 \neq c' \in \tau_{e_{r+2}}(\bbX_{e_{r+2}}).$ 
            
            Conjugating $g'$ by $t_v =   t \eta^{-1} \bar e_r\dots \bar e_1$, one obtains 
            \[ (g')^{t_v} =    g^{(  \bar e_{r+3} c(c')^{-1}\bar e_{r+2}  \bar e_{r+1} \bar e_r \bar e_{r-1} \dots \bar e_1)} \]
            with $c \in \tau_{\bar e_{r+3}}(\bbX_{e_{r+3}})$ and $(c')^{-1}$ is  nontrivial as above. It follows that the element $c(c')^{-1}$ cannot be conjugated into the image $\tau_{\bar e_{r+3}}(\bbX_{e_{r+3}})$, so from the reduced form it is immediate that $(g')^{t_v}$ is not in the stabilizer of $v'$.



 




        \end{proof}
    It is worth noting that if we drop the requirement, in the fiber,  that the images of edge groups of $\bbX$ are surjective onto black vertex groups, then the action of $G\semidirect\alpha\bk t$ on the tree $T$ will no longer be acylindrical.
    Theorem \ref{thm;unipotent-linear-gdt} now immediately implies the following.

    \begin{prop}\label{prop;ULG->PWT}
      Let $F$ be a free group, $\alpha$ be a linearly growing unipotent automorphism, and $\calE$ the class of $\alpha$-non-growing subgroups of $F$. If $T$ is the tree of cylinders for abelian splittings of $F$ relative to $\calE$ then $\Gamma= F\semidirect \alpha \bk t$ is a globally fibered  piecewise trivial suspension over $T$, as a $\Gamma$-tree.
    \end{prop}

    \subsection{Canonical structural trees of  piecewise trivial suspensions}
               
        In this subsection, we will prove the following, and then discuss consequences.
    
        \begin{prop}\label{prop:structural_tree_is_canonical}
        Let $\Gamma$ be a group, with two $\calA$-bipartite $\Gamma$-trees $T, T'$ each 
        such that $\Gamma$ is  a  piecewise trivial suspension over $T$ and over $T'$.  Then,  there is an isomorphism of $\Gamma$-trees $f:T\to T'$.

        \end{prop}

        \begin{proof}

        Let us  define an equivariant bijection between the white vertices of $T$ and the white vertices of  $T'$.
        
        Consider a white vertex $v$ of $T$, and $\Gamma_{v}$ its stabilizer in $\Gamma$. We claim that $\Gamma_{v}$ is elliptic on $T$. Would it be otherwise, it would have a hyperbolic element $g_{v}$, with an axis in $T$. However, every  element of $\Gamma_{v}$ has a subgroup  isomorphic to $\mathbb{Z}^2$ in its centralizer (if $g_v \notin \langle t_v\rangle$, it contains $\langle g_v,t_v\rangle$, and if $g_v\in \langle t_v\rangle$,  any $h\in G_v<\Gamma_v$ commutes with $g_v$). Therefore the axis of $g_v$ in $T'$ is pointwise fixed by an infinite cyclic group, and this contradicts the acylindricity of the action of $\Gamma$ on $T'$ (Proposition \ref{prop;PTS_are_acyl}).  We have established that $\Gamma_v$ is elliptic. But $\Gamma_v$ cannot fix any black vertices of $T'$, because the local direct factor $G_v$ is assumed to be non-abelian. Therefore it fixes a unique white vertex. This produces a well defined equivariant map from the set of white vertices of $T$ to the set of white vertices of $T'$.
        
        By symmetry of the assumptions, we have an equivariant map from the white vertices of $T'$ to those of $T$. The compositions are the identities on $T$ and $T'$: if $v$ is a white vertex in $T$, the image of $v$ is fixed by $\Gamma_v$, so $\Gamma_{v'}$ contains $\Gamma_v$, and then the image of $v'$ is fixed by $\Gamma_{v'}$, so in particular by $\Gamma_{v}$, and it is white, so it is $v$.
        
        Now that we have an equivariant bijection between white vertices of $T$ and $T'$, we extend it to the whole tree.
    
        Consider two white vertices in $T$ that are separated only by a black vertex. They are sent on different white vertices of $T'$ (by injectivity) and the intersection of their stabilizers, which is abelian non-cyclic, by the first point of Proposition \ref{prop;PTS_are_acyl},  (as a subgroup fixing a black vertex in $T$)   fixes the segment between their images. If this segment contains a white vertex in its interior, by  the last point of Proposition \ref{prop;PTS_are_acyl}, the fixator of the segment is cyclic or trivial.  This is a contradiction with non-cyclicity of the intersection. 
        
        

        \end{proof}

         If $\Gamma$ is a piecewise trivial suspension,    Proposition \ref{prop:structural_tree_is_canonical} allows us to talk about the unique (up to isomorphism) tree $T$ over which it is a piecewise trivial suspension. We call it its  \define{canonical structural tree}. All automorphisms of $\Gamma$ must therefore give rise to equivariant automorphisms of  $T$:

    \begin{prop} \label{prop;delta-aut-surjecte-sur-out}
        Let $\Gamma$ be a piecewise trivial suspension and let $T$ be a canonical structural tree. Let $\bbX$ be the  graph-of-group dual to $T$. Then the natural map $\delta \aut\bbX \to \out{\Gamma}$ is surjective.
    \end{prop}

    Another consequence of the canonicity of structural trees is the following computational result which is important for calculations later in the paper.

    \begin{prop}\label{prop;struct_tree_compute}
      Let $F$ be a finitely generated free group, $\alpha$ a linearly growing unipotent  automorphism, and let $\Gamma = \torus\alpha$ as in Proposition \ref{prop;ULG->PWT} then it is possible algorithmically to construct a splitting $\bbX$ of $\Gamma$ that is dual to the structural tree $T$ for $\alpha$ directly from a finite presentation of $\Gamma$.
    \end{prop}
    \begin{proof}
        Let $\bk{X|R}$ be a presentation for $\Gamma$, and enumerate the presentations for $\Gamma$ until there is a presentation of $\Gamma$ as a graph of groups where:
        \begin{enumerate}
            \item The vertex groups are either explicitly isomorphic to $\bbZ^2$ or explicitly direct products of non-abelian free groups and cyclic groups.
            \item The underlying graph is bipartite with $\bbZ^2$ black vertices.
            \item The images of edge groups are maximal abelian subgroups of the vertex groups.
            \item In the white (non-abelian) vertex group, distinct edges groups map to non-conjugate subgroups.
            \item The centralizer of a white vertex group is not equal to the centralizer of another white vertex group at distance 2.
        \end{enumerate}
        Since such a splitting $\bbX'$ of $\Gamma$ exists, an enumeration of all finite presentations of $\Gamma$ will eventually produce such a splitting. Furthermore, all such conditions are easily verifiable in the vertex groups and from the Bass-Serre relations.
        
        The conditions above immediately imply the two first properties of the definition of piecewise trivial suspensions. Let us now check the third property.
         Consider, for each white vertex $w$ in $\bbX'$, a generator $t_w$ of the given cyclic direct factor.  We need to check that if $e_1e_2e_3$ is a reduced path in the Bass Serre tree of $\bbX'$ starting at the vertex fixed by $\Gamma_w$,  then $t_{w} $ does not fix $e_3$. Let $b$ be the black vertex adjacent to $e_1$, by the third condition we verified on $\bbX'$, it follows that the stabilizer of $e_1$ and $b$ coincide and similarly that it coincides with the stabilizer of $e_2$. But, by the fifth condition $t_w$ is not central in $\Gamma_{w'}$, for $w'$ the vertex between $e_2$ and $e_3$. By properties of maximal abelian subgroups of direct products of a free group and a cyclic group and the fourth condition, $t_w$ fixes only one edge adjacent to $w'$, and it must be $e_2$. Therefore it does not fix $e_3$, and we are done. 
         
        Proposition \ref{prop:structural_tree_is_canonical} then implies that $T'$ the Bass-Serre tree dual to $\bbX'$ is isomorphic to the canonical tree $T$.
    \end{proof}

    \subsection{Subgroups of globally fibered piecewise trivial suspensions}

    In this subsection we study finitely generated subgroups of a globally fibered piecewise trivial suspension $\Gamma$. We will denote by $T$ the canonical structural tree of $\Gamma$, and $G$ the global fiber: the local fibres are the intersections of $G$ with the vertex stabilizers. We denote by $t$ an element such that 
    $\Gamma = G\rtimes \bk t$. Our main example is when $G$ is a free group.


        \begin{prop}\label{prop;subpdp}
            Assume that $\Gamma$ is as above, and that the local fibres, i.e. the intersection of  $G$ with the vertex stabilizers, are free.  Let $H<\Gamma$ be a finitely generated freely indecomposable subgroup of $\Gamma$, with minimal subtree $T_0 \subset T$, that is infinite. 
            
            If for all edges $e$ of $T_0$, the index of $H\cap \Gamma_e $ in $\Gamma_e$ is finite, then   $ T_0 $ is a $\calA$-bipartite  $H$-tree, 
            and $H$ is piecewise trivial suspension over $T_0$. 
        \end{prop} 

        \begin{proof}

            Consider a white vertex $v$ in $T_0$.  Its stabilizer $\Gamma_v$ in $\Gamma$ is a direct product   $G_v \oplus \bk {t_v}$. Let $H_v$ denote $H\cap \Gamma_v$. We first show that  $H_v$ is a direct product of a non-abelian free group and a cyclic group.
            
            Consider the quotient map  $\Gamma_v \onto  \Gamma_v/\bk{t_v} $. In restriction to   $H_v$, it is non-injective, because it is so on the adjacent edge stabilizers. Therefore $ H_v$ intersects 
            $\bk {t_v}$, necessarily in an infinite cyclic group. 
            
             The quotient by $H_v\cap \bk{t_v}$ maps $H_v$ into $\Gamma_v/ \bk{t_v}$, which is isomorphic to $G_v$. Hence the image of    $H_v$ by this quotient is free, and therefore $H_v$ splits as a semi-direct product  $H_v= (\bk{t_v}\cap H)\rtimes F_v$, for $F_v$ any lift of the free group  $H_v/(H\cap \bk{t_v}))$ in $H_v$. 
            Since the left factor is central, this is a direct product. Observe however that we cannot guarantee that the other factor is contained in $G_v$.
            
            Since $H_v$ intersects different (at least $2$) edge stabilizers adjacent to $v$ as finite index subgroups, it is non-abelian (recall that, by Proposition \ref{prop;PTS_are_acyl}, two distinct incident edge groups intersect in a cyclic subgroup that itself trivially intersects the local fibre of $v$, thus their union generate a non-abelian group).  
            
         
            For each edge of $T_0$, its stabilizer in $H$ is of finite index in its stabilizer in $\Gamma$, by assumption, so it is free abelian of rank $2$. This holds too for stabilizers of black vertices of $T_0$ in $H$. It remains to show the third point of  Definition \ref{def;pts}  for $H$ on $T_0$. 
            
            If $v$ is a white vertex of $T_0$, we know 
            that all non-trivial powers of $t_v$ fix exactly the ball of radius $2$ around $v$, and we know also that the centre of $H_v$ is generated by a power of  $t_v$. Therefore we have the third point of the definition.
        \end{proof}

        We will say that a subgroup $K\leq H$ is \define{thick} if contains a subgroup isomorphic to $\mathbb Z^2$. If a subgroup $K$ is not thick, we will say it is \define{thin}.

        Recall from \cite{rips_cyclic_1997,fujiwara_papasoglu,dunwoody_sageev,guirardel_jsj_2017} that a group $R$ is \define{rigid} relative to a collection $\calH$ of subgroups if $R$ admits no abelian splittings relative to $\calH$. A vertex groups $\bbX_q$ of a graph of groups $\bbX$ with cyclic edge groups is said to be \define{quadratically hanging (QH) in $\bbX$} it is the fundamental group of a compact surface $\Sigma$ with boundary and the incident edge groups coincide with the $\pi_1$-images of the connected components of the boundary $\partial\Sigma$. We say a QH subgroup is \define{maximal QH} if it is not conjugate into a proper subgroup of another QH group.
        
        The next proposition analyses a freely indecomposable finitely generated subgroup of a piecewise trivial suspension, with a free global fibre. In order to do so we will need to consider a new JSJ theory. Because many properties need to be verified, the proof below is long. We refer to reader to Section \ref{sec:G-trees} for the terminology that will be used in the proof.
        
        \begin{prop} \label{prop;canonical_for_subgroups}
	        Assume that $\Gamma$ is a globally fibered piecewise trivial suspension over a tree $T$, with a global fibre  $G$ that is free.
	        Let $H<\Gamma$ be a finitely generated freely indecomposable subgroup of $\Gamma$, with minimal subtree $T_H \subset T$, that is infinite. Let $\calA$ be the class of cyclic subgroups of $H$, 
	        and $\calH$ be the collection of subgroups of $H$ that are piecewise trivial suspension. Then there exists $\calT$ a collapsed tree of cylinders of an $(\calA,\calH)$-JSJ tree for the relation of commensurability.
	                
	        The action of $H$ on $\calT$ is dual to a bipartite graph of groups $\bbX$ with cyclic edge groups edge groups. Black vertex groups are maximal cyclic and white vertex groups $\bbX_w$ either:
	        \begin{enumerate}[(a)] 
	            \item\label{it:max-pts} isomorphic to piecewise trivial suspensions,
	            \item\label{it:free} free groups that are either do not admit any cyclic splitting relative to incident edge groups or that are maximal QH in $\bbX$.
	       \end{enumerate}
	       Furthermore the following hold:
	       \begin{enumerate}
	           \item\label{it:interstitial-central} For any edge of $\calT$ adjacent to a vertex $w$ of with edge group $\bbX_w$ of type (\ref{it:max-pts}), the edge group $\bbX_e$ is coincides with the centralizer of a subgroup isomorphic to $F\oplus \bbZ$ that is a vertex group of the canonical splitting of $\bbX_w$ as a piecewise trivial suspension.
	           \item\label{it:verts-canonical} Any other splitting of $H$ with the same 
	       \end{enumerate}
	       Finally $\delta\aut{\bbX}$ maps surjectively onto $\out H$.
        \end{prop}  

    \begin{proof}
        Let $T_H$ be the minimal $H$-subtree for the action of $H$ on $T$. Non-abelian thin vertex groups of $T_H$ are subgroups of white vertex groups of $T$, so these must be subgroups that intersect the centre of the white vertex group trivially and therefore are themselves free groups. Since $H$ is one-ended all edge groups of $T_H$ must be isomorphic to $\bbZ$ or $\bbZ^2$.
        
        We call a maximal subtree of $T_H$ consisting of edges with $\bbZ^2$ edge stabilizers a \emph{thick component}. By Proposition \ref{prop;subpdp} the stabilizers of thick components are themselves piecewise trivial suspensions, though it is possible that they do not admit a global fibre.
        
        We colour the vertices of $T_H$ the same as the vertices of $T$ and we say a vertex $v$ is \emph{thick,thin} if the vertex group $H_v$ is thick, thin respectively. Similarly we will say a vertex is \emph{black,white} if the vertex it stabilizes is black, white respectively. We first investigate cyclic edge groups of $T_H$.
        
        It is clear that any edge of $T_H$ with cyclic stabilizer will not be adjacent to a thick black vertex. This next fact is a bit more subtle.
        
        \textbf{Claim:} \emph{If an edge $e$ of $T_H$ with a cyclic stabilizer is adjacent to a thick subgroup $v$ then $v$ must be white and the subgroup $H_e \leq H_v$ must lie in the centre of $H_v$.} Consider the overgroup $\Gamma_v \geq H_v$. We have $\Gamma_v = G_v\oplus\langle t_v\rangle$ and $H_v \simeq K_v\oplus\langle t_v^{n_v}\rangle$ for some subgroup $K_v \leq G_v$ and where the subgroup $\langle t_v^{n_v}\rangle$ is the centre of $H_v$. If the image of $H_e$ in $H_v$ was not central, then $\langle H_e,t_v^{n_v}\rangle \simeq \bbZ^2$ is a subgroup of $H$ that stabilizes the edge $e$ of $T$, contradicting $H_e \simeq \bbZ$. This proves the claim.
        
        The argument above also shows that white vertex groups of $T_H$ are non-abelian, for otherwise such a vertex group would be completely contained in an edge group. Since distinct edges incident to a white vertex group have distinct edge groups, this would force the vertex to have valence 1 which contradicts the minimality of $T_H$.
        
         We define a \emph{thin component} of $T_H$ to be a maximal subtree whose vertices are thin. The acylindricity of the action of $\Gamma$ on $T$ and the claim above imply that a thin black vertex is adjacent to at most one thick white vertex and therefore that non-empty thin components are non-abelian. We say an edge is \emph{interstitial} if it joins a thick and a thin vertex. Interstitial edges are precisely the edges that are neither in a thick or a thin component and their stabilizers coincide with those of the edges considered in item \ref{it:interstitial-central} in the statement of the proposition.
         


        We will now modify $T_H$ in order to get a canonical $H$-tree. However, doing so requires some care. We want $\calA$, our class of edge groups, to be either the class of cyclic or abelian subgroups, but since $T_H$ has a mix of cyclic and abelian edge groups, and since $H$ does not satisfy commutation transitivity, there is no possible equivalence relation on $\calA$ that is admissible described in Section \ref{sec:G-trees}. Our solution is to let $\calH$, our peripheral structure, be the set of subgroups of $H$ that are isomorphic to piecewise trivial suspensions, to let our edge groups lie in $\calA$ the set of infinite cyclic groups, and to let $\sim$ be the commensurability relation. For $H$, given our knowledge of piecewise trivial suspension, it is easy to verify $\sim$ is an admissible equivalence relation on $\calA$. Our first task is to show that there indeed exists a canonical $(\calA,\calH)$-tree.


         
         \textbf{Claim:}\emph{With the equivalence relation $\sim$ and the classes of subgroups $\calA$ and $\calH$ given above, $H$ admits a canonical $(\calA,\calH)$-tree $\calT$ that is constructed as the tree of cylinders of an $(\calA,\calH)$-JSJ tree.} This amounts to showing that the hypotheses of \cite[Corollary 9.1]{guirardel_jsj_2017} are satisfied. In our case, it will be sufficient to show two things. Firstly, that for every $A \in \calA$ the stabilizer (under conjugation) $H_{[A]}$ of the $\sim$-equivalence class $[A]$ is \emph{small in any $(\calA,\calH)$-tree}, which means that $H_{[A]}$ will always fix a vertex or a point in the boundary of an $(\calA,\calH)$-tree. Secondly, that if $A\leq A', [A':A]=2$ and $A \in \calA$, then $A' \in \calA$.
         
         Let $S$ be an arbitrary $(\calA,\calH)$-tree and let $A \in \calA$. The group $H_{[A]}\leq H$ that fixes the equivalence class $[A]$ via conjugation is then the normalizer of the maximal cyclic subgroup containing $A$. If $H_{[A]}$ is not cyclic then $H_{[A]}$ must be contained in a subgroup isomorphic to a piecewise trivial suspension and therefore must be contained in a subgroup in $\calH$, so $H_{[A]}$ fixes a vertex of $S$. If $H_{[A]}$ is cyclic, by the classification of isometries of simplicial trees, either it fixes a vertex or leaves a line of $S$ invariant. In all cases the action of $H_{[A]}$ is \define{small}. We now move on the the second thing we need to show.
         
         Next, note that for any cyclic group in $A\leq H$, we have that if $[A':A]\leq 2$ then $A'\in \calA$. Indeed suppose this was not the case. We pass to the overgroup $\Gamma\geq H$. $A$ is elliptic in the $\Gamma$ tree $T$ if and only if $A'$ is elliptic. By the structure of trivial suspensions, if $A$ is elliptic in $T$ then $A'$ must also be cyclic,. If $A$ is hyperbolic in $T$ then $A'$ must also fix the axis of $A$. We have $A' = \bk{A,g}$ for some element $g\in \Gamma$. If $A'$ is not itself cyclic then, by the classification of discrete isometries of a line, $g$ must fix a point and $g^2$ must fix the entire axis. This contradicts acylindricity since $g^2\neq 1.$
         
         Finally, since any equivalence class stabilizer $H_{[A]}, A \in \calA$, is either cyclic or lies inside a trivial suspension, we have that every group in $\calS_{\mathrm{nvc}}$, the set of equivalence class stabilizers that are not virtually cyclic, is contained in a group in $\calH$. The claim now immediately follows from \cite[Corollary 9.1]{guirardel_jsj_2017}.
         
         Now that we have proven the existence of a canonical $(\calA,\calH)$ tree in our novel JSJ theory, it remains to describe it. We will do so by first constructing a $(\calA,\calH)$-JSJ tree (recall Section \ref{sec:G-trees} for the precise definition) from $T_H$.
         
         Consider the tree $\dot T_H$ obtained by $H$-equivariantly collapsing every thick and every thin component to a point. In particular there is an $H$-equivariant map $T_H \onto \dot T_H$ and the only edges that are not collapsed to points are the interstitial edges. Now $T_H$ is an $(\calA,\calH)$-tree and it has vertices that are either piecewise trivial suspensions that we call \emph{$\calH$-vertices}, or acylindrical graphs of free groups with cyclic edge groups, that we call \emph{thin type}. Let $w$ be a thin-type vertex of $\dot T_H$ and denote by $\calE_w$ the images of the incident edge groups. On the one hand, since $H$ is one ended, $H_w$ is also one ended relative to $\calE_w$. On the other hand $H_w$, being the fundamental group of a graph of free groups without any Baumslag-Solitar subgroups is torsion-free hyperbolic by \cite{bestvina_combination_1992}, so
         \cite[Theorem 9.5]{guirardel_jsj_2017} and immediately tells us that that $H_w$ admits a $(\calA,\calE_w)$-JSJ tree $T_w$, that the collapsed tree of cylinders coincides with the tree of cylinders $(T_w)_c$ and that $T_w$ is in the same deformation space as $(T_w)_c$. Let $\hat T_H$ be the tree obtained by blowing up every thin type vertex $w$ of $\dot T_H$ to the tree $T_w$ as follows: $H$-equivariantly delete $w$ from $\dot T_H$. Now every edge $e$ incident to $w$ in $\dot T_H$ has a stabilizer $H_e$ that fixes a vertex $w_e$ of $T_w$, so we $H$-equivariantly attach $e$ to $w_e$. 
        
        The next three claims are verifications of facts that are unsurprising but still need to be shown since we are dealing with a novel JSJ theory.
          
         \textbf{Claim:} \emph{Every edge group in $\hat T_H$ is $(\calA,\calH)$-universally elliptic.} Indeed, first note that every interstitial edge has stabilizer lying in a subgroup of $\calH$. Let $e$ be some non-interstitial edge. We need to show that $H_e$ acts elliptically on an arbitrary $(\calA,\calH)$-tree $S$. By construction $e$ must lie in some subtree $T_w$ where $w \in \verts{\dot T_H}$ is a thin-type vertex. There is an induced action of $H_w$ on $S$ with abelian edge stabilizers and where the subgroups in $\calE_w$ are elliptic, this $S$, viewed as an $H_w$-tree, is a $(\calA,\calE_w)$-tree. Since $H_e \leq H_w$ is an edge group of $T_w$ and $T_w$ is an $(\calA,\calE_w)$-JSJ tree, $H_e$ must act elliptically on $S$. This proves the claim.
        
        \textbf{Claim:} \textit{$\hat T_H$ is an $(\calA,\calH)$-JSJ tree}. We've already shown that $\hat T_H$ is a universally elliptic $(\calA,\calH)$-tree. Now we must show it dominates every other universally elliptic $(\calA,\calH)$-tree. Suppose towards a contradiction that this was not the case, then there would exist some universally elliptic $(\calA,\calH)$-tree $S$ with a single edge orbit that is not dominated by $\hat T_H$. This implies that some vertex group $H_u$ of $\hat T$ must acts hyperbolically on $S$. This vertex group cannot by thick-type so by universal ellipticity of $\hat T_H$,  $H_u$ must be flexible-type vertex group of some subtree $T_w\subset \hat T_H$ where  $w$ is a thin-type vertex of $\hat T$. By \cite[Theorem 9.14]{guirardel_jsj_2017} $H_u$ must be a QH subgroup in $\hat\calT$, the incident edge groups of $H_u$ must be elliptic. Thus, up to conjugacy, any edge group in $S$ is commensurable to the $\pi_1$-image of some simple closed curve $\gamma$ in the underlying surface $\Sigma_u$ with boundary. In particular, by taking another simple closed curve $\delta$ that essentially intersects $\gamma$, it is possible to produce a $(\calA,\calH)$-splitting of $H$ in which this edge stabilizer of $S$ is hyperbolic, which is a contradiction. The claim now follows.
            
        It now follows from \cite[Corollary 9.1]{guirardel_jsj_2017} that $\cal T$ and $\hat T_H$ have the same vertex stabilizers and by \cite[Theorem 9.14]{guirardel_jsj_2017} that the vertex stabilizers are as in the statement of the proposition since the stabilizers of thin-type vertices of $\dot T_H$ are torsion-free hyperbolic. $\cal T$ can now be obtained from $\hat T_H$ as a tree of cylinders. There is a natural mapping from $\verts{T_H}$ to $\verts\calT$, that is injective on the set of vertices with non-abelian stabilizer, but that may identify cyclic edge groups that lie in the same cylinder. We note that this may ``merge'' subtrees coming from distinct thin components of $T_H$.
        
        \textbf{Claim:} \emph{Suppose that an $H$-tree $S$ is yet another $(\calA,\calH)$-tree whose non-abelian vertices are as described in (\ref{it:max-pts}) and (\ref{it:free}) in the statement of the proposition, then $S$ is in the same deformation space as $\calT$.} In this case, consider a connected component $S'$ of $S\setminus IE$ where $IE$ is the set of edges adjacent to vertex groups that lie in $\calH$. Denote by $H_{S'}$ the stabilizer of the subtree $S'$. As an $H_{S'}$-tree, $S'$ is an $(\calA,\calE')$-tree, where $\calE'$ is the set of subgroups that lie in the images of edges groups coming form edges in $IE$, or equivalently, are adjacent to vertex groups in $\calH$. As for $H_w$ above, $H_{S'}$ must be one ended relative to $\calE_{S'}$ and $H_{S'}$ must be torsion-free hyperbolic. Since the vertex groups of $H'_{S'}$ in this $(\calA,\calE')$ splitting are either rigid or maximal QH, we have that, by \cite[Theorem 3.22]{dahmani_deciding_2019}, $S'$ must be a $(\calA,\calE')$-JSJ tree for $H_{S'}$. This claim now follows from the two previous claims, but where $S$ is used instead of $\hat T_H$. This works because the subtree $S'$ of $S$ and the subtree $T_w$ of $\hat T_H$ above share the same essential properties.
        
        Finally \cite[Corollary 9.1]{guirardel_jsj_2017} states that $\cal T$ is invariant under automorphisms of $H$ that fix $\calA$ and $\calH$, since all automorphisms of $H$ preserve these two classes we have a surjection $\delta\aut{\bbX_c} \onto \out H$, where $\bbX$ is the graph of groups splitting dual to $\calT$.
        \end{proof}

       The previous proposition proves that there exists a suitable decomposition for our subgroup $H$ of $\Gamma$. The next statement covers the computability of this decomposition and also gives a refinement of this decomposition that we will use in Section \ref{sec;EMP}. The corollary below has some extra hypotheses (solution to word problem, effective coherence) that will be proven in Section \ref{sec;algo_trac}.    

    \begin{cor}\label{coro;hybrid_canonical_for_sbgp}
        Let $G$ be a finitely generated free group. 
        Suppose 
        that $\Gamma = G \semidirect\alpha \bk t$ is a globally fibered piecewise trivial suspension over $T$ with fiber $G$, where $G$ has decidable word problem and $\Gamma$ is effectively coherent. 
        
        Then there is an algorithm that, given  $H \leq \Gamma$ be a finitely generated one-ended subgroup, constructs the splitting $\bbX$ of $H$ dual to the tree $\calT$ given in Proposition \ref{prop;canonical_for_subgroups}. Furthermore, the maximal thick vertex groups can be algorithmically refined to give a graph of groups $\hat\bbX$ whose vertex groups are either trivial suspension, i.e. isomorphic direct products of subgroups of $G$ with $\bbZ$, or to finitely generated subgroups of $G$. The edge groups are either isomorphic to $\bbZ^2$ or $\bbZ$, and if a $\bbZ$-edge group maps into a trivial suspension, then its image is in the centre of that suspension. Furthermore, we still have a surjection $\delta\aut{\hat\bbX} \onto \out H$.

    \end{cor}    

    \begin{proof}
      The idea is along the same lines as the proof of Proposition \ref{prop;struct_tree_compute}. Given a generating set $S \subset \Gamma$ for $H$, we can use the effective coherence of $\Gamma$ to construct a presentation $\bk{X|R}$ for $H$. Next, we enumerate presentations of $H$. By Proposition \ref{prop;canonical_for_subgroups} one of these  presentations of $H$ is  explicitly that of  a graph of groups  isomorphic to $\bbX$. For the algorithm constructing $\bbX$ to terminate we need to recognize the following conditions:
      \begin{enumerate}
          \item\label{bipartite} \textit{The underlying graph should be bipartite, black vertices are maximal cyclic in $H$.} Verifying that the underlying graph is bipartite is obvious, verifying that vertex groups are cyclic follows from effective coherence and a solution to the word problem. Verifying that a vertex group is maximal cyclic in $H$ can be done by computing centralizers in vertex groups.
          \item\label{vert_gps}\textit{White vertex groups are as given in Proposition \ref{prop;canonical_for_subgroups}.} This falls into two subcases:
          \begin{enumerate}
              \item \textit{Piecewise trivial suspensions.} This can be verified by finding a presentation of the vertex groups of a bipartite graph of groups whose red vertex groups are isomorphic to $\bbZ^2$ and whose green vertices are direct products of subgroups of $G$ with $\bbZ$, which again can be seen from the right kind of presentation.
              \item \textit{Free.} Because we can find presentations of the vertex groups from the presentation of $H$, because the word problem is decidable in these vertex groups, and because maximal abelian subgroups are malnormal in these vertex groups, we can apply \cite[Theorem B]{touikan_detecting_2018} to certify whether a thin vertex group is rigid relative to incident edge groups. We can verify that a subgroup is QH by finding a presentation of that subgroup that is a standard presentation of the fundamental group of a with boundary, that the incident edge groups are conjugate to the boundary subgroups. We can verify that such groups are maximal QH by ensuring that there are no black vertices of valence 2 that are adjacent to two QH groups such that collapsing the edges gives a larger QH group. 
          \end{enumerate}
          \item\label{cylII}\textit{If two edges $e,e'$ terminate in a vertex $w$ then the images $\tau_e(\bbX_e)$ in $\bbX_w$ intersects all conjugates of the image of $\tau_{e'}(\bbX{e'})$ trivially.} This again can be verified by computing centralizers in vertex groups and by using hereditary algorithmic tractability. In particular, we need only solve the conjugacy problem either in subgroups of $G$ or in subgroups of $G\times\bbZ$. The latter reduces to the conjugacy problem in subgroups of $G$.
      \end{enumerate}

      Proposition \ref{prop;canonical_for_subgroups} already implies that if $\bbX$ satisfies \ref{vert_gps} then $\bbX$ is dual to a JSJ tree. Conditions \ref{bipartite} and \ref{cylII} imply that the tree dual to $\bbX$ is equal to its tree of cylinders. It follows that there is an algorithm that will terminate and output the splitting $\bbX$ dual to $\calT$.
      
      Finally a canonical graph of groups $\hat\bbX$ as in the statement of the Proposition can be constructed as follows. For every vertex $v\in \verts X$ whose vertex group $\bbX_v$ is a piecewise trivial suspension, blow up $v$ to the canonical splitting $\bbY^v$ that can be computed by Proposition \ref{prop;struct_tree_compute}. We must now reattach the the edges originally incident to $v$ to vertices of $Y^v$. There is a subtlety here: such an incident edge group will actually lie in multiple vertex groups of $\bbY^v$. By Proposition \ref{prop;canonical_for_subgroups} item \ref{it:interstitial-central}, the edge group associated to every such edge must coincide with the the centre of exactly one of the non-abelian vertex groups of $\bbY^v$. This thus gives us a canonical choice of which vertex group of $\bbY^v$ we want to send the edge group.
      
     It is clear from Propositions \ref{prop;canonical_for_subgroups} and \ref{prop;delta-aut-surjecte-sur-out} that $\hat\bbX$ is still $\out H$-invariant, so the last claim follows. 
    \end{proof}
    
    We will call such a decomposition, \emph{canonical trivially partially suspended}.

        \subsection{Fibre and orientation preserving automorphisms of suspensions}\label{sec;UL_IP_aut}

        We are interested in the fibre and orientation-preserving automorphisms of suspensions. We begin by drawing a connection between these automorphism groups and centralizers of outer automorphisms.
        \subsubsection{Centralizers and $\outfo{\torus \alpha}$ }

Let $F$ be a group, and $\alpha$ an automorphism. The subgroup $\autfo{F\rtimes_\alpha \bk t}$ of $\aut{F\rtimes_\alpha \bk t}$ is the subgroup of all automorphisms that preserve the fibre and the orientation of the semi-direct product, i.e. those that preserve $F$ and send  $t$ in $tF$.  The subscript $\mathrm{fo}$ stands for fibre and orientation.

All inner automorphisms are in  $\autfo{F\rtimes_\alpha \bk t}$. The quotient by this normal subgroup of inner automorphisms is the group  $\outfo{F\rtimes_\alpha \bk t}$.   

As a general observation, one can note that the image of $\autfo{F\rtimes_\alpha \bk t}$ in $\aut{F}$ by the restriction map composed by the quotient map $\aut F \to \out F$, is the centralizer of $[\alpha]$  in $\out F$.  Also, the image of $\outfo{F\rtimes_\alpha \bk t}$ in $\out F$ by the restriction map is $Z_{\out F} ([\alpha])/\bk{[\alpha]}$, for   $Z_{\out \bbF} ([\alpha])$ the centralizer in $\out \bbF$ of the outer class of $\alpha$.

\begin{prop} Let $\torus \alpha = F \rtimes_\alpha \langle t \rangle$. Then 
	$\Outfo {\torus \alpha}  \simeq Z_{\out \bbF} ([\alpha])/ \langle [\alpha]\rangle $. 
\end{prop}

\begin{proof}
	There is a natural homomorphism $\Autfo {\torus \alpha}  \to \Aut F$, by restriction, and hence by composition, one has the homomorphism $\Autfo   {\torus \alpha} \to \out F$. 
	
	If $\phi$ is in $\Autfo   {\torus \alpha} $, one has  $\phi(t) = f_0t$ for some $f_0\in F$.  For $f\in F$,   one has 
	\[\phi(f^t) = \phi (\alpha(f)) = \phi(f)^{\phi(t)} = \alpha( \phi(f)^{f_0} ) ={\alpha( \phi(f))^{\alpha(f_0)}}.\] 
	
	It follows that the image of $\phi$ in $\out F $ commutes with $\alpha$, so one can restrict the target of the previous homomorphism, and obtain the homomorphism $\Autfo {\torus \alpha}  \to Z_{\out F}([\alpha])$.  Conversely, any automorphism of $F$ commuting with $\alpha$ allows us to define an automorphism of the semidirect product, hence this later homomorphism is surjective.

	Also by composition, one has the homomorphism $$\Autfo  {\torus \alpha}  \to Z_{\out F}([\alpha]) / \langle [\alpha] \rangle.$$  
	
	Observe that any inner automorphism of $\torus \alpha$ is in the kernel of this homomorphism (it is obvious for inner automorphism by conjugation by $t$, and by any $f\in F$, hence by any product $ft^k$).
	
	Thus, there is a natural surjective homomorphism  $$\Outfo {\torus \alpha}\to Z_{ \out F}([\alpha])/\langle  [\alpha]  \rangle.$$

	Now consider an element  $\Phi$ in the kernel of this homomorphism. Realize $\Phi$ as an automorphism $\phi$. In restriction to $F$, it is in the group generated by $\alpha$ and the inner automorphisms of $F$. Therefore, its restriction to $F$  is that  of an inner automorphism of    $\torus \alpha $. Up to choosing $\phi$ suitably in the class of $\Phi$, we may thus assume that $\phi|_F= Id_{F}$. But then  again, $\phi(t) = t f_0$, and $  \alpha(f) = f^t =  \phi(f^t) = \phi(f)^{\phi(t)} = \alpha(f)^{f_0}$. Thus, $f_0= 1$, and $\phi$ is the identity. Thus, the  homomorphism  $\Outfo {\torus \alpha}\to Z_{ \out F}([\alpha])/\langle  [\alpha]  \rangle$ is injective, hence an isomorphism.
	
\end{proof}

        \subsubsection{Fiber and orientation preservation in unipotent linearly growing  suspensions}

        We will now specialize to the case of a suspension $\Gamma = \torus\alpha = F \semidirect\alpha\bk t$ where $F$ is free and $\alpha \in \aut F$ is a unipotent linear automorphism. The data $\Gamma = F \semidirect\alpha\bk t$ encodes a homomorphism $\Gamma \onto \Gamma/F \simeq \bbZ$ as well as a distinguished generator $tF$ of the $\bbZ$-quotient. We call the free normal subgroup $F$ a \define{fibre} and the coset $tF$ an \define{orientation}. Fibres and orientations are not unique to $\Gamma$ and we may consider other fibrations $\Gamma = F' \semidirect{\alpha'}\bk{t'}$ which correspond to different homomorphisms to $\bbZ$ or different distinguished generators of the quotient.
 
        Since $\Gamma$ is a piecewise trivial suspension it admits a canonical graphs of groups decomposition $\bbX$, by Proposition \ref{prop:structural_tree_is_canonical}. $X$ the graph underlying $\bbX$ is bipartite and we colour the vertices as follows: a vertex $b$ is \define{black} if the corresponding vertex group $\bbX_b \simeq \bbZ^2$ and a vertex $w$ is \define{white} if $\bbX_w \simeq F_w\times \bbZ$, where $F_w$ is a free group.

        We note that the centre of every white vertex group is characteristic (invariant by automorphisms) and isomorphic to $\bbZ$, if $\Gamma$ is given an orientation $tF$, then for each $\bbX_w$ we will pick an element $t_w \in tF \cap \bbX_w$ that generates the centre of $\bbX_w$. We note that, although without fixing a basepoint or spanning tree of $\bbX$, vertex groups only define subgroups up to conjugacy, the choice of orientation of $\Gamma$ still uniquely determines a distinguished central element of $\bbX_w$. If we are also explicit about the fibre $F \triangleleft \Gamma$ then we can write $\bbX_w = F_w\times\bk{t_w}$ where $F_w = \bbX_w \cap F$. This is well-defined because $F$ is a normal subgroup.
 
        We start by examining how fibres coming from unipotent linear automorphisms intersect the black vertex groups.
 
        \begin{lem}\label{lem:black-fibers}
            Suppose that $\Gamma=F\rtimes_\alpha \bk{t} = F'\rtimes_{\alpha'}\bk{t'}$ where both $\alpha$ and $\alpha'$ are unipotent and linear, and let $\bbX$ be the canonical graph of groups decomposition of $\Gamma$ as a piecewise trivial suspension. Then for every black vertex in $\bbX$ one has
            \[
            \bbX_b \cap F = \bbX_b \cap F'.
            \] 
            In particular, black vertex groups have canonical fibres.
        \end{lem}
        
        \begin{proof}
            Let $\bbX_b$ be a black vertex group and let $e,f$ be two distinct edges incident to $b$. Let $w,u$ be the white vertices adjacent to $e,f$ respectively. For $v=u,w$ we have $\bbX_v \simeq F_v\oplus \bbZ$ where $F_v = \bbX_v\cap F$ and $F_w' = \bbX_w\cap F'$, and in particular that $F'_w \simeq F_w$. Now although the fibres may not match up in a white vertex group, the couple $\{t_w^{\pm 1}\}$ is canonical: it is the symmetric generating couple of the centre of the vertex group. Let $t_w$ be the oriented central element corresponding to an orientation of $\Gamma$ with fibre $F$ and let $t'_w = t_w^{\pm 1}$ be the oriented central element corresponding to an orientation of $\Gamma$ with fibre $F'$.
    
            The element $t_w$ maps to a generator of either quotient $\bbX_w/(\bbX_w\cap F)$ or $\bbX_w/(\bbX_w\cap F')$. Moreover $t_w$ lies in the image $i_e(\bbX_e)$ so we can pick the unique pre-image $t_e$ such that $i_e(t_e)=t_w$.
            If we take the oriented central element $t_u \in \bbX_u$ for the same orientation as we took for $\Gamma/F$ and define $t_f \in \bbX_b$ to be the pre-image such that $i_f(t_f)=t_u$, then by looking at the quotient $\Gamma \to \Gamma/F$ we have
            \[
            t_e t_f^{-1}=c_{ef} \in \bbX_b\cap F
            \] 
            and restricting the quotient to the abelian subgroup $\bbX_b$ we see that $\bbX_b \cap F$ is the cyclic direct summand of $\bbX_b$ containing $c_{ef}$. Furthermore we note that $c_{ef}\neq 1$, for otherwise we would have $t_e=t_f$ which would imply that the element $t_u \in \bbX_u$ also centralizes $\bbX_w$ so that $t_u$ pointwise fixes an infinite subtree of the Bass-Serre tree dual to $\bbX$, contradicting acylindricity guaranteed by Proposition \ref{prop;PTS_are_acyl}.

            Now if we were to consider the orientation for $\Gamma/F'$ then we would have corresponding elements $t'_w, t'_e,t'_u,t'_f$ and we have $t'_w = t_w^{\pm 1}$
            if and only if $t'_e = t_e^{\pm 1}, t'_u = t_u^{\pm 1}, t'_f = t_f^{\pm 1}$. So that 
            \[ t_e' {t_f}^{-1}{}' = {c_{ef}}^{\pm 1} \in \bbX_b \cap F'.
            \] 
            It therefore follows that $\bbX_b\cap F$ and $\bbX_b\cap F'$ have a common non-trivial element and, since these subgroups are codimension 1 direct summands of $\bbZ^2$, they coincide.
        \end{proof}

    We, therefore, have the peculiar situation that although $\torus\alpha$ could admit multiple unipotent linear fibrations, all fibres must intersect the black vertex groups the same way, yet unlike in the white vertex groups, there is no canonical pair of central elements. We now work towards matching up fibres in white vertex groups.
    
    \begin{lem}[Fibre matching vertex group automorphism]\label{lem;align-aut}
        Let $\Gamma$ be a trivial suspension of a free group with centralizer $\bk{t}$ and let $F,F'\leq \Gamma$ be two free fibres, i.e. $F \simeq F'$ are free groups and
            \[
            F'\oplus\bk t \simeq \bk{F',t}=\Gamma=\bk{F,t} \simeq F\oplus\bk t.
            \] 
        Then there is an automorphism $\eta_{F',F}\in \aut \Gamma$ such that $\eta_{F',F}(F')=F$ and such that for all $g \in F \cap F'$, $\eta_{F',F}(g)=g$. 
    \end{lem}
        
    \begin{proof}
        Let $d:F\oplus\bk t \to \Gamma=\bk{F,t}$ be the isomorphism given by $(f,t^n) \mapsto ft^n$ and let $d':F'\oplus \bk t \to \Gamma$ be defined analogously.            Let $\rho_F:\Gamma \onto F$ be the projection onto the fibre $F\leq \Gamma$ given by $d\circ p_F\circ d^{-1}$ where $p_F:F\oplus\bk t \to F \oplus\bk t$ is the projection $(f,t^n) \mapsto (f,1)$. Let $\kappa: F'\oplus\bk t \to F\oplus \bk t$ be the isomorphism given by
            \[
            (f',t^n) \mapsto (\rho_F|_{F'}(f'),t^n),
            \] 
        where $F'$ is viewed as $F'\leq \Gamma$ and $\rho_F|_{F'}:F' \to \Gamma$ is the restriction of the projection onto $F$. Then the composition of isomorphisms $\eta_{F',F}=d\circ\kappa\circ(d')^{-1}$ has the desired properties.
    \end{proof}

    \begin{prop}\label{prop;full-fiber-align}
        Suppose that $\G=F\rtimes_\alpha \bk{t} = F'\rtimes_{\alpha'}\bk{t'}$ where both $\alpha$ and $\alpha'$ are unipotent and linear, and let $\bbX$ be the canonical graph of groups decomposition of $\G$ as a piecewise trivial suspension. Then there is an automorphism of $\Psi\in \delta_0\aut \bbX \leq \aut \G$ that restricts to an isomorphism $\Psi|_F:F' \to F$.
    \end{prop}

    \begin{proof}
        We first note that $F$ and $F'$ are both normal subgroups of $\G$, therefore the notion of an intersection with a vertex group $\bbX_v,v \in \verts X$ is well-defined irrespective of the choice of basepoint in $\bbX$. By Lemma \ref{lem:black-fibers}, for any black vertex $b \in \verts X$, $F \cap \bbX_b = F' \cap \bbX_b$.
    
        From this it follows that in every white vertex $w$, for every edge $e$ joining $w$ to some black vertex $b$ we have $F\cap i_e(\bbX_b)= F'\cap i_e(\bbX_b)  = \bk{g_e} \leq \bbX_w$. Since the images of the edge group is generated as $i_e(\bbX_b) = \bk{g_e,t_v}$, the matching isomorphism $\phi_w=\eta_{\bbX_w\cap F,\bbX_w\cap F'}:\bbX_w \stackrel\sim\to \bbX_w$ given Lemma \ref{lem;align-aut} fixes a generating of $i_e(\bbX_b)$ pointwise and thus restricts to the identity on $i_e(\bbX_b)$. Thus, each matching automorphism extends to a graph of groups automorphism of $\bbX$. Applying them all in parallel gives a graph of groups automorphism $\Phi \in \delta_0\aut\bbX\leq \aut \G$ given by\[
            \Phi=(Id,(\phi_v)_v, (Id_e)_e, (1)_e),
        \] where $\phi_v$ is as given above for white vertex and is the identity on black vertex groups. It follows that $\Phi(F')\cap \bbX_v = F \cap \bbX_v$, when $v$ is a white vertex group. If $v$ is a black vertex group, then the equality follows from Lemma \ref{lem:black-fibers}. 
        We still cannot yet conclude that $\Phi(F')=F$ as we haven't fully considered what happens "outside" the vertex groups.
    
        We now fix a basepoint $x_0 \in \verts X$ and identify $\G = \pi_1(\bbX,x_0)$. We have our target fibre $F \leq \G$ and we denote the image of the fibre $F'' = \eta(F')$. Let $\tau \subset X$ be a spanning tree and let $G_\tau\leq \G$ be the subgroup of $\G$ that is the iterated amalgamated product along the edges groups with edges lying in $\tau$. Because $F'' \cap \bbX_v = F \cap \bbX_v$ for every $v \in \verts X$, and $G_\tau$ is an iterated fibered coproduct, in passing to the category of abelian groups we see that $F\cap G_\tau$ and $F''\cap G_\tau$ have the same  image $\bar F_A$ in the abelianization  $A=G_\tau/[G_\tau,G_\tau]$.
    
        Now $G = \bk{G_\tau,e_1,\ldots,e_{-\chi(X)+1}}$ where the edges $e_i$ play the role of stable letters realizing $\G$ as multiple HNN extension of $G_\tau$ as is the case in the construction of $\pi_1(\bbX,\tau)$ given in Section \ref{sec;vocab}. Treating these edges as group elements also makes sense given that $G$ is realized as a quotient of the Bass group $\Bass(\bbX)$ given in \eqref{eqn:bass}. We think of the graph $X/\tau$, obtained by collapsing the subtree $\tau$ to a vertex as the graph underlying this multiple HNN extension. We also note (recall the construction in Theorem \ref{thm;unipotent-linear-gdt}) that $F''$ gets an induced splitting as an iterated HNN extension $\bk{F''\cap G_\tau,r_1,\ldots,r_{-\chi(X)+1}}$ with the same underlying graph $X/\tau$. In particular we have that $r_i = e_ig_i$ for some $g_i, \in G_\tau$.
    
        Consider the abelianization $\G_{ab} = \G/[\G,\G]$. By functoriality, the inclusion $G_{\tau}\leq \G$ gives rise to a canonical (possibly non-injective) map $i:A\to \G_{ab}$. The subgroup generated by stable letters $\bk{r_1,\ldots,r_{-\chi(X)+1}} \leq F''$ maps to a direct factor of $\G_{ab}$ with rank $-\chi(X)+1$.
    
        The group $F$ is precisely the preimage in $\G$ of $\ker(f)=\bar F$ of some mapping $f:\G_{ab} \onto \mathbb Z$. The image $\bar F''$ of $F''$ in $\G_{ab}$ is generated by $i(\bar F_A)\leq \ker(f)$ and by the images of the stable letters $\bar{r_i}$. One has $F'' \neq F$ if and only if for some stable letter image, one has $f(\bar{r_i})\neq 0$. For every stable letter $e_i$, the corresponding edge group $\bbX_{e_i}$ contains a central element $t_i$ such that $f(\bar t_i)=1$, where $\bar t_i$ is the image of $t_i$ in $\G_{ab}$. Thus setting $n_i = f(\bar{r_i}), i=1,\ldots,-\chi(X)+1$ and applying the generalized Dehn twist specified by the $\edges{X}$-substitution
        \[
            e_i \mapsto e_i{t_i}^{-n_i}, e_i \in \edges{X\setminus\tau}
        \] 
        which can be written as $\Delta \in \delta_0\aut\bbX$ given by\[
            \Delta = (Id,(Id_v)_v,(Id_e)_e, (\gamma_e)_e),
        \] where $\gamma_e=1$ if $e\in \edges \tau$ and $\gamma_e = {t_i}^{-n_i}$ if $e=e_i \in \edges{X\setminus \tau}$, gives us $\Delta(F'')$ which which vanishes in the map $\G\onto \G/F \simeq \bbZ$ and we therefore have $\Delta(F'')=F$. The result now follows by taking $\Psi=\Delta\circ\Phi.$
        \end{proof}

        From this proposition we immediately have:

        \begin{cor}[Unipotent linear monofibration]\label{cor;ULMF}
          
            Let $F, F'$ be two finitely generated free groups
            and $\phi \in \aut{F}$ and $\psi \in \aut{F'}$ be unipotent linearly growing automorphisms.  
            The corresponding suspensions are abstractly isomorphic, i.e. 
            \[
            F\rtimes_\phi \bbZ \simeq F' \rtimes_\psi \bbZ
            \] 
            if and only $F$ and $F'$ have same rank, and for an isomorphism  $\iota: F\to F'$,   $\iota^{-1}\psi \iota$ and $\phi^{\pm 1}$ have conjugate images in $\out{F}$.
        \end{cor}
        
        It is worth noting that by a result of Button \cite[Theorem 3.4]{button_mapping_2007}, for all such groups, there are infinitely many other ways they can be given as suspensions of non-unipotent automorphisms with fibre ranks going to infinity.

        \subsubsection{Detection of isomorphisms preserving the fibre in vertex groups and graphs of groups}
        
        \begin{prop}\label{prop:iso-problem-for-suspensions}
            Let $\bbX,\bbX'$ be two canonical graph of groups decompositions of globally fibered piecewise trivial suspensions $\G=F\semidirect\alpha\bk t, \G'=F'\semidirect{\alpha'}\bk{t'}$ both with the same underlying graph $X$. Suppose furthermore that we have fixed an orientation for each. There is an algorithm which decides whether there is an isomorphism $\bbX\stackrel{\sim}{\to}\bbX'$ that restricts to isomorphisms
            \[
            \bbX_v \stackrel\sim\to \bbX_v'; v \in \verts X\qquad \bbX_e \stackrel\sim\to \bbX_e'; e \in \edges X,
            \] which preserves the orientation on the white vertex groups.
        \end{prop}
        
        Only for the proof below, since we are following \cite[\S 4]{dahmani_deciding_2019}, we will call the images of edge groups in vertex groups \define{peripheral subgroups} of the vertex groups. An \define{unmarked ordered peripheral structure} is a tuple whose entries form a set of representatives of conjugacy classes of images of edge groups in vertex groups and a \define{marked ordered peripheral structure} is a tuple of tuples of generators of the images of edge groups in vertex groups.
            
        \begin{proof}
            Our algorithm consists of three steps. Each step will introduce a level of structure and will end with a test. It will be evident that the desired isomorphism of graphs of groups will exist if and only if all three tests are passed.
            
            \noindent \textbf{Step 1 Isomorphism of vertex groups.} The hypotheses put us in the setting of \cite[\S 4]{dahmani_deciding_2019}. The first test is straightforward.

            \noindent \textbf{Test 1.} Are the vertex groups $\bbX_v$ and $\bbX'_v$ isomorphic for each vertex $v \in \verts X$? Otherwise, we can conclude that no suitable isomorphism of graphs of groups exists.
            
            \noindent \textbf{Step 2 Extensions adjustments and matching peripheral subgroups in white vertex groups.} Having passed the first test, we have that $X$ is a bipartite graph and we can assume that white vertices have vertex groups that are direct products of non-abelian free groups and cyclic groups and that the black vertices are isomorphic to $\bbZ^2$. For each $v \in \verts X$ we can find a fibre and orientation preserving isomorphism $\varphi_v:\bbX_v \to \bbX'_v$. Let $b,w \in \verts X$ and let $e \in \edges X$ join these vertices. Consider the following diagram:
            \begin{equation}\label{eq:ext-adj-square}
                \begin{tikzpicture}[xscale=2]
                    \node (xb) at (0,0) {$\bbX_b$};
                    \node (xpb) at (0,-1) {$\bbX_b$'};
                    \node (xe) at (1,0) {$\bbX_e$};
                    \node (xpe) at (1,-1) {$\bbX_e$'};
                    \node (xw) at (2,0) {$\bbX_w$};
                    \node (xpw) at (2,-1) {$\bbX_w$'};
                    \draw[right hook ->] (xe) --node[above]{$i_{e,w}$} (xw);
                    \draw[left hook ->] (xe) --node[above]{$i_{e,b}$} (xb);
                    \draw[right hook ->] (xpe) --node[above]{$i_{e',w'}$} (xpw);
                    \draw[left hook ->] (xpe) --node[above]{$i_{e',b'}$} (xpb);
                    \draw[->] (xw) --node[right]{$\varphi_w$} (xpw);
                    \draw[->] (xb) --node[left]{$\varphi_b$} (xpb);
                \end{tikzpicture}
            \end{equation}
            where the horizontal monomorphisms are the edge group monomorphisms and the vertical maps are fibre and orientation preserving isomorphisms between vertex groups. By \cite[Proposition 4.4]{dahmani_deciding_2019}, the graphs of groups $\bbX$ and $\bbX'$ are isomorphic if and only if for each $v \in \verts X$ and each $e \in \edges X$ incident to $v$ there is an automorphism $\alpha_v \in \aut{\bbX'_v}$ and elements $g_{e,v} \in \bbX'_v$ such that for every $\gamma \in \bbX_e$ we have
            \begin{equation}\label{eq:ext-adj-eqn}
                i_{e',b'}^{-1}\circ\ad{g_{e,b}}\circ\alpha_b\circ\varphi_b\circ i_{e,b}(\gamma)= i_{e',w'}^{-1}\circ\ad{g_{e,w}}\circ\alpha_w\circ\varphi_w\circ i_{e,2}(\gamma).
            \end{equation}
            This equation should be interpreted as using the automorphisms $\alpha_b,\alpha_w$ and the conjugations $\ad{g_{e,b}},\ad{g_{e,w}}$ to make the diagram \eqref{eq:ext-adj-square} commutative. Now the horizontal maps in\eqref{eq:ext-adj-square} are not bijective so they do not have inverses, however, because they are injective, they are bijective onto their ranges. Therefore for \eqref{eq:ext-adj-eqn} to make sense we must be able to match up the peripheral subgroups. Since black vertex groups coincide with the images of incident edge groups, we must first only focus on white vertex groups
            
            \noindent \textbf{Test 2.} Decide for each white vertex $w\in\verts X$ whether there exists a fibre preserving automorphism $\alpha_w \in \aut{\bbX'_v}$ such that for each $e \in \edges X$ incident to $w$, with $b$ being the other vertex incident to $e$, there is a conjugating element $g_{e,w}$ that\[
                \ad{g_{e,w}} \circ \alpha_w(\varphi_w(i_{e,w}(\bbX_e))=i_{e',w'}(\bbX_e'),
            \] 
            where the various mappings are those given in \eqref{eq:ext-adj-square}. By our colouring convention $\bbX_w = F_w \oplus \bk{t_w}$. Then $i_{e,w}(\bbX_e) = \bk{c_{e,w},t_w}$ where $c_{e,w}$ lies in the fibre and $\varphi_w(i_{e,w}(\bbX_e)) = \bk{\varphi_w(c_{e,q}),t_w'}$, where $t_w'$ is a generator of the centralizer of $\bbX_w' = F_w' \oplus\bk{t_w'}$. The second test therefore amounts to finding an automorphism $\alpha_w$ such that each edge $e$ incident to $w$ the element $\alpha_w(\varphi_w(c_{e,w})$ is conjugate to $c_{e',w'}^{\pm 1}$, where $i_{e',q'}(\bbX_e') = \bk{c_{e',w'},t_w'}$ and $c_{e',w'}$ lies in the fibre. But this is precisely what the classical Whitehead algorithm for tuples of elements in $F_w'$ decides.
            
            So either the Whitehead algorithm gives us an automorphism of the fibre, which gives us a fibre preserving automorphism of $\bbX_w'$ that matches up images of edge groups, or there is no such automorphism of the fibre. If there is no such automorphism of the fibre, then we can conclude that there is no automorphism at all which matches the peripheral subgroups. This is because a non-fibre preserving automorphism of $\bbX_w'$ will still descend to an automorphism of the image\[
                \bbX_w' = F_w'\oplus\bk{t_w'} \onto F_w',
            \] since the centre is characteristic. Therefore, if $\bbX$ and $\bbX'$ are isomorphic, then there is a fibre-preserving automorphism of each vertex group that passes Test 2. Furthermore, composition with the automorphisms of $\bbX_w'$ that inverts the generator of the centre and fixes the fibre pointwise gives another automorphism that passes Test 2, so we may assume that the automorphism we found is also orientation preserving.
            
            \noindent \textbf{Step 3 Orbits of markings and matching peripheral subgroups in black vertex groups.} Passing Test 2 ensures the right hand side of \eqref{eq:ext-adj-eqn}, namely\[
            i_{e',w'}^{-1}\circ\ad{g_{e,w}}\circ\alpha_w\circ\varphi_w\circ i_{e,2}(\gamma) 
            \] is well-defined. Recall that the marking of a group is a choice of a tuple of generators. Picking $(c_e,t_w)$ as a marking of $i_{e,w}(\bbX_e) \leq \bbX_w$, as shown in diagram \eqref{eq:ext-adj-square}, gives us a pullback marking of $\bbX_e$, which we can then push forward to $\bbX_b$ to get the pair of elements
            \begin{equation}\label{eqn:ccw-pair}
                (c_e',t_e') = (\varphi_b\circ i_{e,b}\circ i_{e,w}^{-1}(c_e),\varphi_b\circ i_{e,b}\circ i_{e,w}^{-1}(t_w))\in \bbX_b'\times \bbX_b'.
            \end{equation}
            Similarly for each of the pairs $(c_e,t_w) \in \bbX_w\times \bbX_w$ given above we can push them forward to \[
                (c_e'',t_w'')=(\ad{g_{e,w}}\circ\alpha_w\circ\varphi_w(c_e),\ad{g_{e,w}}\circ\alpha_w\circ\varphi(t_w)) \in \bbX_w'\times\bbX_w'.
            \] We note that the hypothesis that $\varphi_w$ is fibre and orientation preserving gives us that $t_w'' = t_w'$, the generator of centre of $\bbX_w'$ prescribed by the orientation. Let $\calP_{uo}$ denote the unmarked ordered peripheral structure on $\bbX_w'$, then $A_w'=\autfo{\bbX_w',\calP_{uo}}$ denotes the set of fibre and orientation preserving automorphisms of $\bbX_w'$ that map each peripheral group to a conjugate of itself. 
            It is easy to see that for any $\beta \in A_w'$, after post composing by inner automorphism so that $\ad{p_{\beta,e'}}\circ\beta(c_e'')=c_e^\beta \in i_{e',w'}(\bbX_e')$, the element $c_e^\beta$ depends only on $\beta$ and, in fact, we must have $c_e^\beta = c_e''^{\pm 1}$. It follows that the orbit of the marked peripheral structure, originally marked by the pairs $(c_e'',t_w')$, has a finite orbit $\calP_m^{\beta_1},\ldots,\calP_m^{\beta_{k_w}}$ under $A_w'$. We call these the \emph{admissible marked peripheral structures of $\bbX_w'$}.
            
            Let $\vec\calP_m$ be a choice of admissible marked peripheral structures for each $\bbX_w'$ where $w$ ranges over the white vertices of $X$. For each edge $e$ incident to a white vertex $w$ there is an element $c_e^{\beta_j}\in \bbX'_w$ which is the generator of the fibre component the image of the edge group for the marking in $\vec\calP_m$. Denote\[
                (c_e^{\vec\calP_m},t_e) = (i_{e',b'}\circ i_{e',w'}^{-1}(c_e^{\beta_j}),i_{e',b'}\circ i_{e',w'}^{-1}(t_w')) \in \bbX_b'\times \bbX_b'.
            \]
            
            \textbf{Test 3.} For each collection of admissible marked peripheral structures $\vec{\calP_m}$ we do the following: in each black vertex $b$ we have, for each edge $e$ incident to $b$, 
            two pairs\[
                (c'_e,t_e') \textrm{~and~} (c_e^{\vec\calP_m},t_e).
            \] Where $(c'_e,t_e')$ is defined in \eqref{eqn:ccw-pair}. We determine if there is a collection of automorphisms $\alpha_b \in \aut{\bbX_b'}$ for each black vertex $b$ such that for every pair we have\[
            (\alpha_b(c'_e),\alpha_b(t_e'))=(c_e^{\vec\calP_m},t_e).\]
            
            This amounts to finding some $\alpha_b \in \mathrm{GL}(2,\mathbb Z)$ that sends a tuple of vectors to another tuple of vectors and is standard linear algebra over $\mathbb Z$.
            
            If we can find such admissible markings and such automorphisms of the black vertex groups, then we will have constructed orientation-preserving isomorphisms of the graph of groups of the desired form. On the other hand, the nonexistence of such markings and automorphisms implies that there are no isomorphisms $\bbX \stackrel\sim\to \bbX'$ that restrict to fibre preserving isomorphisms of the vertex groups and edge groups. By Proposition \ref{prop;full-fiber-align}, this then implies that no isomorphism $\bbX \stackrel\sim\to \bbX'$ of the desired form exists.
        \end{proof}
        
        Noting that a graph automorphism can give us two different graph of groups structures on the same underlying graph gives:
        
        \begin{cor}\label{cor:graph-aut-to-gog-aut}
            Let $\bbX$ be the canonical graph of groups decomposition of a piecewise trivial suspension $\torus\alpha$ with underlying graph $X$. Then there is an algorithm that decides if an automorphism $s:X\to X$ of graphs can be extended to fibre preserving automorphism $\Sigma =(s,(\varphi_v),(\varphi_e),(\gamma_e)) \in \delta\aut\bbX$.
        \end{cor}
        
        Recall that $\delta_0\autfo{\bbX}$ is the set of fibre and orientation preserving graph of groups automorphisms that induce a trivial automorphism of graph $X$ underlying $\bbX$. Recall that  $\delta_1\autfo{\bbX} \leq \delta_0\autfo\bbX$ is the subgroup that acts trivially on the edge groups (see  Section \ref{sec;autom_gog}).

        \begin{lem}\label{lem:delta_0-in-delta_1}
            We can compute a complete set $\{\phi_1,\ldots,\phi_k\}  \subset \delta_0\autfo\bbX$ of $\delta_1\autfo\bbX$-coset representatives.
        \end{lem}
        \begin{proof}
            Let $w \in \verts X$ be a white vertex and let $t_w$ be the generator of the centre of $\bbX_w$ given by the orientation. For each $e \in \edges X$ fix some $c_e$ such that $\bk{c_e,t_w}$ generates the the image of the edge group $\bbX_e$ in the unique incident white vertex group $\bbX_w$ and $c_e$ lies in the fibre.
            
            The restriction $\varphi_w$ of any element $\Phi=(Id,\varphi_v,\varphi_e,\gamma_e) \in \delta_0\autfo\bbX$ to a white vertex group $\bbX_w$ must map $t_w$ to itself and must send $c_e$ to the conjugate $\gamma_e^{-1} c_e^{\pm 1} \gamma_e$.
            Consider now the set $\{1,-1\}^{\edges X}$ of functions from $\edges X$ to $\{-1,1\}$. For every $\sigma \in \{1,-1\}^{\edges X}$, it is possible to determine whether there is a fibre and orientation preserving $\Psi =(Id,\psi_v,\psi_e,\delta_e) \in \delta_0\autfo\bbX$ such that in the  white vertex groups we have equality between conjugacy classes \[
                [\psi_w(c_e)] =  [c_e^{\sigma(e)}],
            \] 
            Indeed, this is precisely an instance of the classical Whitehead problem on the fibre in $\bbX_w$, which is a free group. The $\delta_e$ in $\Psi$ can then be chosen to be the appropriate conjugators.
            
            Next, determining whether such a collection of vertex group automorphisms can then be extended to an element of $\delta_0\autfo\bbX$ amounts to finding automorphisms of the black vertex groups (which are isomorphic to $\bbZ^2$) that complete the diagram \eqref{eq:ext-adj-square}.
            
            This collection of found automorphisms gives a complete set of $\delta_1\autfo\bbX$-coset representatives in $\delta_0\autfo\bbX$.
        \end{proof} 
        
        Corollary \ref{cor:graph-aut-to-gog-aut}, Lemma \ref{cor:graph-aut-to-gog-aut}, and definitions immediately imply the following.
        
        \begin{cor}\label{cor:delta_1_repr}
            We can compute a complete list of $\delta_1\autfo\bbX$-coset representatives in $\delta\autfo\bbX$.
        \end{cor}

    \section{Piecewise trivial suspensions are hereditarily algorithmically tractable}\label{sec;algo_trac}

        This section is devoted to a few uniform algorithmic properties of the class of unipotent linear suspensions or piecewise trivial suspensions.

        A finitely presented group is \emph{coherent} if any finitely generated subgroup is finitely presented. A class of group is \define{effectively coherent} if, from a presentation of the group and a finite set of elements, an algorithm computes a finite presentation of the generated subgroup.

        For example, free abelian groups are an  effectively coherent class of groups: given a finite subset of size $n$ in a free abelian group $A$, the Smith normal form of the given homomorphism $\bbZ^n\to A$ reveals a basis for the generated subgroup.
        
        Free groups are nevertheless the archetype of an effectively coherent class of groups: any subgroup of a free group is free, but the task is to compute a basis. Stallings foldings provide an  algorithmic and elegant way for this.

        In a work that deepens the foldings technique, Kapovich Weidmann and Miasnikov proved that fundamental groups of so-called benign graphs of groups are effectively coherent \cite{kapovich_foldings_2005}.

        Free-by-cyclic groups are coherent by a deep result of Feighn and Handel \cite{feighn_mapping_1999}.
        
        We will prove here that unipotent linear suspensions of finitely generated free groups are  effectively coherent (Proposition \ref{prop;eff_coh_ump_for_GDT}) by proving the benign character of their structural graph of groups decompositions.  
         
          Following \cite{touikan_detecting_2018}, we say that a class of groups is
          hereditarily algorithmic tractable, if
          the class is closed for taking finitely generated subgroups, is effectively coherent,  
          the presentations of the groups in the class being furthermore recursively enumerable, and if the class has a uniform solution to the generation problem, and to the conjugacy problem.

          In other words, there is an algorithm that given a group presentation in the class, and any finite subset of the group, produces a finite presentation of the generated subgroup, indicates if it is the whole group, and tells whether the given elements are conjugate. 
          
          This combination of algorithmic problems is useful to treat and compare decomposition in graphs of groups in which the elements of the class are vertex groups, see for instance \cite{touikan_detecting_2018, dahmani_deciding_2019, dahmani_touikan_reducing_2021}. Through Propositions \ref{prop;eff_coh}, \ref{prop;eff_gen}, and \ref{prop;eff_CP}, we will prove the following.
        
        \begin{thm}\label{thm;algo_trac}
         The class of finitely generated subgroups of a given globally fibered piecewise trivial suspension  
         of a free group is a hereditarily algorithmically tractable class.
        \end{thm}
        
        An immediate application of this work is that we will be able to fulfill the hypotheses of Proposition \ref{prop;struct_tree_compute} and Corollary \ref{coro;hybrid_canonical_for_sbgp} and therefore algorithmically construct the canonical structural splittings of suspensions as well as of subgroups of suspensions, which in turn is foundation for the computations in Sections \ref{sec;MWP} and \ref{sec;EMP}.

        \subsection{Effective coherence, effective generation, subgroup membership}

          Recall a definition introduced by Kapovich, Weidmann, and Miasnikov \cite{kapovich_foldings_2005}.

          \begin{defn}[Benign graphs of groups, {\cite[Definition 5.6]{kapovich_foldings_2005}}]\label{defn:benign}
            A finite graph of groups $\bbA$ is \emph{benign} if the following
            conditions are satisfied:
            \begin{enumerate}[(i)]
            \item\label{it:B1} For each vertex $v \in \verts A$ and an edge $e
              \in \edges A$ with $o(e)=v$ there is an algorithm with the
              following property. Given a finite set $X \subset A_v$ and an
              element $a \in A_v$ the algorithm decides whether \[I = \bk X
                \cap a \alpha_e(A_e)\] is empty. If $I \neq \emptyset$, the
              algorithm produces an element of $I$.
            \item\label{it:B2} Every edge group $A_e$  of $\bbA$ is slender.
            \item\label{it:B3} Every edge group $A_e$ of $\bbA$ has solvable uniform
              membership problem, i.e. there is an algorithm which, given a
              finite subset $X \subset A_e$ and an element $a \in A_e$ decides
              whether or not $a \in \bk X$.
            \item\label{it:B4} For each vertex $v \in \verts A$ and edge $e \in \edges A$
              with $o(e) = v$ there is an algorithm with the following
              property. For any finite subset $X \subset A_v$ the algorithm computes a
              finite generating set for the subgroup $\alpha_e(A_e) \cap \bk X.$
            \end{enumerate}
         \end{defn}
        
         Observe that conditions $(i)$ to $(iv)$ are conditions on vertex
         groups, edge groups and the attachment maps of their adjacent edge groups. From an algorithmic viewpoint, we prefer to state this definition in the context of graphs of groups, however, it can also be stated in the context of $G$-trees.

          \begin{prop}\label{prop:direct-prod-benign}
            Let $\bbX$ be  a finite graph of groups, such that any vertex group  $\bbX_v$  
            is of the form $\freegp{v} \oplus \bk {z_v}$, where $\freegp{v}$ is a
            free group and any attached image of edge group $\alpha_e(\bbX_e)\leq \bbX_v$ is of the form
            $\bk{c_v,z_v}, c_v \in \freegp{v}$, then $\bbX$  is benign: properties (\ref{it:B1}) and
            (\ref{it:B4}) of Definition \ref{defn:benign} hold.
          \end{prop}

          \begin{proof}
            Consider the vertex group $\bbX_v= \freegp{v} \oplus \bk z$, and an
            adjacent edge $e$, with attaching map $\alpha_e: \bbX_e\to \bbX_v$.  Consider the standard projection
            $\pi_{\freegp{v}}: \freegp{v} \oplus \bk z \onto \freegp{v}$. Given any finite set
            $S \subset \freegp{v} \oplus \bk z$ it is possible to compute the image
            $\bar S = \pi_{\freegp{v}}(S)$. By hypotheses the groups $\alpha_e(\bbX_e)$
            that occur in Definition \ref{defn:benign} are subgroups of the
            form $\bk{c, z}$ where $c \in \freegp{v}$.  In particular such groups
            contain $\ker(\pi_{\freegp{v}})$. Therefore
            $I = \bk S \cap a \alpha_e(\bbX_e) \neq \emptyset$ if and only if
            $\bar S \cap a\bk c \neq \emptyset$, which is routine to check in
            a free group. Because $g \in a\bk{c, z}$ if and only if $[a^{-1}g,c]=1$, an
            element of the intersection can be found by enumerating $\bk{S}$
            and checking whether the commutator vanishes. (\ref{it:B1}) of Definition
            \ref{defn:benign} therefore holds.
        
            We now show that property (\ref{it:B4}) of Definition
            \ref{defn:benign} holds. 
            Write\[ S =(f_1z^{n_1},\ldots,f_kz^{n_k})\] and perform
            Nielsen reduction (see \cite[\S I.2]{lyndon_combinatorial_2001})
            on the tuple $(f_1z^{n_1},\ldots,f_kz^{n_k})$ ignoring the
            $z$-letters. This will give a tuple\[
              R=(r_1z^{m_1},\cdots, r_j z^{m_j},1\cdot z^{m_{j+1}},\ldots,1\cdot z^{m_k})
            \] with $r_i \neq 1$ for $i\leq j$ (we allow $j=0$). Let
            $e = \gcd(m_{j+1},\ldots,m_k)$ then we
            have\[ \bk{S} = \bk{r_1z^{m_1},\cdots, r_j z^{m_j}}\oplus\bk{z^e}
            \] with $\pi_{\freegp{v}}(\bk{S})=\bk{r_1,\cdots, r_j}=\bk{\bar S}$, where the
            $r_i$ form a basis of $\bk{\bar S}$. It is possible to decide the
            minimal exponent $l$ such that $\bk{\bar S} \cap \bk c =\bk{c^l}$
            and to find the explicit (possibly empty) product (or word)
            $c^l=W(r_1,\cdots, r_j)$. It follows
            that\[ \bk{S} \cap \bk{c,z} = \bk{W(r_1z^{m_1},\cdots,
                r_jz^{m_j}),z^e} = \bk{c^lz^{m_1+\cdots+m_j},z^e},\] with $e$
            possibly zero and with $l=0$ if and only if
            $\bk c \cap \bk{\bar S} =\{1\}$. Since every step was computable,
            (\ref{it:B4}) of Definition \ref{defn:benign} holds.
            
          \end{proof}

        We will use the benign property with the following result  of Kapovich Weidmann and Miasnikov.

          \begin{thm}[{\cite[Theorems 5.8 (b)) and 5.13]{kapovich_foldings_2005}}]\label{thm:folding}
        	Let $\bbA$ be a benign graph of groups such that every vertex
        	group is effectively  
        	coherent
        	and has  solvable uniform membership problem. Then
        	$\pi_1(\bbA, a_0)$ is also effectively coherent and has solvable
        	uniform membership problem. Moreover, the algorithms for these two problems can be constructed from those given by the assumption.
        \end{thm}
        
          From this theorem we deduce the following, about piecewise trivial suspensions of free groups. 

          \begin{prop}\label{prop;eff_coh_ump_for_GDT}
         If $\torus \alpha$ is a piecewise trivial suspension 
         globally fibered with fiber a finitely generated  free group, then it is 
         effectively coherent and has effectively
            solvable uniform membership problem.
          \end{prop}
          \begin{proof}
            Given the abelian splitting 
            of
            $\torus \alpha$, properties (\ref{it:B2}) and (\ref{it:B3}) of
            Definition \ref{defn:benign} follow immediately. Let us now focus
            on groups of the form $\freegp{v} \oplus \bk z$, where $\freegp{v}$
            is free. Properties (\ref{it:B1}) and (\ref{it:B4}) follow from
            Proposition \ref{prop:direct-prod-benign}. The result now follows from
            Theorem \ref{thm:folding}.
          \end{proof}

          Treating a finite subset of a subgroup of $\torus \alpha$ as generating a subgroup of $\torus \alpha$, we immediately obtain: 
          
          \begin{prop}\label{prop;eff_coh}
          	The class of finitely generated subgroups of piecewise trivial suspensions globally fibered, with fiber a  finitely generated free group, is effectively coherent.  
          \end{prop}
         
          We can also solve the generation problem.
          
          \begin{prop}\label{prop;eff_gen}
          	The class of finitely generated subgroups of piecewise trivial suspensions  globally fibered, with fiber a finitely generated free group, has a uniform solution to the generation problem.

            More generally, if a class of groups  is (uniformly) effectively
          coherent, and has a uniform (over the groups in the class) solution
          to the uniform (over the f.p. subgroups in a group) membership
          problem, then it has a uniform  solution to the generation problem
          for its subgroups.
        \end{prop}
        
        \begin{proof}
        Let us prove the second assertion. Given a subgroup $H$ given by generators $\calS_H$, and a family $\calS_K$ of elements, one can compute a presentation of $H$ and a presentation of $\langle \calS_K \rangle$, one can then use the uniform membership problem for f.p. subgroups to determine whether $ \calS_K  < H$ and $\calS_H  < \langle \calS_K \rangle$. The first assertion then follows from Proposition \ref{prop;eff_coh}.
        \end{proof}
          
          \subsection{Consequences of effective coherence: computing canonical splittings and solving the conjugacy problem in subgroups}
          
          We are now in a position to compute canonical splittings of finitely generated subgroups of $\torus\alpha$. Since effective coherence lets us algorithmically obtain a presentation for finitely generated subgroups and since subgroups of $\torus\alpha$ are without 2-torsion and have solvable word problem, we can apply \cite[Corollary 1.1]{touikan_detecting_2018} to get:
          
          \begin{cor}\label{cor;subgroupsgrushko}
            Given a finite generating set of a subgroup $H$ of a piecewise trivial suspension, we can compute a Grushko decomposition and presentations of the indecomposable free free factors. In particular, we can decide if $H$ is freely indecomposable.
          \end{cor}
         
        \begin{prop}\label{prop;construct-canonical}
            Let $S=\{h_1,\ldots,h_r\}$ be a collection of elements in a piecewise trivial suspension $\Gamma=\torus\alpha$. Then we can compute a Grushko decomposition of $H = \bk S$ and for each one-ended free factor of $H$, we can compute the canonical splitting $\hat\bbX$ given by Corollary \ref{coro;hybrid_canonical_for_sbgp}.
        \end{prop}
        \begin{proof}
            Proposition \ref{prop;eff_coh_ump_for_GDT} and Corollary \ref{cor;subgroupsgrushko} enable us to get presentations for the freely indecomposable free factors of $H$ or to assert that $H$ is in fact a free group.
             
            Note that the free group $F$ is hereditarily algorithmically tractable and that it is possible to compute centralizers of cyclic subgroups in subgroups of $F$. It follows that the requirements of Corollary \ref{coro;hybrid_canonical_for_sbgp} are fulfilled for the one-ended free factors of $H$, and therefore that it allows to compute the canonical splitting $\hat\bbX$.
        \end{proof}

         Bogopolski, Martino, Maslakova, and Ventura \cite{BMMV} proved that the conjugacy problem in [finitely-generated free]-by-cyclic groups is solvable. We can extend this result to finitely generated subgroups. 
        
        \begin{prop}\label{prop;eff_CP}
           The Conjugacy Problem is uniformly solvable for the class of finitely generated subgroups of a piecewise trivial suspension of a finitely generated free group.
        \end{prop}
            Given that we have a canonical tree for such a group, it is not surprising that the proof will split in two cases, the case of elliptic elements (in which the trivial action of the normalizer of edge groups helps to reduce to the vertex case), and the case of hyperbolic elements. Although the proof in this later case invokes computations with normal forms,  and Proposition \ref{prop;compute_finitely_many_short}  which is only proved later, we note that this Proposition \ref{prop;eff_CP} is not actually used anywhere else in this paper so there is no circular logic. Proposition \ref{prop;compute_finitely_many_short} is a statement about the normal forms that we develop in Section \ref{sec;MWP}.
        \begin{proof} We are given a piecewise trivial suspension as in the statement, a finite set $S$ of elements in it, and we denote by $H$ the subgroup generated by $S$.  By Corollary \ref{cor;subgroupsgrushko} we can decide whether $H$ is freely decomposable, and if so, find a Grushko decomposition of $H$. Through a classical reduction to free factors (see for instance 
        \cite{lyndon_combinatorial_2001}), it is enough to prove that the conjugacy problem is solvable in the case where $H=\bk S$ is one-ended. Hence, we assume in the following that $H$ is one-ended. 
        
        Compute (by Proposition \ref{prop;construct-canonical})  the canonical splitting of $H$ from Corollary \ref{coro;hybrid_canonical_for_sbgp}, and denote it by $\bbX$.   
        
        We are now given $g,g' \in H$ and must decide whether they are conjugate. Given the splitting constructed in \ref{coro;hybrid_canonical_for_sbgp} and our understanding of free abelian groups and conjugacy in the vertex groups, which are all either free, free abelian of rank 2, or a direct product of a free group and a cyclic group, we can decide if $g, g'$ are in reduced form, and if not perform elementary reductions to obtain words of shorter syllable length. In this manner we can decide whether $g$ and $g'$ are hyperbolic or elliptic. Obviously, if they are not both hyperbolic or both elliptic, they cannot be conjugate.
        
        If $g,g'$ are both hyperbolic elements, Proposition \ref{prop;compute_finitely_many_short} can be used to decide whether they are conjugate. 
        
       If $g,g' \in H$ are elliptic, then since $\bbX$ is computed,  we can decide whether $g,g'$ are conjugate into the same vertex group $\bbX_v$. If for each vertex $v$ of $\bbX$,  one of them is not conjugate into $\bbX_v$, then they cannot be conjugate to one another. Hence we may assume that we found $v$ such that both $g, g'$ are conjugate into  $\bbX_v$. We may as well assume that they are both in $\bbX_v$, or (by a slight abuse of notation and a choice of lift of $v$ in the dual Bass-Serre tree $T$), both in $H_v$ (the stabilizer of $v\in T$ in $H$).  
       

        Let $T_g = \{x \in T|g\cdot x = x\}$ be the pointwise fixed subtree  of $g$, and  $T_{g'}$ defined analogously. By hypothesis, $v \in T_g \cap T_{g'}$. If there exists $h\in H$ such that $hgh^{-1}=g'$, then  $h\cdot T_g = T_{g'}$. 
        
        If $g$  is not conjugate into any images of edge groups,  then $T_g=v$, then any such conjugator  $h$ must be in $\bbX_v$.
        In this case, the existence of such $h$ is decided by a solution to the conjugacy problem  in $\bbX_v$, a free group or a trivial suspension of a free group.

        We may therefore assume that  $g$ and $g'$ are conjugate to the image of a certain edge group, adjacent to $v$.   
        Observing that $H=\pi_1(\bbX,v)  \leq \Bass(\bbX)$ is a free factor,  if $g,g'$ are conjugate in $\Bass(\bbX)$, then there must be some reduced word $w=a_0e_1a_1\cdots e_n a_n$ with the $e_i$ being symbols in $(\edges X)^{\pm 1}$ and $a_i \in \bbX_{v_i}$, for which
        \begin{equation}\label{eqn:conj}
             g' = w^{-1} g w = (a_n^{-1}\bar e_n \cdots \bar e_1 a_0^{-1}) g (a_0e_1a_1\cdots e_n a_n) \, \in \bbX_v.
        \end{equation} 
        For the right-hand side expression to cancel down to $g'$ we need $e_1,\ldots,e_n$ to be a loop based at $v$,  and the word $w$ defines an element of $\pi_1(\bbX,v)$. Furthermore, in the Bass-Serre  $T$, 
        then \eqref{eqn:conj} says that $g$ stabilizes $v$ and $w\cdot v$. Since $d(v,w\cdot v)=n$, and since the action on $T$ is  4-acylindrical,  $n\leq 4$. We distinguish two cases.
        
        \noindent \textbf{Case 1:} If $v$ is an abelian black vertex, then we can assume $a_0$ is trivial and $\bar e_1 g e_1=g_1$ lies in lies in the image of an edge group in $\bbX_{v_1}$. Now if $e_2 = \bar e_1$ then $a_1$ needs to centralize $\bar e_1 g e_1$ which means $a_1$ is in the image of the same edge group; contradicting that $w$ is reduced. Otherwise,  $g_1$ is conjugate into the image of the edge group associated to $e_2$. By our description of the canonical splitting of $H$, this is only possible if $g_1$ is a non-trivial element of the centre of the white vertex group $\bbX_{v_1}$, in which case $a_1^{-1}g_1a_1= g_1$. Let $g_2 = \bar e_2 g_1 e_2$ then $g_2 \in \bbX_{v_2}$ which, again, is a black abelian vertex group so $a_2^{-1}g_1 a_2=g_2$, so we can assume $a_2$ is trivial and since $w$. If $v_2=v$ and $n=2$ we have $g' = \bar e_2 \bar e_1 g e_1 e_2$.
        
        If $n > 2$ then $n=4$ as all loops in $X$ must have even length. Let $g_3 = \bar e_3 g_2 e_3 = \bar e_3 \bar e_2 \bar e_1 g e_1 e_2 e_3$. We must have that $g_3$ is conjugate in $\bbX_{v_3}$ to the image of the edge group associated to $e_4$, which is distinct from the image of the edge group associated to $e_3$ in $\bbX_{v_3}$. Again this is only possible if $g_3$ is in the centre of $\bbX_{v_3}$, which is a non-abelian white vertex group. Now, the equality $g_3 = \bar e_3 \bar e_2 g_1 e_2 e_3$ in $\Bass(\bbX)$ translates to the fact that there are vertices $w,w'$ distance 2 apart and some non-trivial element $h \in H$ that lies in the centre of $H_w$ and $H{w'}$, but since we know that central elements in white vertices fix the ball of radius 2 about those vertices, we get a contradiction to 4-acylindricity.
        
        It follows that in Case 1, it is enough to take all loops $\ell=e_1e_2$ of length 2 based at $v$ in $X$ and verify if $g' = \bar{\ell} g \ell$, to decide whether $g,g'$ are conjugate.
        
        \noindent\textbf{Case 2:} If $v$ is a non-abelian white vertex, then a similar analysis to Case 1 above lets us conclude that either\[
            g' = a_2^{-1}\bar e_2\bar e_1a_0^{-1} g a_0 e_1e_2a_2
        \] where $g$ is conjugated in $\bbX_v$ by $a_0$ into the image of the edge group associated to $e_1$ and $g'$ is conjugated in $\bbX_v$ by $a_2^{-1}$ into the image of the edge group associated to $e_2$, or\[
        g' = a_4^{-1}\bar e_4\bar e_3 \bar e_2e_1a_0^{-1} g a_0 e_1e_2e_3e_4a_4, 
        \] where $g$ is conjugated in $\bbX_v$ by $a_0$ into the image of the edge group associated to $e_1$ and $g'$ is is conjugated in $\bbX_v$ by $a_4^{-1}$ into the image of the edge group associated to $e_4$, and in particular $\bar e_2 \bar e_1 a_0^{-1} g a_0 e_1 e_2$ is in the centre of $\bbX_{v_2}$. Note that for $H$, in a non-abelian white vertex group, the normalizer of the image of an edge group acts trivially on that subgroup by conjugation.
        
        It follows that in Case 2, to decide if $g,g'$ are conjugate, it is enough to consider all loops $\ell$ of length 4 based at $v$, to decide if there are elements $a,a'$ that conjugate $g,g'$ (respectively) into the appropriate edge groups and check whether $g'=a'^{-1} \bar\ell a^{-1} g a \ell a'$.
        
        Thus in both cases the problem reduces to deciding  whether $g,g'$ are conjugate into the edge groups (which is a trivial task for abelian black vertex groups) and then if these conjugates, which are unique, are "immediately" conjugate by the Bass-Serre relations, i.e. a finite composition (at most 4) of the edge group monomorphisms specified in $\bbX$ and their inverses bring the conjugate of $g$ to the conjugate of $g'$.
        \end{proof}
    
    The following restatement of what was shown in the proof above will be useful.
    \begin{lem}\label{lem:conj-elliptic}
        $g,h \in \bbX_v$ are conjugate in $\Bass(\bbX)$ if and only if either\begin{enumerate}
            \item $g$ and $h$ are conjugate in $\bbX_v$, or
            \item $g$ is conjugate in $\bbX_v$ to $g' \in \tau_{e}(\bbX_e)$, $h$ is  conjugate in $\bbX_v$ to $h' \in \tau_{e'}(\bbX_{e'})$. The conjugates $g',h'$ are unique and there is some $\ell$ that is an edge path in $X$ based at $v$ of length at most 4 such that $g' = \ell g' \bar\ell$ in $\Bass(\bbX)$
        \end{enumerate} 
    \end{lem}
    
    \section{The mixed Whitehead problem in  globally fibered piecewise trivial suspensions.}\label{sec;MWP}

        The action of a group $G$ on itself by conjugation extends on the set of all tuples, which is    $G^{(\bbN)}= \sqcup_{n>0}  G^n $ (all our tuples are ordered).   If $T\in G^{(\bbN)}  $, we write $[T]$ for its conjugacy class, that is,  the orbit of $T$ under the action of $G$ by conjugation.

        The mixed Whitehead problem under a subgroup $H$  of $\aut G$ or of $\out G$  is the following decision problem. Given an integer $k$, and   for  all $i\in \{1, \dots, k\}$   two tuples  $S_i \in G^{(\bbN)} $ and $T_i \in G^{(\bbN)} $,    decide whether there exists an automorphism $\sigma \in H$ such $\forall i, \,     [\sigma (S_i)] = [T_i]$.  
        If there exists such $\sigma$, we say that  $\vec S = (S_1, \dots, S_k)$ and $ \vec T = (T_1,\dots T_k)$  are \emph{mixed Whitehead equivalent} under $H$. This problem was emphasized and studied by Bogopolski and Ventura \cite{Bogopolski_Ventura}, initially for torsion-free hyperbolic groups.

        In \cite{dahmani_deciding_2019} the mixed Whitehead problem is used for comparing attaching maps of edge groups in vertex groups in certain graphs of groups. Typically, if one is  given a graph, a group for each vertex and for each edge, and  two families of attaching maps from edge groups to vertex groups, one has  two graphs of groups; the mixed Whitehead problem helps to decide whether these two graphs of groups are isomorphic as such. 
        More specifically, suppose for example that  we have a graph. with two vertices $v,w$ and two edges $e_1, e_2$ from $w$ to $v$. Assume that the group of $w$ is a given group $H$, and the edge groups are $H_1, H_2$ with given attaching maps to $H$. Let $G$ be the group of $v$. We need attaching maps $H_i\to G$. Consider  two pairs of monomorphisms $\alpha_1,\alpha_2:H_1 \to G$ and $\beta_1,\beta_2: H_2 \to G$. The choices $(\alpha_1, \beta_1)$ and $(\alpha_2, \beta_2)$ define two graphs of groups. In order to  
        decide whether they 
        are (globally) isomorphic, 
        we want to decide whether there exists an automorphism $\sigma \in \aut G$ such that $\alpha_i$ is conjugate in $G$ to $\sigma\circ\beta_i$, for $i=1,2$. This is the mixed Whitehead problem over $\aut{G}$ for the tuples that are images of a choice of generating tuples in $H_1$ and $H_2$.

        In the refinement \cite{dahmani_touikan_reducing_2021} that is designed to study the conjugacy problem in $\out{F_n}$, the vertex and edge groups have a fibre and an orientation. Typically, the group $G$ that appears there is isomorphic to $\torus \alpha$ for $\alpha$ a polynomially growing automorphism of a free group.  In this particular context, we need to solve the mixed Whitehead problem over the subgroup of $\out{G}$ that preserves the fibre and the orientation.

        The goal of this section is to prove the following:

        \begin{thm}\label{thm;MWP_fo}
            If $\torus \alpha= F\rtimes \bbZ$  is a globally fibered piecewise trivial suspension with fiber a free group $F$, then the mixed Whitehead problem for $\torus \alpha$ under  $ {\torus \alpha}$ is decidable. 
        \end{thm}
    
        Given $\torus\alpha$, by Proposition \ref{prop;construct-canonical}, we can compute its structural tree to get $\torus\alpha = \pi_1(\bbX,b)$ and express any $g \in \torus\alpha$ as some word\[
        g = g_1\, e_1 \, g_2\, e_2\, \cdots \, e_n\, g_{n+1}
        \] where the $g_i$ are elements of the vertex groups that are isomorphic to $F_n\oplus \bbZ$ (for different $n$, depending on the vertices). In this sense, elements of $\torus\alpha$ can be thought of as tuples of elements of the ``simpler'' vertex groups. Unfortunately, this approach is too naive. 

        In Subsection \ref{sec:nf} we present normal form conventions that enable us to ``coordinatize'' elements $g\in \torus\alpha$, or elements of one-ended subgroups of $\torus\alpha$, using a tuple $DC(g,e)$ of double cosets, a tuple of integers $\mathrm{pow}(g,e)$ of powers of generators of cyclic groups, and some power $t^k$ of the stable letter. These normal forms rely on a basepoint $b$ of $\bbX$, that defines $\pi_1(\bbX,b)$, and on some edge $e$ called a polarizing edge, whose importance will only become fully apparent in Subsection \ref{sec:power-tuples}.

        In subsection \ref{sec;short_pos} we introduce short positions, which are a strengthening of cyclically reduced conjugates. Although the concept of cyclically reduced is well-understood in free groups or simple graphs of groups such as amalgamated products or HNN extensions, the concept of cyclically reduced in $\pi_1(\bbX,b)$ is not satisfactory when the axis of translation of a hyperbolic element in the dual Bass-Serre tree doesn't contain a lift of the base point $b$. This leads us to ``changing basepoints'', i.e. to consider $\pi_1(\bbX,v)$ instead of $\pi_1(\bbX,b)$, where both $b,v \in \verts X$. This is achieved by noting that \[
        \pi_1(\bbX,v),\pi_1(\bbX,b) \leq \Bass(\bbX)
        \] are conjugate subgroups. By exploiting the 4-acylindricity of the canonical trees for $\torus\alpha$ and its one-ended subgroups, we can show that hyperbolic elements have a finite number of distinguished conjugates, called conjugates in short position and that these can be computed. On the one hand, this result covers the conjugacy problem for hyperbolic elements in one-ended subgroups of $\torus\alpha$. On the other hand, short positions play another important role. Given two tuples of tuples $\vec S = (S_1,\ldots,S_k)$ and $\vec T = (T_1,\ldots,T_k)$, the mixed Whitehead problem is actually an orbit problem for $\autfo{\torus\alpha}\times(\torus\alpha^k)$ as we need to find not only $\sigma \in \autfo{\torus\alpha}$ but also conjugators $g_1,\ldots,g_k$ such that $g_i\sigma(S_i)g_i^{-1}=T_i$ for $i=1,\ldots,k$. Short positions will play a role in restricting the conjugators we need to consider.

In general, for a free group $F$ there is no reasonable way in which $\aut F$ will act on a double coset space $H\backslash F /K$ and no known algorithms that decide if there is an automorphism that sends a double coset to some target double coset. In Section \ref{sec:link-config} we associate a pair of subgroups to the double cosets that occur in $DC(g,e)$. The reason for doing so is that there is an algorithm due to Gersten (see Theorem \ref{thm:gersten}) that can decide if there is an automorphism of a free group that sends a subgroup of a free group to another subgroup. We then introduce linkage configurations which is a way of collecting all the linkages from a sequence $DC(g,e)$ (as well as their duals) that will be used later on to decide if it is possible to bring all the double cosets, in the right order, arising from words in the tuples in $\vec S$ to corresponding cosets arising from the content of $\vec T$.

In Section \ref{sec:tuples-of-tuples} we describe slight modifications to the tuples in $\vec S$ and $\vec T$ that will make them easier to work with. These amount to ensuring that the first entry of a tuple is hyperbolic, if possible, and conjugating in $\Bass(\bbX)$ so that the first entry is also in centred short position.

In Section \ref{sec;action_delta_1} we verify that a special kind of action of $\delta_1\autfo\bbX$ on tuples of tuples behaves well with respect to choices of normal forms and other conventions. In particular, we show that a variation of Gersten's algorithm: the mixed Whitehead problem for subgroups of free groups (see Theorem \ref{thm;gersten_hack}) can be used to simultaneously send all double cosets that occur in $\vec S$ to the double cosets that occur in $\vec T$.

In Section \ref{sec:solving-MWHP} we finally show, through a series of reductions, how to solve the fibre and orientation preserving mixed Whitehead problem in $\torus\alpha$ by using all the material computations developed up to that point. Finally in Section \ref{sec;hack} we give the proof of Theorem \ref{thm;gersten_hack}, which is substantial and was delayed so as not to break the flow of our arguments.

\subsection{Describing elements}\label{sec:describe}

\subsubsection{Normal forms}\label{sec:nf}

For this subsection, we will exceptionally consider $G$ to either be a piecewise trivial suspension or a finitely generated one-ended subgroup of a piecewise trivial suspension. We have a  canonical graph of groups $\bbX$ and for some $b \in \verts X$ we identify $G = \pi_1(\bbX,b)$, where $b$ is chosen to be a white vertex, or equivalently, that $\bbX_b$ is non-abelian. We also fix a \emph{breadth first search} (BFS) spanning tree $\tau \subset X$ based at $b$, this means that for any $v \in \verts X$, $\tau$ contains a shortest path from $b$ to $v$.

Consider a word in the generators of the Bass group representing an element in $\pi_1({\bbX},b)$: 
    \begin{equation}\label{eqn:bass-grp-elt}
        w= g_1\, e_1 \, g_2\, e_2\, \cdots \, e_n\, g_{n+1} \in \Bass(\bbX)
    \end{equation}
with $g_j \in \bbX_{\tau(e_{j-1})}$ if $j>1$ and  $g_j \in \bbX_{i(e_{j})}$ if $j\leq n$. We start by giving vocabulary to describe a few characteristics of $w$ and the element it defines. Its \emph{underlying $X$-path} (in $X$) is the path $(e_1, \dots, e_n)$, from $b$ to $b$. The word $w$ is \define{reduced} if its $X$-path has minimal length among all words representing the same element. Its \emph{syllable length} is the number of edges that appear in the $X$-path for any reduced word. 

Even among reduced words, the choice of elements $g_i$ that appear in \eqref{eqn:bass-grp-elt} are not unique. A normal form is a function that assigns to every element of $g\in G$ of syllable length $n$ a tuple $(g_1,\ldots,g_{n+1})$ of elements lying in the appropriate vertex groups, i.e. it is a function that assigns to every $g \in G$ a specific representative word.

The vertex groups of $\bbX$ are either isomorphic to $F_v$, or $F_v\oplus\bk{t_v}$, or $\bbZ\oplus \bbZ$, or $\bbZ$, where $F_v$ is a free group of rank $n_v$ (isomorphic to a subgroup of the fiber $F$, but not necessarily in $F$, in the case where $G$ is not the entire globally fibered ambient group). A choice of ordered basis for each $F_v$ induces the \texttt{shortlex} well-order $\leq_v$ on $F_v$. If $\bbX_v \simeq F_v\oplus\bk{t_v}$  then we write $(w,t_v^n) < (u,t_v^m)$ if and only if either $|n|<|m|$ or, if $|n|=|m|$ then $|m|/m < |n|/n$, or if $n=m$ then $w <_v u$. In this manner, we have a well-order on each non-abelian vertex group.

For all edges $e$ with $\tau(e)= v$, we denote $I_e = \tau_e(\bbX_{e})\leq \bbX_v$. Similarly, if $v = i_e(e)$, we denote by $O_e = i_e(\bbX_{e}) \leq \bbX_v$. Note that $O_e=I_{\bar e}$. A pair of edges $(e,e')$ such that $\tau(e)=i(e')$ is called a \emph{turn}. An element represented by a word $w$ as in \eqref{eqn:bass-grp-elt} is an element of $\pi_1(\bbX,v)$ if and only if $i(e_1)=b$ and each $(e_i,e_{i+1})$ is a turn where indices $i,i+1$ are taken are modulo $\bbZ/n\bbZ$. The \emph{double coset space associated to the turn $(e,e')$} is $I_e\backslash\bbX_v/O_{e'}$ and we define\[
    DCR(e,e') = \{\bar g \in \bbX_v : \bar g \textrm{ is minimal w.r.t $<$ in its $I_e\backslash\bbX_v/O_{e'}$ double coset}\},
\] to be the set of \emph{double coset representatives for the turn $(e,e').$}

We will now define normal forms. First, we note that $\bbX$ will always have an underlying bipartite graph with black and white vertices. Furthermore, if $v$ is a black vertex then $\bbX_v$ will be abelian (either isomorphic to $\bbZ$ or $\bbZ^2$) and it will coincide with the images of all incident edge groups. In particular, we can always use the Bass group relations to rewrite $w$ as in \eqref{eqn:bass-grp-elt} so that $g_{2j}=1$ for all $1\leq j \leq n/2$.

\begin{conv}
    When writing elements of $\pi_1(\bbX,b)$ as in \eqref{eqn:bass-grp-elt}, we will always ensure that $b$ is a white vertex with $\bbX_b$ non-abelian and we will ensure that $g_i$ that lie in abelian vertex groups are trivial. 
\end{conv}

Given a word $w$ as in \eqref{eqn:bass-grp-elt}, due to not being ``sandwiched'' between two letters from $\edges X$, there is some ambiguity in the choice of double coset space for the elements $g_1,g_{n+1}$. We remedy this by picking an edge $e \in \edges X$ such that $i(e)=b$. Let $g \in \pi_1(\bbX,b)$ be written as a reduced word as in \eqref{eqn:bass-grp-elt} then the \emph{tuple of double cosets for $g$ polarized by $e$} is the tuple\[
DC(g,e)=(I_{\bar e}g_1O_{e_1},1,I_{e_2}g_3O_{e_3},1,\ldots,I_{e_n}g_{n+1}O_{e}),
\] in particular we take the first and last cosets to lie in the double coset spaces associated to the turns $(\bar e,e_1)$ and $(e_n,e)$. This choice of $e$ will become important later on.

We also have the corresponding sequence of \emph{double coset representatives for $g$ polarized by $e$.}
\[
    DCR(g,e) = (\bar f_1, 1, \bar f_3 ,\ldots,\bar f_{n+1}).
\]
        We note that, as indicated by the notation, $DCR(g,e)$ and $DC(g,e)$ depend only on $g \in \pi_1(\bbX,b)$ and not on the choice of word used to represent $g$. 
        Let $\bbX_v = F_v\oplus\bk{t_v}$ be a non-abelian vertex group with infinite cyclic centre generated by $t_v$.  Since $G$ is a one-ended subgroup of a piecewise trivial suspension, Proposition \ref{prop;canonical_for_subgroups} item \ref{it:interstitial-central} implies that the element $t_v$ lies in $O_{e}$ for all $e$ such that $i(e)=v$.

        For any $g \in \pi_1(\bbX,b)$, expressed as a word $w$ of the form of  \ref{eqn:bass-grp-elt},  let  $\bar f_i$ be the $i$-th entry in the tuple $DCR(g,e)$, and for all edge $e'\in \edges X $, let $c_{e'}$ is the $<$-minimal generator of the cyclic group $I_{e'}\cap F_{\tau(e')}$. The previous observation that $t_{i( \bar {e'})}$ lies in $O_{\bar {e'}} = I_{e'}$,  allows to express $g$ as a new word \begin{equation}\label{eq:normalform}
    \bar{w}=c_{\bar e}^{d_0}\bar f_1 e_1e_2 c_{e_2}^{d_2}\bar f_3 \cdots e_nc_{e_n}^{d_n} \bar f_{n+1} c_{\bar e}^{d_{n+1}} t_b^k.
\end{equation}

The word $\bar w$ given in \eqref{eq:normalform} is called the \emph{normal form for $g \in \pi_1(\bbX,b)$ polarized by $e$} the tuple of exponents\[
\mathrm{pow}(g,e)=(d_0,d_2,\ldots,d_n,d_{n+1},k)
\] is called the \define{power vector} and the exponent $k$ is called the \emph{abscissa}, where we set $k=0$ if $t_b=1$. In summary.
\begin{lem}[Normal form polarized at $e$]
  Given $g\in \pi_1(\bbX,b)$ and some $e \in \edges X$ such that $i(e)=b$ the $X$-path underlying $g$, the  tuples $DCR(g,e)$ and $\mathrm{pow}(g,e)$, and the abscissa of $g$ are well-defined.
\end{lem}

\subsubsection{Short positions}\label{sec;short_pos}

We say that $g \in \pi_1(\bbX,b)$ is \emph{cyclically reduced} if it has minimal syllable length among all its conjugates in $\pi_1(\bbX,b)$. For graphs of groups that are more complicated than amalgamated products or HNN extensions, it may be that $g$ is cyclically reduced, but its syllable length exceeds the translation length of the action of $g$ on the dual Bass-Serre Tree. By our convention that $\tau$ is a BFS-tree we have that $g$ has a family of cyclically reduced conjugates of the form\[
    E g_v \bar E
\] where $E$ is a path from $b$ to some white vertex $v$ in $\tau$ and $g_v\in \pi_1(\bbX,v) \leq \Bass(\bbX)$. Furthermore, the syllable length of $g_v$ coincides with the translation length of $g$. We say that $g_v$ is a \emph{recentering of $g$ at $v$}. We note that the axis of translation of $g$ in the Bass-Serre tree must contain a lift of $v$.

We assume that $g \in \pi_1(\bbX,b)$ is a hyperbolic element and therefore that any recentering $g_v$ is also hyperbolic. We will switch the basepoint to $v$ and use the notation developed in Section \ref{sec:nf} in this new context.

We say that $h \in \pi_1(\bbX,v)$ is \emph{centred $e_1e_2$-prefix-short} if it is cyclically reduced, has syllable length equal to its translation length, and can be written in normal form
\[
    h = e_1e_2 c_{e_2}\bar f_2 \cdots.
\] It is obvious that any hyperbolic element recentered at $v$ is conjugate in $\pi_1(\bbX,v)$ to an element in prefix-short form. Two prefix-short elements\[
 e_1e_2 c_{e_2}^{d_2}\bar f_2 \cdots, e_1e_2 c_{e_2}^{d'_2}\bar f_2' \cdots
\] are \emph{elliptically-conjugate} if they are conjugate by an element of $O_{e_1}\leq \pi_1(\bbX,v)$ and we say they lie in the same \emph{elliptic-class}. 

We say that $h\in\pi_1(\bbX,b)$ is \emph{$e_1e_2$-prefix-short} if there exists a path $E\subset \tau$ from $b$ to $v$ such that \[
    h = Eg_v\bar E
\] and $g_v$ is centred $e_1e_2$prefix-short. It is clear that every hyperbolic  $g\in\pi_1(\bbX,b)$ has at least one prefix-short conjugate. 

If $h \in \pi_1(\bbX,v)$ is centred $e_1e_2$-prefix-short, we say it is in \emph{centred short position} if its $e_1$-polarized normal form is
 \begin{itemize}
     \item $h = e_1e_2\bar c_{e_2}^{d_2} f_2c_{\bar e_1}^{d}t_v^k$, with $d_2=0$ and $d$ minimal (w.r.t some fixed well-order on $\bbZ$) among all other conjugates of $h$ that are in $e_1e_2$-prefix-short position with $d_2=0$, if $h$ has translation length 2
     \item $h = e_1e_2c_{e_2}^{d_2}\bar f_2e_3e_4c_{e_4}^{d}\bar f_4 \cdots$,  with $d_2=0$ and $d$ minimal (w.r.t some fixed well-order on $\bbZ$) among all other conjugates of $h$ that are in $e_1e_2$-prefix-short position with $d_2=0$, if $h$ has translation length 4 or more.
 \end{itemize}
 
 Although they are not explicitly invoked in the definition, the prefix $e_1e_2$ as well as the basepoint $v$ are completely specified. To show that this is canonical we have the following:
 
\begin{lem}\label{lem:6-edges}
  The exponent $d$ given in the definition of centred short position determines the first 6 edges of the segment $[\tilde v,h\cdot \tilde v]$ of the axis of $h$ for the action on the Bass-Serre tree dual to $\pi_1(\bbX,v)$.
\end{lem}
\begin{proof}
  First, consider the case where $h$ has translation length 2. Then the axis of $h$ is the same as the axis of $h^3$, passing to $e_1$-polarized normal forms we have\[
    h^3= e_1e_2\bar f_2 e_1e_2c_e^{d'}\bar f_2 e_1e_2 c_{e_2}^{d''} \bar f_2 t_v^{3k},
  \] where ${c_{e_2}}^{d'}$ is obtained by migrating $c_{\bar e_1}^{d}$ across $e_1e_2$, and it is clear that the exponent $d'$ determines the first 6 edges of the path $[\tilde v,h^4\tilde v]$ and therefore of the axis of $h$.
  Similarly, if $h$ has translation length 4 or more we can take $h^2$ and pass to $e_1$-polarized normal form\[
    h^2 = e_1e_2\bar f_2e_3e_4c_{e_4}^{d}\bar f_4 e_5e_6 c_{e_6}^{d_6} \bar f_6 \cdots,
  \] and again it is apparent that the first 6 edges of the path $[\tilde v,h^4\tilde v]$ and therefore of the axis of $h$ are determined by the exponent $d$.
\end{proof}

\begin{cor}
    Let $g \in \pi_1(\bbX,v)$ be a hyperbolic element whose syllable length coincides with its translation length. Then there is a unique element of $\bbX_v\leq \pi_1(\bbX,v)$ that conjugates $g$ to centred short position.
\end{cor}
\begin{proof}
    In $e_1$-polarized normal form we have $g=c_{\bar e_1}^{d_0}\bar f_1e_1e_2{c_{e_2}}^{d_2}\bar f_2 \cdots$ conjugating by $f_1 \in \bbX_V$ and then by the element of $a\in O_{e_1} \leq \bbX_v$ such that $a e_1e_2 = e_1e_2 {c_{e_2}}^{-d_2}$ gives us a normal form\[
        g_1=(a^{-1}\bar f_1^{-1} c_{\bar e_1}^{-d_0})g (c_{\bar e_1}^{d_0}\bar f_1 a) = e_1e_2 \bar f_2' \cdots
    \] whose "$f_1$-syllable" and $"d_2"$-exponent are trivial. We further note that any conjugation by an element of $\bbX_v \setminus O_{e_1}$ will undo this condition.
    
    Now there is a subgroup $Z_{e_1}\leq O_{e_1}$ such that $ze_1e_2 = e_1e_2 z'$ and $z'\bar f_2' = \bar f_2' z'$, which is either trivial or cyclic, that can affect the exponent $d$ given in the definition of centred short position while keeping the exponent $d_2=0$. I.e. $Z_{e_1}$ is a group of elements that fix an arc of length 4 in the Bass-Serre tree. 
    
    Let $k \in Z_{e_1}$ be the conjugator that gives the minimal exponent $d$. We have that $p=kc_{\bar e_1}^{d_0}\bar f_1 a$ is a conjugator that brings $g$ to centred short position. Note that the initial conjugator $c_{\bar e_1}^{d_0}\bar f_1 a$ is well defined up to left multiplication by some element $z \in Z_{e_1}$. It follows that in all cases, the double coset representative $\bar f_2'$ is common to all possible centred short positions.
    
    To see that it is unique, suppose there was another $p'$ such that $(p')^{-1} g p'$ was in centred short position. Then\[
        p^{-1}p' (e_1e_2\bar f_2' *\cdots) (p')^{-1}p^{-1} = e_1e_2 f_2' *'\cdots,
    \] where, both $h=e_1e_2 f_2' *$ and $h'e_1e_2 \bar f_2' *'$ are in centred short position with the same minimal exponent $d$. Now note that $\mathrm{axis}(h') = (p^{-1}p')\cdot \mathrm{axis}(h)$ and so by Lemma \ref{lem:6-edges} $(p^{-1}p')$ fixes a segment of length 6, which by 4-acylindricity implies that $p'=p$. This completes the proof.
\end{proof}

Finally we say that a hyperbolic $g \in \pi_1(\bbX,b)$ is in \emph{short position} if \[
    g = E g_v \bar E
\] where $E$ is a path in $\tau$ from $b$ to $v$ and $g_v$ is in centred short position. We call the vertex $v \in \verts X$ the \emph{anchor vertex of the short position} and the first edge $e_1$ of $g_v$ the \emph{polarizing edge of the short position}. Thus, given some $g \in \pi_1(\bbX,b)$, by going through the $\bk{g}$-orbits of the white vertices of $\mathrm{axis}(g)$ and applying the corollary above gives:

\begin{prop}\label{prop;compute_finitely_many_short}
  Let $G=\pi_1(\bbX,v)$ be a one-ended subgroup of a piecewise trivial suspension. Let $g\in \pi_1(\bbX,b)$ be hyperbolic with translation length $n$ then its conjugacy class contains at most $n/2$ elements in short position and the normal forms, anchor vertices and polarizing edges, of these elements can be computed.
\end{prop}

\subsection{Tuples of cosets and linkage configurations in piecewise trivial suspensions}\label{sec:link-config}

We now return our focus to a globally fibered piecewise trivial suspension $\torus\alpha = F \semidirect \alpha \bk t = \pi_1(\bbX,b)$. In the notation of subsection \ref{sec:nf} we will assume every white vertex group is decomposed as $\bbX_v = F_v \oplus \bk{t_v}$ where $F_v = F \cap \bbX_v$ (which is well-defined since $F \triangleleft \torus\alpha)$), thus elements in the fibre will always be smaller with respect to $<$ than other elements. In particular, if we write \[
g = a_1\bar f_1 b_1 e_1e_2 a_3 \bar f_3 b_3 \cdots e_n a_{n+1}\bar f_{n+1} b_{n+1}
\] where $\bar f_i$ is the $i$th entry in $DCR(g,e)$ for some $e$ with $i(e)=b$, then the representation of this form that minimizes the tuple $(a_1,b_1,\ldots,a_{n+1},b_{n+1})$ is precisely the normal form polarized by $e$. When it is not necessary to see the exponents in the power vector, we will call an expression\[
    g = f_1e_1e_2 f_3 \cdots e_n f_{n+1}t^k
\] in \emph{reduced fiber and abscissa form} provided $g_f=f_1e_1e_2 f_3 \cdots e_n f_{n+1} \in \pi_1(\bbX,b)\cap F$ is a reduced word. This initial subword $g_f$ is called the \emph{fibre part} and is well-defined by our normal form conventions. We record some more observations.

\begin{lem}\label{lem:DC-fibre}
  The fibre part $g_f$ of $g \in \torus\alpha$ is well-defined. Furthermore if $g,g' \in \torus\alpha$ and $g_f=g'_f$ then \[
    DC(g,e)=DC(g',e) \textrm{~and~} pow(g,e)=pow(g',e)
  \] for all $e$ with $i(e)=v$.
\end{lem}

We wish to describe how automorphisms act on the double cosets that appear in the sequences $DC(g,e)$. Even in a  free group, we shouldn't expect automorphisms to act in any sensible way on a double coset space. That said, it will turn out that the constraints of the Bass diagram will lead to a sensible action on the double coset spaces we are interested in. 

We will also develop an algorithmic theory to compute orbits of our double cosets. Our approach is to reduce the double coset orbit problems to something related to orbit problems of finitely generated subgroups of free groups (see Theorem \ref{thm:gersten}.) This next subsection establishes the dictionary between double cosets and subgroups of free groups.

    \subsubsection{Linkages of double cosets}
    Turns were defined in Subsection \ref{sec:nf}. We now define the entrance of a turn and  the linkage of a double coset at a given turn. Before continuing we will make a slight modification to our notation.
 
    \begin{conv}[Double cosets for elements of $\torus\alpha$]
        For the rest of Section \ref{sec;MWP} we will assume that the double cosets that appear in $DC(g,e)$ are actually double cosets in the fibre part $F_v = \bbX_v\cap F$. In particular, if we write $I_e\bar f O_{e'}$ then $I_e = \bk{c_e}$ and $O_{e'}=\bk{c_{\bar e'}}$, where $c_e,c_{\bar e'} \in F_v$ are, as usual, generators of the images of the corresponding incident edge groups.
    \end{conv} 
    
    We note that by Lemma \ref{lem:DC-fibre} and the fact that our ordering convention forces double coset representatives to lie in $F_v$ considering double cosets to lie in the fibre part of the vertex groups causes no loss of information.

    Whenever $f \in F_v$, and $(e, e')$ is a turn at $v$, the entrance of the turn is the element $c_e$ (it generates $I_e <\bbX_v$). We define the  linkage $\link{e, f, e'}$ and the dual linkage $(\mathscr{l}_*)_{e, f, e'}$ to be the following  subgroups of $ \bbF_v$:
        \[ \begin{array}{cclcl} \link{e, f, e'} &  = &   \langle I_e,  f O_{e'} f^{-1} \rangle & &  \\ (\mathscr{l}_*)_{e, f, e'} & =  & \mathscr{l}_{\bar e', f^{-1}, \bar e}   & = &   \langle O_{e'},  f^{-1} I_{e} f \rangle  \end{array}   \]

        If the double coset appears in the reduced form of an element $g$ in the fibre, the dual linkage is the linkage of a coset appearing in the reduced form of $g^{-1}$.

        Double cosets determine linkages, as stated in the following Lemma, which  is immediate.
        
        \begin{lem}\label{lem;same_double_coset_implies_linkages}
            If $f$ and $f'$ are in the same double coset 
            $I_e x O_{e'}$ then   $\link{e, f, e'}=\link{e, f', e'}$. 
            
            In particular, if $f$ is in the fibre of $G$, all of the reduced words representing it share the same sequence of edges, of entrances, and of linkages.
        \end{lem}
        
        However, in principle, different elements $f$ can lead to the same linkage while being in different double cosets. 
         It is thus  worth warning that a given linkage can have different duals if it can be defined by elements in different double cosets.
                
        We propose a converse to Lemma \ref{lem;same_double_coset_implies_linkages}, in which double cosets are characterized by a pair of linkages that are dual to one another.
        
        \begin{prop}\label{prop;same_linkages_imples_double_coset}
          Consider a pair of linkages $\mathscr{l} = \mathscr{l}_{e,f,e'}$ and  $\mathscr{l}'= \mathscr{l}_{e,h,e'}$ and a   dual pair $\mathscr{l}_* =  \mathscr{l}_{\bar e', f^{-1}, \bar e}$, and $\mathscr{l}_*'=  \mathscr{l}_{\bar e', h^{-1}, \bar e} $. 
          
          If one has equality of pairs   \[( \mathscr{l}, \mathscr{l}_*   ) = (\mathscr{l}', \mathscr{l}_*') \]
           then the double cosets $I_e f O_{e'}$ and $I_e h O_{e'}$ are equal.
        \end{prop} 
        
        \begin{proof}
        The result will then follow from this statement.
        
        \begin{lem}\label{lem;h_hprime_double_cosets}
          Let $\tau(e_-)= i(e_+) =v$ and, to simplify notation, let $\bk{c_{e_{-}}} = I_{e_{-}}$ and $\bk{c_{e_{+}}} = O_{e_+}.$
          
          If $h, h' \in \bbF_v$ are such that  \[\langle c_{e_-},\, h' c_{e_+} h'^{-1}  \rangle \, = \, \langle c_{e_{-}},\, h c_{e_+} h^{-1}  \rangle,  \; \hbox{ and } \; \langle c_{e_+},\, h'^{-1} c_{e_-} h'  \rangle \, = \, \langle c_{e_{+}},\, h^{-1} c_{e_-} h  \rangle      \]
          then  $h\in  \langle c_{e_{-}}\rangle   h'    \langle c_{e_{+}}\rangle$.
        \end{lem}
        
        \begin{proof}
            Start  with the equality $\langle c_{e_-},\, h' c_{e_+} h'^{-1}  \rangle \; = \; \langle c_{e_{-}},\, h c_{e_+} h^{-1}  \rangle$. 
            
            These two groups are subgroups of $\bbF_{v_i}$, hence free. 
        
        If they have rank one, then  by maximality of the edge groups, $h'c_{e_+}h'^{-1} = c_{e_{-}}^{\pm 1}= (h c_{e_+} h^{-1})^{\pm 1}$. It follows  that $h^{-1}h'$  centralizes $c_{e_+}$ (it cannot  conjugate it to its inverse). Therefore it is in $O_{e_+}$.  
        
        If they are rank two, one has that \[  h c_{e_+} h^{-1} \in I_{e_{-}}\, h' c_{e_+} h'^{-1} \, I_{e_{-}}.  \]

        Let us write  $ h c_{e_+} h^{-1} \in I_{e_{-}}\, h' c_{e_+} h'^{-1} \,c_{e_{-}}^{-r} $. Let $h'' =c_{e_{-}}^{r} h' $. We have that  $ h c_{e_+} h^{-1} \in I_{e_{-}}\, h'' c_{e_+} h''^{-1}$ and also that $\langle c_{e_{-}},\, h'' c_{e_+} h''^{-1}  \rangle \; = \; \langle c_{e_{-}},\, h c_{e_+} h^{-1} \rangle $, and we still want to show that $h\in  I_{e_{-}}   h''    O_{e_+}$.

        Let us write that, for some integer $k$,  $ h c_{e_+} h^{-1} = c_{e_{-}}^k\, h'' c_{e_+} h''^{-1}$: 
      
       \[  c_{e_{-}}^k =  h c_{e_+}\; h^{-1}h'' \; c_{e_+}^{-1} \; h''^{-1},  \]  
        
        \[ c_{e_{-}}^k =   ( [ h''^{-1}h, \; c_{e_+}])^{h''^{-1}}.\]
        
        In a free group, a non-trivial commutator is not a proper power 
        \cite{schutzenberger_equation_1959}, \cite[Lemma 36.4]{baumslag_aspects_1960}, \cite{duncan_genus_1991},  
        hence $|k|\leq 1$.  
        
        If $k=0$, then $h''^{-1} h$ commutes with $c_{e_{+}}$ hence is in $O_{e_{+}}$, which is what we wanted to show.
        
        If $|k|= 1$, then $c_{e_{-}}$ is conjugate to a commutator of an certain element with $c_{e_+}$. This is a property of the turn  $(e_{-}, e_+)$ and in this case we say that $(e_{-}, e_+)$ is a \define{dangerous turn.} 
        
        We claim that in this case, $(\bar e_+, \bar e_{-})$ (which is also a turn at the same vertex) is not dangerous. Indeed, if it was dangerous too, it would mean that $c_{e_+}$ is conjugate to a commutator of an element and $c_{e_{-}}$ too. To get the contradiction, it suffices to prove the following lemma. 
        \begin{lem}
            In a free group, if $a$ is conjugate to $[b,c]$ and $b$ is conjugate to $[a,d]$, then $a=b=1$.
        \end{lem}
        \begin{proof}
            Iwasawa and Magnus proved  that free groups are residually nilpotent, meaning that any element survives in a nilpotent quotient (see \cite[\S 6.1.9, \S 6.1.10]{robinson_course}. 
            Any such a quotient satisfies the law that all iterated commutators $[[\dots[ \cdot, \cdot ], \cdot ], \cdot \dots ]$ of certain depth must vanish. However, $a$ and $b$ can be written as commutators of arbitrary depth, by iterative substitution. 
        \end{proof}
        
        Now that we know that the turn $(\bar e_+, \bar e_{-})$ is not dangerous, we may apply  our analysis above  on the dual configuration to get that $h^{-1}\in O_{e_+}(h'_i)^{-1} I_{e_{-}}$, which is equivalent to what we wanted to show.
        \end{proof}
        \end{proof}

        \subsubsection{The linkage configuration at a vertex induced by a path and a sequence of double cosets}\label{sec;link_conf}\label{sec;linkage_config_from_DC}
        
        Given a loop  $(e_1, \dots e_n)$ based at $v \in \verts X$ and a sequence of double cosets over this loop \[( I_{\bar e_0}f_1 O_{e_1}, 1 ,I_{e_2}f_3O_{e_3}, \dots, 1 ,I_{e_n}f_{n+1} O_{e_0}),\] where $e_0$ is some polarizing edge, we can construct the associated linkage tuple \[\vec{\link{}}=(\link{1}, \dots \link{n+1}) = (\link{\bar e_0,f_1,e_1},1,\ldots,\link{e_n, f_{n+1}, e_0}).\]

        We reverse this path to get $\bar e_n,\ldots,\bar e_1$ and the corresponding \emph{dual coset tuple}\begin{multline*}
            (O_{e_0}f_{n+1}^{-1}I_{e_n}, 1 ,\ldots , O_{e_3}f_3^{-1}I_{e_2},1,O_{e_1}f_1^{-1}I_{\bar e_0})=\\
            (I_{\bar e_0}f_{n+1}^{-1}O_{\bar e_n},1,
            \ldots,I_{\bar e_3} f_3^{-1} O_{\bar e_2},1,
            I_{\bar e_1} f_1^{-1} O_{e_0})
        \end{multline*}
        and the corresponding dual linkage tuple that works out to\[
            \vec{\link{}} = ((\link *)_{n+1},\ldots,(\link *)_1),
        \] i.e. reading $\vec{\link{}}$ and $\vec{\link *}$ in opposite directions we get pairs of dual linkages of the double cosets $I_{e_{j-1}}f_j O_{e_j}.$ We note that if a sequence of double cosets came from $DC(g,e)$ then the dual sequence would come from $DC(g^{-1},e).$  By Lemma \ref{lem;same_double_coset_implies_linkages} and Proposition \ref{prop;same_linkages_imples_double_coset}, given the underlying loop and polarizing edge, the dual pair $(\vec{\link{}},\vec{\link *})$ fully determines the tuple of double cosets and vice-versa.

        Consider the two pairs\[
            (\bar e_0,e_1,\ldots,e_n,e_{n+1}), (\link{1}, \dots \link{n+1}) \textrm{~and~} (\bar e_{n+1},\bar e_n,\ldots,\bar e_1, e_0), ((\link{*})_{n+1}, \dots (\link{*})_1)
        \] with common indices $1\leq i\leq n+1$, where $e_0$ is the polarizing edge and $e_{n+1}=e_0$. Consider $v$ a vertex in $X$, and let $J_v \subset \{1, \dots, n\}$ be the set of indices $j\leq n$ such that $\tau(e_j) = v$, and $J^*_v$ the set of indices for which $\tau(\bar e_j) = v$. For each oriented edge $e$ in $X$, $J_{v, e}$  is the subset of $J_v$ for which $e_j= e$ and, similarly, ${J^*}_{v,e} = \{ j \in {J^*}_v : \bar e_j=e\}$.
    
        For each  $J_{v,e}$ (possibly empty), one associates the following $(|J_{v,e}|)$-tuple $\calP_{v,e} = (\link{{j+1}} : j\in J_{v, e})$, where the $\link{j+i}$ are entries from $\vec{\link{}}$. We note that $\calP_{v,e}$ consists of entries from $\vec{\link{}}$ of the form $\link{e,f,e'}$, where $f,e'$ are arbitrary. Similarly, we define $\calP^*_{v,e} = ((\link{*})_{j}: j\in J^{*}_{v,e})$, where the $(\link{*})_{j+i}$ are entries from $\vec{\link{*}}$. We note that entries of $\calP^*_{v,e}$ consist of linkages of the form $\link{e ,f,e'}$ where $f,e'$ are arbitrary.

        The triple $\calQ_{v,e}(g,e_0) =  (c_e, \calP_{v,e},  \calP^*_{v,e})$  is the \emph{linkage configuration of the $e_0$-polarized element $g$ at $v$ for $e$}. The following will be useful to remember in the proof of Proposition \ref{prop;same_linkages_imples_double_coset}.
        \begin{lem}\label{lem:all-e-linkages}
          Every entry $\link{e,f,e'}$ in $\vec{\link{}}$ and $\vec{\link{*}}$ that is associated to a turn $(e,e')$ at $v$ is in $\calQ_{v,e}$.
        \end{lem}
        Note that the order of appearance of each linkage is part of the data of $\calQ_{v,e}(g,e_0)$, i.e. $\calQ_{v,e}(g,e_0)$ is not just a set of linkages.

        Ordering the edges of $X$,  one thus obtains, from a linkage tuple, and for each $v$,  a tuple $\calQ_v(g,e)$ of conjugacy classes of linkage configurations at $v$: $\calQ_v(g,e_0)$ is the tuple of the conjugacy classes of the tuple  $\calQ_{v,e}(g,e_0)$, $e$ ranging over the edges of $X$ terminating at $v$, in the chosen order.  We call it the \emph{total linkage configuration} at $v$ for the given sequence of double cosets.

       \subsection{Tuples of tuples and associated data}\label{sec:tuples-of-tuples}

        \subsubsection{Elliptic, lineal and hyperbolic tuples, and their enrichments}
        
        We consider a tuple of $r-2$ elements $S=(g_2, \dots, g_r)$ of $\torus\alpha$ (we start at index $2$ on purpose).
        
        The tuple is said elliptic if there exists $\tilde v \in T$ that is fixed by all $g_i$. It is said small elliptic if $\tilde v \in \tilde \tau$. Define its enriched tuple $S^+ = (g_0, g_1, g_2, \dots, g_r)$ by setting $g_0=g_1=g_2$.

        If $S$ is not elliptic, either there exists $i$ a smallest index such that $g_i$ is not elliptic, or  all $g_i$ are elliptic and there exists $(i,j)$ a smallest pair  in lexicographical order so that $g_i$ and $g_j$ do not fix the same vertex (there exists such a pair because pairwise intersecting convex sets in trees have global intersection). Let $g_0=g_i$ in the first case and $g_0=g_ig_j$ in the second case. It is a hyperbolic element. 
        
        Again there are two cases: either all elements of $S$ commute with $g_0$ or one of them does not. 
        
        In the first case, we say that the tuple is lineal and we set its  enriched tuple $S^+ = (g_0, g_1, g_2, \dots, g_r)$ by setting $g_1=g_0$.
        
        In the second case, we say that the tuple is hyperbolic, and  we set its  enriched tuple $S^+ = (g_0, g_1, g_2, \dots, g_r)$ by setting $g_1=g_0^{g_i}$ where $i$ is the first index such that $g_i$ does not commute with $g_0$.

            \subsubsection{Small tuples}
        
        An enriched tuple  $(g_0, g_1, g_2, \dots, g_r)$ is \emph{small} if: 
        \begin{itemize}
            \item if it is elliptic:  the elements have a common  fixed point  is in $\tilde \tau$.  
            \item if it is lineal: $g_0$ is cyclically reduced in short position. 
            \item if it is hyperbolic: $g_0$ is cyclically reduced in short position, and $g_1$ has minimal syllable length among its conjugates by powers of $g_0$.
        \end{itemize}
        
        In the first case, there exists a simple path $(e_0e_1 \dots e_k)$ in $\tau$ starting at $b_0$,  such that for each $i$, there exists $h_i \in \bbX_{\tau(e_k)}$ for which  \[g_i \, =\, (e_0e_1\dots e_k) h_i (\bar e_k \bar e_{k-1}\dots  \bar e_0).\]
        
        In the second case,  $g_0$ (hence the tuple)  preserves a line $L$ in $T$ that intersects $\tau$. Let $\tilde v_0$ in this line closest to $\tilde b$, and $v_0$ its image in $X$.
        The element $g_0$ can thus be expressed as   $(e_0 e_1\dots e_k) h_0 (\bar e_k \bar e_{k-1}\dots  \bar e_0)$ for  $(e_0e_1 \dots e_k)$ a simple path  in $\tau$ from $b$ to $ v_0$, and $h \in \pi_1(\bbX,  v_0)$ is in centred short position and $v_0$ is the anchor vertex of the short position.

        From Proposition \ref{prop;compute_finitely_many_short},  we easily obtain:
        
        \begin{prop}\label{prop;compute_short}  Given a tuple, it is decidable whether it is an elliptic, lineal or hyperbolic tuple, and a small conjugate of the enriched  tuple is computable. If it is not elliptic a finite list of all small conjugates of the enriched tuple is computable. 
        \end{prop}

        \subsubsection{Centred non-elliptic tuples}\label{sec:centered-hyp}
        
        Given a small tuple $S$ of elements in $\pi_1(\bbX,b)$, we define its \emph{central vertex} $v_S \in \verts X$ as follows. If the leading entry $g_0$ is elliptic then every entry in $S$ is conjugate in $\Bass(\bbX)$ into some $\bbX_{v_T}$, so $v_T$ is the central vertex. Otherwise, if $S$ is lineal or hyperbolic, recalling the terminology of Subsection \ref{sec;short_pos}, we take $v_S$ to be the anchor of the short position of the leading entry $g_0$. If $S$ is lineal or hyperbolic we further define the \emph{polarizing edge} $e_S \in \edges X$ to be the  polarizing edge of the short position of the leading entry $g_0$.
        
        Now given a small tuple $S$ of elements in $\pi_1(\bbX,b)$ we can conjugate it in $\Bass(\bbX)$ be the element $E$ given by the path in the spanning tree $\tau$ from $b$ to $v_S$ such that $g_0 = E (g_0)_{v_S} \bar E$, where $(g_0)_{v_S}$ is in centred short position.
        
        It follows that the conjugate $\bar E S E$ is now a tuple of elements in $\pi_1(\bbX,v_S)$, its leading entry is in centred short position, and $e_S$ can be read off as the first symbol occurring in the leading entry. We call the conjugate of $\bar E S E$ the \emph{the re-centring $S$} and we say that $\bar E S E$ a \emph{centred tuple}. Proposition \ref{prop;compute_short} immediately implies that for any tuple $S$ we can compute the complete set of all its centred tuples in $\Bass(\bbX)$.
        
        The following result will have the immediate benefit of providing reassurance that taking conjugates outside $\pi_1(\bbX,b)$ is sensible.
        
        \begin{lem}\label{lem:conj-in-Bass}
          Let $g,h \in \pi_1(\bbX,v)$. $g$ and $h$ are conjugate in $\pi_1(\bbX,v)$ if and only if they are conjugate in $\Bass(\bbX)$.
        \end{lem}
        \begin{proof}
          On the one hand we have that $\pi_1(\bbX,v)\leq \Bass(\bbX)$, on the other hand, for a spanning subtree $\tau$ of $X$, we have a commuting diagram \[
            \begin{tikzpicture}[xscale=2]
                \node (B) at (0,0) {$\Bass(\bbX)$};
                \node (H) at (0,-1) {$\pi_1(\bbX,v)$};
                \node (Bbar) at (2,-1){$\Bass(\bbX)/_{\bk{\bk{\edges{\tau}}}}$};
                \draw[->>] (B) -- (Bbar);
                \path (H) -- node[sloped]{$\leq$} (B);
                \draw[->] (H) --node[above]{$\sim$} (Bbar);
            \end{tikzpicture}
          \] where the horizontal arrow is an isomorphism. If two elements in $\pi_1(\bbX,v)$ are conjugate in $\Bass(\bbX)$ then the image of the conjugator $\Bass(\bbX)/_{\bk{\bk{\edges{\tau}}}}$ gives a conjugator in $\pi_1(\bbX,v)$.

        \end{proof}
    
    \subsubsection{Centred elliptic tuples and variations}\label{sec:centered-elliptic}

    If $S = (g_0,\ldots,g_{r})$ is an elliptic tuple then its \emph{central vertex} $v_s \in \verts X$ is the vertex such all the $g_i$ are simultaneously conjugate in $\Bass(\bbX)$ into $\bbX_{v_s}$. We say $S$ is centred if its entries are in  $\bbX_{v_s}$. If $S$ is centred and all its entries are simultaneously conjugate into the image of the same edge group, then we will require the \emph{centring} of $S$ to be a conjugate in which all entries lie in the image of some common edge group and we call the \emph{variations} of a centred $S'$ to be all the tuples $(g'_0,\ldots,g'
    _r)$ such that \[
        g_i' = \ell g_i \bar \ell \in \bbX_{v_s}
    \] where $\ell$ is a loop based at $v_s$ of length at most 4. (See Lemma \ref{lem:conj-elliptic}.)
   
   We define the \emph{linkage configuration of $S$} to be the tuple \[\calQ(S) = (\bk{(g_0)_f},\ldots,\bk{(g_r)_f}) \in F\cap \bbX_{v_s}\] of cyclic subgroups of free groups generated by fibre parts.
   
   \begin{prop}
    Given an elliptic tuple, it is possible to compute its centring and all its variations.
   \end{prop}

\subsection{The actions of $\delta_1 \autfo \bbX$}\label{sec;action_delta_1}

The mixed Whitehead problem involves the action of $\delta\autfo(\bbX)$ as well as the action of $(\torus\alpha)^k$ on a tuple of tuples. We start by studying the action of $\delta\autfo{\bbX}$ on centred tuples.

\subsubsection{Preliminaries}
            
    Recall that $\delta \autfo \bbX$ surjects on $\outfo \G$ and that by Corollary \ref{cor:delta_1_repr} we can compute a complete list $\nu = \{\nu_1,\ldots, \nu_v\}$ of $\delta_1\autfo\bbX$-coset representatives. If for each $j\in \{1, \dots, v\} $,  we define $\vec S_j = \nu_j(\vec S) = (\nu_j(S_1),\ldots,\nu_j(S_k))$, then we immediately get a first reduction, before diving into the main argument:

            \begin{lem}\label{lem:variations} \label{lem;reduc_delta_1}
                $\vec S$ and $\vec T$ are mixed Whitehead equivalent under $\autfo \G$  if and only if there is     
                $j\in \{1, \dots, v\}$
                such that      $\vec S_j$ and $\vec T$ are Whitehead equivalent under $\delta_1 \autfo \bbX $.   In particular, the mixed Whitehead  problem under $\autfo \G$ reduces to the mixed Whitehead problem under $\delta_1 \autfo \bbX $.
            \end{lem}
        
        Observe that $\delta_1 \autfo \bbX$ preserves the paths (hence the syllable lengths), and the abscissa. Note that, according to our convention,  these paths are to be understood as paths in $X$, not in $T$. 
        
        \begin{lem}
           If $\alpha \in \delta_1\autfo \bbX$ is seen as a homomorphism of the Bass group, then it preserves the $X$-paths and abscissa of elements in $\pi_1(\bbX,v)$ for all $v \in \verts X$. If $w$ is reduced, then $\alpha(w)$ is reduced and of the same syllable length. If $g$ is cyclically reduced, then $\alpha(g)$ is cyclically reduced.
        \end{lem}
 
        \begin{proof} 
        Any fibre and orientation-preserving automorphism induces the identity on the cyclic quotient of the suspension, hence  preserves the abscissa. 
        
        Being in $\delta_1\autfo \bbX$, $\alpha$ is of the form $\alpha = (Id_X, (\phi_v), (Id_e), (\gamma_e))$. It sends each edge $e_i$ to $\gamma_{\overline {e_i}}^{-1} e_i \gamma_{e_i} $  in the Bass group. In particular, it preserves the paths.  
        
        Also, it sends the word $g_1e_1g_2\cdots e_ng_{n+1}$ to $h_1e_1h_2e_2\cdots e_nh_{n+1}$ where  $h_i =    \gamma_{e_{i-1}}  \phi_{v_i}(g_i) \gamma_{\overline {e_{i}}}^{-1} $, if $1< i < n+1$. 
        In particular, it  is in the same vertex groups as  $g_i$.
         If the word  $h_1e_1h_2e_2\cdots e_nh_{n+1}$  fails to be reduced, there exists $i$ in $\{2, \dots,  n-1\}$ such that $e_{i+1}= \bar e_i$, and  $h_i \in \tau_{e_i}(\bbX_{e_i}) < \bbX_{v_{i+1}}$. The Bass diagram implies that $g_1e_1g_2\cdots e_ng_{n+1}$     is not reduced either. 
         
         If $g$ is cyclically reduced, consider $\alpha(g)$, and assume by contradiction that it  is not cyclically reduced. There is a shorter conjugate, and its image by $\alpha^{-1}$ (also in $\delta_1\autfo \bbX$, inducing an isometry of the Bass-Serre tree),  
         is a conjugate of $g$ that is represented by a word that is reduced by the previous argument, and shorter: a contradiction.  
        \end{proof}
 
    \subsubsection{Action on linkages, and on double cosets}
        \begin{conv}
            For the remainder of this section, we assume any tuple $T$ of elements is  centred and its elements lie in $\pi_1(\bbX,v_T)$ where $v_T$ is the central vertex of the centred tuple. The leading entry of $T$ is in centred short position and the first edge in its prefix is $e_T$. We will always use $e_T$-polarized normal forms to work with the element of $T.$
        \end{conv}

    Let us compute the action  of $\delta_1 \autfo \bbX$ on the linkage tuples. Observe that it is equivalent to consider $\delta_1 \aut \bbF$. First of all, we will want to fix an edge $e$ and consider $e$-polarized normal forms. Obviously, this only makes sense for elements of $\pi_1(\bbX,v)$ with $i(e)=v$. Given $g \in \pi_1(\bbX,i(e))$ and $\phi=(Id_X, (\phi_v), (Id_e), (\gamma_e))$ an element of $\delta_1 \autfo \bbX$ we define the \emph{$e$-polarized action of $\phi$} to be \[
        \phi\cdot_e g = (\ad{\gamma_{\bar e}^{-1}}\circ \phi)(g),
    \] where for this section $\ad{x}(y) = x^{-1}yx$. Note that $(\ad x \circ \ad y) (z) = \ad x (y^{-1}zy) = x^{-1}y^{-1}z y x = \ad{yx}(z)$, which due to the fact that our conjugation convention gives a right action but functions act on the left.
    
    \begin{lem}
      The $e-$polarized action is a well-defined group action.
    \end{lem}
    \begin{proof}
      Let $\phi=(Id_X, (\phi_v), (Id_e), (\alpha_e))$, $\psi=(Id_X, (\psi_v), (Id_e), (\beta_e))$  and let $\theta = \psi\circ\phi = (Id_X, (\psi_v), (Id_e), (\gamma_e))$. On the one hand, we have\[
        \psi(\phi(e))=\psi(\alpha_{\bar e}^{-1})
        \beta_{\bar e}^{-1}
        e \beta_e \psi(\alpha_e) \Rightarrow \gamma_{\bar e}^{-1} = \psi(\alpha_{\bar e}^{-1})\beta_{\bar e}^{-1}.
      \] (See also the Bass diagram \eqref{eq:delta_0-composition}.)  On the other hand, we have\[
        (\ad{\beta_{\bar e}^{-1}}\circ\psi)\circ(\ad{\alpha_{\bar e}^{-1}}
        \circ \phi) = \ad{\beta_{\bar e}^{-1}}\circ\ad{\psi(\alpha_{\bar 
        e}^{-1})} \circ \psi\circ\phi = \ad{\psi(\alpha_{\bar 
        e}^{-1})\beta_{\bar e}^{-1}} \circ \psi\circ\phi.
      \] It follows that\[
        (\psi\circ \phi)\cdot_e (g) = (\ad{\gamma_{\bar e}^{-1}}\circ \psi\circ \phi)(g) = \psi\cdot_e(\phi\cdot_e g),
      \] as required.
    \end{proof}

    \begin{lem}\label{lem;image_of_reduced}
        Let $g\in \pi_1(\bbX, v)$, and let its reduced fiber form be \[g=\bar f_1e_1 \bar f_2\cdots e_nf_{n+1} t^k,\] where $t$ is our oriented generator of the centre of $\bbX_v$ and write $v_i = \tau(e_{i-1}) = i(e_i)$, so that $f_i \in \bbX_{v_i}\cap F$.
        Let $\phi=(Id_X, (\phi_v), (Id_e), (\gamma_e))$ an element of $\delta_1 \autfo \bbX$. Let $e$ be such that $i(e)=v$.
        The image of $\phi\cdot_e g$  has the same $X$-path and the same sequence of turns as $g$. 
        
        The linkage tuple of $\phi \cdot_e g$ with respect to the $e$-polarization has $i$-th coordinate that is equal to $ (\phi_{v_i} ( \langle I_{e_{i-1}}, f_i O_{e_i} f_i^{-1} \rangle ))^{ \gamma_{e_{i-1}}^{-1}}      $, which is $\phi_{v_i} ( \vec{\link{}}(g)_i)^{\gamma_{e_{i-1}}^{-1}} $. 
          
        \end{lem}
        
        \begin{proof}
        First note that $\phi(t) = t$ since $\phi$ is orientation preserving and $t$ is a generator of the centre of $\bbX_{b}$. The equality of the paths is immediate, given that $\phi$ induces $Id_X$, hence the sequence of turns is the same too.

        We write 
        \begin{multline}\label{eqn:phieofg}
        \phi\cdot_e g = \gamma_{\bar e} \phi(g) \gamma_{\bar e}^{-1} =  \gamma_{\bar e} \phi_{v_1}(f_1) \gamma_{\bar e_1}^{-1} \cdot e_1 \cdot \gamma_{e_1}\,\phi_{v_2} (f_2) \gamma_{\bar e_2}^{-1} \cdot e_2 \cdots\\ e_n \cdot \gamma_{e_n}\cdot\phi_{v_{n+1}} (f_{n+1}) \gamma_{\bar e}^{-1}  \, t^{\pm k}
        \end{multline}
         
        Considering the sequence $DC(\phi\cdot_e g,e)$ we see that we can express $\vec{\link{}}(\phi(g))_i$ as 
        \[\vec{\link{}}(\phi(g))_i = \langle I_{e_{i-1}}, \gamma_{e_{i-1}} \phi_{v_i} (f_i) \gamma_{\bar e_i}^{-1}\, \cdot \, O_{e_i} \,\cdot \,  \gamma_{\bar e_i}     \phi_{v_i} (f_i)^{-1} \gamma_{e_{i-1}}^{-1}    \rangle, \] for $1\leq i \leq n+1$ and where, for the sake of defining cosets in the context of $e$-polarization, we define $e_0=\bar e$ and $e_{n+1}=e$. The Bass diagram (for $e_{i-1}$ and $\bar e_i$) ensures:
        \begin{equation}\label{eq:bass-diag-oeie}
        \begin{array}{rclr}
             \gamma_{\bar e_i}^{-1} O_{e_i}  \gamma_{\bar e_i} & = & \phi_{v_i} (O_{e_i}) &  (\hbox{image of } O_{e_i}) 
             \\  I_{e_{i-1}} & = &  \gamma_{e_{i-1}} \phi_{v_i} ( I_{e_{i-1}})\gamma_{e_{i-1}}^{-1} &  (\hbox{image of } I_{e_{i-1}}).
        \end{array}
        \end{equation}
        
        It follows that  \begin{multline*}\langle I_{e_{i-1}}, \gamma_{e_{i-1}} \phi_{v_i} (f_i) \gamma_{\bar e_i}^{-1} O_{e_i}  \gamma_{\bar e_i}     \phi_{v_i} (f_i)^{-1} \gamma_{e_{i-1}}^{-1}    \rangle = \\
        \langle \phi_{v_i} (I_{e_{i-1}}),  \phi_{v_i} (f_i) \phi_{v_i} (O_{e_i})       \phi_{v_i} (f_i)^{-1}    \rangle^{\gamma_{e_{i-1}}^{-1}}, \end{multline*} which is what we wanted. 
        
          \end{proof}

            For all vertex $v$ and edges $e_-, e_+$ such that $\tau(e_-) = i(e_+)=v$,  we consider the following action of $\delta_1 \autfo \bbX$ on the set  $I_{e_-} \backslash \bbF_v/ O_{e_+}$ of double cosets   $I_{e_-} f O_{e_+}$   in $\bbF_v$.
            
            \begin{conv}\label{conv;action_DC}
                Let  $\phi= (Id_X, (\phi_v), (Id_e), (\gamma_e)) \in \delta_1 \autfo \bbX$, let  $w\in \verts X$ be a white vertex, and let $e_-, e_+$ edges such that $\tau(e_-) = i(e_+)=w$.   Let   $f\in \bbF_{v}$. and consider  the set   of  double cosets  $I_{e_-} f O_{e_+}$   in $\bbF_v$. We set  
                \[\phi \cdot  (I_{e_-} f O_{e_+}) \; = \;  I_{e_-} \, \gamma_{e_-} \phi_v(f) \gamma_{\bar e_+}^{-1}  \, O_{e_+}. \]

            \end{conv}        
            One easily checks from the composition formula in Section \ref{sec;autom_gog} that this defines an action of $\delta_1 \autfo \bbX$  on  $I_{e_-} \backslash \bbF_v/ O_{e_+}$. By Lemma \ref{lem;image_of_reduced},  it is the natural action of the thus obtained automorphisms of $\torus \alpha$ on the double cosets defined by polarized reduced forms:
            
            \begin{prop}\label{prop;orbit_DC_determined_by_LC}
                Consider $g \in \pi_1(\bbX,v)$ which, in $e$-polarized normal form, is written $g=f_1e_1f_2\cdots e_nf_{n+1} t^k$,  and $\phi \in \delta_1 \autfo \bbX$. For all $j$, denote $v_j=i(e_j)$. Then for all $i$, there exists $f'_i \in \bbF_{v_i}$ such that  \[\phi\cdot_e g = (\ad{\gamma_{\bar e}^{-1}}\circ \phi (g))=f'_1e_1f'_2e_2\cdots e_nf'_{n+1} t^k,\] and the double coset $I_{e_{i-1}} f'_i O_{e_i}$ is $\phi\cdot (I_{e_{i-1}} f_i O_{e_i})$. Equivalently if $DC(g,e)=(I_{\bar e}f_1 O_{e_1},1,\ldots,1, I_{e_n} f_{n+2} O_{\bar e})$ then \[\phi\cdot DC(g,e)=DC(\phi\cdot_e g,e) = \left(\phi\cdot(I_{\bar e}f_1 O_{e_1}),1,\ldots,1, \phi\cdot(I_{e_n} f_{n+2} O_{e})\right).\]
         
            \end{prop}
            \begin{proof}
            The expression of $\phi\cdot_e g$ has already been computed in \eqref{eqn:phieofg} in the proof Lemma \ref{lem;image_of_reduced}, where setting $e_0=\bar e$ and $e_{n+1}=e$ we have 
             \[\phi\cdot e(g) = \gamma_{e_0}\phi_{v_1}(f_1) \gamma_{\bar e_1}^{-1} \cdot e_1 \cdot \gamma_{e_1} \phi_{v_2} (f_2) \gamma_{\bar e_2}^{-1} \cdot e_2 \cdots e_n \cdot \gamma_{e_n} \phi_{v_{n+1}} (f_{n+1})\gamma_{\bar e_{n+1}}^{-1} \; t^{\pm k}.  \] We further note that in the word above the factors $(\gamma_{e_{2i-1}}\gamma_{\bar e_{2i}}^{-1}), i=1,\ldots,n/2$ that appear in the abelian vertex groups can be migrated through either $e_{2i-1}$ or $e_{2i}$ without affecting the double coset. The proposition now follows.
            \end{proof}
        
        \subsubsection{The action of automorphisms on sequences of double cosets, and a linkage configuration}
        
        \begin{prop}\label{prop;sending_DC_tuple} 
        Consider two double coset sequences $DC(g,e_0)$ and $DC(g',e_0)'$ over the same $X$-path with polarizing edge $e_0$. Consider the linkage configurations induced for each vertex and oriented edge in $X$ by these sequences, $\calQ_{v,e}=\calQ_{v,e}(g,e_0)$ and $\calQ'_{v,e}(g',e_0)$.

        There exists an automorphism $\phi =(Id_X, (\phi_v), (Id_e), \gamma_e) \in \delta_1 \autfo \bbX$ that sends $DC(g,e)$ to $DC(g',e)'$ (for the action defined by Convention \ref{conv;action_DC} and as in Proposition \ref{prop;orbit_DC_determined_by_LC})  if, and only if,  
         for each $v \in \verts X$ and each $e \in \edges X$, there exists a $\phi_v$ and $\gamma_e$,  such that $\phi_v(\calQ_{v,e}) = (\calQ'_{v,e})^{\gamma_e}$. 
         
         \end{prop}
        
        \begin{proof} 
        
        Assume $\phi=(Id_X, (\phi_v), (Id_e), (\gamma_e))$ sends  $DC(g,e)$ to $DC(g',e)$. 
        Consider a linkage $\mathscr{l}$ that is associated, or dually associated, to a double coset $I_{e_-} f O_{e_+}$ in $DC(g,e)$, and $\mathscr{l}'$ the corresponding linkage for the image $I_{e_-} f' O_{e_+} $ by $\phi$ of this double coset.
        
        Write $\mathscr{l} = \langle c_{e_-} , f c_{e_+} f^{-1} \rangle $, and $\mathscr{l}'= \langle c_{e_-} , f' c_{e_+} f'^{-1} \rangle $. We want to show that $\phi_v (\mathscr{l})$ is conjugate to $\mathscr{l}'$ by an element that only depends on $\phi$ and $e_-$.
        
        We know (by Proposition \ref{prop;orbit_DC_determined_by_LC}) that $\phi \cdot I_{e_-} f O_{e_+} = I_{e_-} f' O_{e_+} $, but it is also, by definition of the action  $I_{e_-}  \gamma_{e_-}  \phi_v(f)   \gamma_{\bar e_+}^{-1}  O_{e_+}$. 
        
        It follows that  $\mathscr{l}'= \langle c_{e_-} ,
          \gamma_{e_-}  \phi_v(f)   \gamma_{\bar e_+}^{-1}  c_{e_+}  (\gamma_{e_-}  \phi_v(f)   \gamma_{\bar e_+}^{-1})^{-1}\rangle      $.

        
        By the Bass diagrams (as recorded in \eqref{eq:bass-diag-oeie} in the proof of Lemma \ref{lem;image_of_reduced}), we also know that  $\phi_v (\mathscr{l})=   \langle c_{e_-}^{\gamma_{e_-}} ,  \phi_v(f) c_{e_+}^{\gamma_{\bar e_+}} \phi_v(f)^{-1} \rangle $. 
        
        This means that $\phi_v (\mathscr{l})=  \langle c_{e_-} ,  \gamma_{e_-} \phi_v(f) c_{e_+}^{\gamma_{\bar e_+}}  \phi_v(f)^{-1}  \gamma_{e_-}^{-1} \rangle^{\gamma_{e_-}}$, which is a conjugate of 
         $\mathscr{l}'$ by $\gamma_{e_-}$, as required.
        
        We now show the converse. Suppose we are given automorphisms $\phi_v$ for each vertex of $X$ and an element $\gamma_e$ for each edge of $X$. By hypothesis, 
         $\phi_v(\calQ_{v,e}) =  (\calQ_{v,e}')^{\gamma_e}$. 
        Observe  the first coordinate of $\calQ_{v,e}$ and  $\calQ_{v,e}'$: one has   $\phi_v(c_e)= c_e^{\gamma_e}$.  This ensures that the tuple $(Id_X, (\phi_v), (Id_e), (\gamma_e))$ satisfies the Bass diagrams for all edges. This therefore defines an automorphism $(Id_X,(\phi_v),Id_e,\gamma_e)\in \delta_1 \autfo \bbX$. 
        
        Consider the $j$-th double coset in the tuple $DC(g,e)$, and denote it $I_{e_-} f O_{e_+} $. and denote by   $I_{e_-} f' O_{e_+} $  the $j$-th double coset of $DC(g',e)$ (observe that it must be a double coset for the same pair of groups because the underlying path is the same).
        
        We want to show that  $\phi \cdot  I_{e_-} f O_{e_+}$ is  equal to $I_{e_-} f' O_{e_+} $.  Recall that by definition  $\phi\cdot I_{e_-} f O_{e_+} = I_{e_-}  \gamma_{e_-}  \phi_v(f)   \gamma_{\bar e_+}^{-1}  O_{e_+} $. So we immediately rephrase our goal as to show equality between  
        
        \[  I_{e_-}  \gamma_{e_-}  \phi_v(f)   \gamma_{\bar e_+}^{-1}  O_{e_+} \; \hbox{ and } \;  I_{e_-} f' O_{e_+}.  \]

         The double coset $I_{e_-} f O_{e_+} $ defines the two linkages $\mathscr{l}_j$ and $(\mathscr{l}_*)_j$, which, with the notations of Section \ref{sec;linkage_config_from_DC} are respectively part of     
         the linkage configurations $\calQ_{v,e_-}$ and $\calQ_{v, \bar e_+}$ (see Lemma \ref{lem:all-e-linkages}.) 
        
        Let $\mathscr{l}'_j, ({\mathscr{l}'}^*)_{j}$ be the (mutually dual) linkages of $I_{e_-} f' O_{e_+} $ for the corresponding index for $DC(g',e)$.

        By assumption on the preservation of linkage configurations, $\phi_v(\mathscr{l}_j) =  (\mathscr{l}'_j)^{\gamma_{e_-}}$ and  $\phi_v(\mathscr{l}^*_j) =  ({\mathscr{l}'}^*_j)^{\gamma_{\bar e_+}}$.

        Let us write these equalities in more detail. Recall that $\mathscr{l}_j = \langle c_{e_-},  f c_{e_+}f^{-1}\rangle $, and that, by the Bass diagrams, we have   $ \phi_v(c_{e_-}) = c_{e_-}^{\gamma_{e_-} }$, and   $\phi(c_{e_+}) = c_{e_+}^{\gamma_{\bar e_+} }$. 
        
        We have 
        
         \[\begin{array}{rcl}
        
         \phi_v(\mathscr{l}_j) & = &  \langle c_{e_-}^{\gamma_{e_-}},  \phi_v(f) c_{e_+}^{\gamma_{\bar e_+}} \phi_v(f)^{-1} \rangle \\  
         & = &  \langle c_{e_-}, \;    \gamma_{e_-} \phi_v(f)\gamma_{\bar e_+}^{-1} \,  c_{e_+} \,{\gamma_{\bar e_+}} \phi_v(f)^{-1} \gamma_{e_-}^{-1} \rangle^{\gamma_{e_-}} \end{array}  \]     
        
         However, as already mentioned,  $\phi_v(\mathscr{l}_j) = (\mathscr{l}'_j)^{\gamma_{e_-}}$.  Therefore \[\mathscr{l}'_j \; = \;  \langle c_{e_-}, \;    \gamma_{e_-} \phi_v(f)\gamma_{\bar e_+}^{-1} \,  c_{e_+} \,{\gamma_{\bar e_+}} \phi_v(f)^{-1} \gamma_{e_-}^{-1} \rangle. \]
        
        In the same way, considering the duals, one checks that 
        \[{\mathscr{l}^*}'_j \; = \;  \langle c_{\bar e_+},    \;\gamma_{\bar e_+} \phi_v(f)^{-1} \gamma_{e_-}^{-1} \,  c_{e_-} \, {\gamma_{e_-}} \phi_v(f) \gamma_{\bar e_+}^{-1} \rangle.  \]
        
        It remains to apply Proposition \ref{prop;same_linkages_imples_double_coset} to obtain, as desired,  that 
         \[  I_{e_-}  \gamma_{e_-}  \phi_v(f)   \gamma_{\bar e_+}^{-1}  O_{e_+} \;  =   \;  I_{e_-} f' O_{e_+}.  \]
        \end{proof}

    \subsubsection{Orbits in vertices I: subgroups,  linkages}
        
        Suppose we are given two centred tuples of tuples $\vec S=(S_1,\ldots,S_k)$ and $\vec T=(T_1,\ldots, T_k)$. Denote by $e_i$ the polarizing edge $e_{S_i}$ and $v_i$ the central vertex $v_{S_i}$. Consider the action\begin{equation}\label{eq:tupletuple-action}
            \left.
            \begin{array}{l}
            \phi\cdot \vec S = (\phi\cdot S_1,\ldots,\phi\cdot S_k)\\
            \phi\cdot S_i = \begin{cases} 
            (\phi\cdot_{e_i} S_{i1},\ldots,\phi\cdot_{e_i} S_{il_i}) 
             & \text{if  $S_i$ is not elliptic,}\\
            (\phi_{v_i}(S_{i1}),\ldots,\phi_{v_i}(S_{il_i})) 
             & \text{if $S_i$ is elliptic.}
            \end{cases}
            \end{array}
            \right.
        \end{equation}

        Our first goal is to decide if this action can bring all the double cosets in $\vec{DC}(S_i)$ to $\vec{DC}(T_i)$ for $1\leq i \leq k$ for the $S_i,T_i$ that are non-elliptic, where $$\vec{DC}(T_i) = DC(T_{i1},e_{T_i})\odot \cdots \odot DC(T_{il_i},e_{T_i}),$$ where $\odot$ denotes concatenation of tuples.

        We recall the following theorem of Gersten improving the classical Whitehead algorithm.

                \begin{thm}[{See \cite[Theorems W and M]{gersten_whiteheads_1984}}]\label{thm:gersten}
                    Let  $\bbF$ be a finitely generated free group, and $A,B$ be two tuples of conjugacy classes of finitely
                    generated subgroups of $\bbF$ then,
                    \begin{enumerate}
                        \item\label{it:gersten-orbit-G} There is an effective procedure for determining if there is
                        some $\alpha \in \aut{\bbF}$ such that $\alpha A = B$.
                        \item\label{it:stab-pres-G} The stabilizer of $B$ in $\aut{\bbF}$ is finitely
                        presented and a finite presentation can be effectively
                        determined.
                    \end{enumerate}
                \end{thm}

                We need the following refinement, in which the conjugators for certain subgroups are required to be equal.
                
                \begin{thm}[Mixed whitehead problems for subgroups of free groups] \label{thm;gersten_hack}
                    Let $\bbF$ be a finitely generated free group, and $$A=([A_1], \dots [A_r]), B=([B_1], \dots, [B_r])$$ be two tuples of \emph{ conjugacy classes of tuples of} finitely
                    generated subgroups of $\bbF$ then,
                    \begin{enumerate}
                        \item\label{it:gersten-orbit} There is an effective procedure for determining if there is
                        some $\alpha \in \aut{\bbF}$ such that, for all $i$,   $[\alpha A_i]=[B_i]$ (as conjugacy classes of tuple of subgroups).
                        
                        \item\label{it:stab-pres} The stabilizer of $B$ in $\aut{\bbF}$ is finitely
                        presented and a finite presentation can be effectively
                        determined.
                    \end{enumerate}
                \end{thm}

                Due to its  specific independence with the current notations and objects,  we differ the proof of Theorem \ref{thm;gersten_hack} to Section \ref{sec;hack}. We finish now our  study of the Mixed Whitehead Problem assuming that Theorem \ref{thm;gersten_hack} is established.

                \begin{cor}[Aligning coset tuples and elliptic tuples]\label{cor;DC-tuples-of-tuples}
                  Given two centred tuples of tuples $\vec S=(S_1,\ldots,S_k)$ and $\vec T=(T_1,\ldots, T_k)$, it is decidable whether there exists some $\phi \in \delta_1\autfo\bbX$ such that\[
                    \vec{DC}(\phi\cdot S_i) = \vec{DC}(T_i)
                  \] for $i=1,\ldots,k$ and $S_i,T_i$ are non-elliptic, and such that for elliptic tuples \[
                      [\phi\cdot S_i]_{v_i}=[Ti]_{v_i}
                  \]
                  where the action is as defined in \eqref{eq:tupletuple-action} and where $[S_i]_{v_i}$ denotes conjugacy classes in $\bbX_{v_i}$. Furthermore, a finite generating set of the stabilizer $A_{\vec T}\leq\delta_1\autfo\bbX$ of all the tuples $\vec{DC}(T_i)$ and $[T_j]$ for $T_i$ non-elliptic and $T_j$ elliptic (respectively) can be computed.
                \end{cor}
            \begin{proof}
                
                This problem is solved by using an instance of the mixed Whitehead Problem for subgroups for each free group $F_v, v \in \verts X$. For a given vertex $v$ the tuple of tuples $\calT_v(\vec S)$ for a tuple of tuples $\vec S$ is constructed as follows:
                \begin{enumerate}
                    \item Let $I_v$ denote the indices $i$ such that $S_i$ is elliptic and the central vertex $v_{S_i}=v$, and let $J$ denote the indices of the non-elliptic tuples.
                    \item Let $\calE_v = (\calQ(S_i):i \in I_v)$, be the tuple of linkage configurations of elliptic tuples centred at $v$.
                    \item Denote by\[
                        \calH_{v,e}(S_i) = \bigodot_{j=1}^{l_j}\left( \calP_{v,e,j}\odot\calP_{v,e,j}^*\right)
                    \] where $\calQ_{v,e}(S_{ij},e_j) = (c_e,\calP_{v,e},\calP_{v,e}^*)$ is the linkage configuration for the entry $S_{ij}$ of $S_i$ and $\odot$ denotes \emph{concatenation of tuples.}
                    \item Set $\calH_{v,e} = (c_e) \odot\left(\bigodot_{i \in J} \calH_{v,e}(S_i)\right)$ then make the tuple of tuples of subgroups \[\calH_v = (\calH_{v,e}: e \in \edges X).\]
                    \item The required tuple of tuples\[
                        \calT_v = \calE_v \odot \calH_v.    
                    \]
                \end{enumerate}
                
                By Theorem \ref{thm;gersten_hack} the Mixed whitehead problem for tuples of tuples of subgroups of $F_v$ is decidable so we can decide $\calT_v(\vec S)$ and $\calT_v(\vec T)$ are equivalent via some $\phi_v \in \aut{F_v}$. The issue remains that even we preserved the conjugacy classes of subgroups $\bk{c_e}$ or entries $\bk{(T_{ij})_f}$ in $\phi_v(\calE_v)$, we may have sent the certain of the $c_e$ or $(T_{ij})_f$ to conjugates of their inverses. However, because we can compute a finite generating set of the stabilizer of $\calT_v(\vec T)$, we can determine if there is an element of that stabilizer that also preserves the sign of those elements. Once we have done this and amended $\phi_v$ we can then compute a generating set of the finite index subgroup of the stabilizer of $\calT_v(\vec T)$ that also preserves the signs of individual elements.
                
                The result now follows by applying Proposition \ref{prop;sending_DC_tuple} for corresponding pairs of non-elliptic tuples and by the definition of the Mixed Whitehead problem applied corresponding pairs of elliptic tuples.
            \end{proof}             
             
     \subsubsection{Orbits in vertices II: actions on power tuples.}\label{sec:power-tuples}

                Fix an $X$-loop $\rho=e_1,\ldots,e_n$ based at $v$ a polarizing edge $e$ with $i(e)=v$ and a sequence of cosets $D = (I_{\bar e} f_1 O_{e_1},1,I_{e_2}f_3O_{e_3},\ldots,I_{e_n}f_{n+1}O_{e})$. We call $\calE_{(D,p,e)} = \mathbb Z^{n/2+3}$ the \emph{embedding space} for $(D,\rho,e)$. There is a bijection from $\calE_{(D,\rho,e)}$ to the set:
                \[ \calG(D,\rho,e)=\{ g \in \pi_1(\bbX,v)| g \textrm{~has underlying $X$-path $\rho$ and~} DC(g,e)=D\}\]
                given by the inverse maps:
                \begin{eqnarray*}
                    (d_0,d_2,\ldots,d_{n+1},k) & \mapsto & c_{\bar e}^{d_0}\bar f_1 e_1e_2 c_{e_2}^{d_2}\bar f_3 \cdots e_nc_{e_n}^{d_n} \bar f_{n+1} c_{\bar e}^{d_{n+1}} t_b^k\\
                    \mathrm{pow}(g,e) & \mapsfrom & g
                \end{eqnarray*}
                Where the $\bar f_i$ are the coset representatives from $DCR(g,e)$. This next lemma justifies the use of polarized actions.
                
        \begin{lem}[Prefix-short preservation]\label{lem:bonus}
        If $g \in \pi_1(\bbX,i(e_1))$ is centred $e_1e_2$-prefix-short and $\phi=(Id_X, (\phi_v), (Id_e), (\gamma_e)) \in \delta_1 \autfo \bbX$. Then $\phi\cdot_{e_1} g$ is also centred $e_1e_2$-prefix-short. In other words the $e_1$-polarized action of $\delta_1 \autfo \bbX$ acts on the set of centred $e_1e_j$-prefix-short elements.
        \end{lem}
        \begin{proof}
        Let $g = e_1e_2f_2\cdots$ and applying $\phi$ to each symbol gives\[
        \phi\cdot_{e_1} g = \gamma_{\bar e} \phi(g)\gamma_{\bar e}^{-1}
            = \gamma_{\bar e_1}( \gamma_{\bar e_1}^{-1}e_1\gamma_{e_1}\gamma_{\bar{e_2}}^{-1}e_2\gamma_{e_2}\phi_{v_2}(f_2) \cdots )\gamma_{\bar e_1}^{-1}. 
       \] Since $v_1=\tau(e_1)$ is abelian and coincides with the images of incident edge groups we can migrate $\gamma_{e_1}\gamma_{\bar{e_2}}^{-1}$ across $e_2$ to get\[
       \phi\cdot_{e_1} g = e_1e_2f_2'\cdots,
       \] as required.
       \end{proof}

    \begin{cor}\label{cor:prefix-short-pres}
      For any $\phi \in \delta_1\autfo\bbX$ and $g \in \pi_1(\bbX,i(e))$, $g$ is $ee_j$-prefix short, for some other $e_j \in \edges X$ if and only if $\phi\cdot_e g $ is $ee_j$-prefix short.
    \end{cor}

    With this detail established we now have:
    \begin{prop}\label{prop:power-moves}
      Let $D$ be a tuple of cosets, and let $A_D \leq \delta_1\autfo\bbX$ be a subgroup of automorphisms that preserves the cosets in $D$ via the action in Convention \ref{conv;action_DC}. Then $A_D$ acts on $\calG(D,\rho,e)$ and also, viewing $\calE_{D,\rho,e}$ as an abelian group, there is a homomorphism \[\epsilon_{D,\rho,e}: A_D \to \calE_{D,\rho,e}\] such that \[
        \mathrm{pow}(\phi\cdot_e g,g) = \mathrm{pow}(g,e)+\epsilon_{D,\rho,e}(\phi).
      \]
    \end{prop}
    \begin{proof}
      The proof amounts to applying $\phi\cdot_e$ and keeping track of exponents. Applying $\phi$ to every symbol and making obvious cancellations gives:
      \begin{multline*}
        \phi\cdot_e g=  \phi\cdot_e (c_{\bar e}^{d_0}\bar f_1 e_1e_2 c_{e_2}^{d_2}\bar f_3 \cdots e_nc_{e_n}^{d_n} \bar f_{n+1} c_{\bar e}^{d_{n+1}} t_b^k)\\
        = \gamma_{\bar e} \phi(c_{\bar e}^{d_0})\phi(\bar f_1)\phi(e_1)\phi(e_2) \cdots 
        \phi(e_n)\phi(c_{e_n}^{d_n}) \phi(\bar f_{n+1})\phi(c_{\bar e}^{d_{n+1}})\phi( t_b^k)
        \gamma_{\bar e}^{-1}\\
        = \gamma_{\bar e} \gamma_{\bar e}^{-1}c_{\bar e}^{d_0}\gamma_{\bar e}\phi(\bar f_1)\gamma_{\bar e_1}^{-1}e_1 \gamma_{e_1} \gamma_{\bar e_2}^{-1}e_2 \gamma_{e_2} \cdots\\ 
        \gamma_{\bar e_n}^{-1}e_n\gamma_{e_n}\gamma_{e_n}^{-1}c_{e_n}^{d_n} \gamma_{e_n} \phi(\bar f_{n+1})\gamma_{\bar e}^{-1} c_{\bar e}^{d_{n+1}} \gamma_{\bar e}\gamma_{\bar e}^{-1}t_b^k\\
        =c_{\bar e}^{d_0}(\gamma_{\bar e}\phi(\bar f_1)\gamma_{\bar e_1}^{-1})e_1 (\gamma_{e_1} \gamma_{\bar e_2}^{-1})e_2 \gamma_{e_2} \cdots
        \gamma_{\bar e_n}^{-1}e_nc_{e_n}^{d_n} (\gamma_{e_n} \phi(\bar f_{n+1})\gamma_{\bar e}^{-1}) c_{\bar e}^{d_{n+1}}t_b^k
      \end{multline*}
    Now note that for every odd index $i$, noting that the vertex group $\bbX_{\tau(e_i)}$  must be abelian and coincides with images of incident edge groups we have
    \[e_i(\gamma_{e_i}\gamma_{\bar e_{i+1}}^{-1}) e_{i+1} = e_{i}e_{i+1}c_{e_i+1}^{m_i},\]for some $m_i \in \bbZ$. This gives 
    \begin{multline*}
        \phi\cdot_e g = c_{\bar e}^{d_0}(\gamma_{\bar e}\phi(\bar f_1)\gamma_{\bar e_1}^{-1})e_1 e_2 c_{e_2}^{m_2} \cdots
        e_n c_{e_n}^{m_n}c_{e_n}^{d_n} (\gamma_{e_n} \phi(\bar f_{n+1})\gamma_{\bar e}^{-1}) c_{\bar e}^{d_{n+1}}t_b^k
    \end{multline*}
    
    Finally note that our assumption that $\phi$ preserves $D$ implies that since $\bar f_{i+1} \in \bk{c_{e_i}}\bar f_{i+1} \bk{c_{\bar e_{i+1}}}$ then we must have that $\phi(\bar f_{i+1}) = c_{e_i}^{n_i}\bar f_{i+1} c_{\bar e_{i+1}}^{n_{i+1}}$ for some $n_i,n_{i+1} \in \bbZ$,
    provided we set $e_0 = \bar e$ and $e_n = e$. Applying the identities \[
    c_{\bar e_{i+1}}^{n_{i+1}} e_{i+1}e_{i+2}=e_{i+1}e_{i+2} c_{e_{i+2}}^{n_{i+1}}
    \] finally gives\[
    \phi\cdot_e g = c_{\bar e}^{d_0+n_0}\bar f_1 e_1e_2 c_{e_2}^{d_2+n_1+m_1+n_2}\bar f_3 \cdots e_nc_{e_n}^{d_n+n_{n-1}+n_n+m_{n-1}} \bar f_{n+1} c_{\bar e}^{d_{n+1}+n_{n+1}} t_b^k.
    \] Thus, if $g$ is not $ee_j$-prefix clean we have \[\mathrm{pow}(\phi\cdot_e g,e) = \mathrm{pow}(g,e)+\underbrace{(n_0,m_1+n_1+n_2,\ldots,m_{n-1}+n_{n-1}+n_n,n_{n+1})}_{\epsilon_{D,\rho,e}(\phi)},\] since $\epsilon_{D,\rho,e}(\phi)$ doesn't depend on $\mathrm{pow}(g,e)$, but only on $D,\rho,e,\phi$.
    
    If $g$ is $ee_j$ prefix clean, then note that since $e_1=e$ we have
    \[
        \phi\cdot_e e_1e_2c_{e_2}^{d_2}\bar f_3 \cdots = \gamma_{\bar e} \gamma_{\bar e}^{-1}e_1\gamma_{e_1}\gamma_{e_2}^{-2}e_2 \cdots =e_1\gamma_{e_1}\gamma_{e_2}^{-2}e_2 \cdots = e_1e_2\cdots
    \]
    In particular if $g$ is $ee_j$ prefix clean then the $n_0$-entry in $\epsilon_{D,\rho,\epsilon}(\phi)$ is trivial.    
    \end{proof}

    Up to this point, we have been focusing mainly on fibre parts of words we will now consider conjugation by elements that are not in the fibre.
    
    \begin{prop}\label{prop:power-edge-moves}
      Let $D$ be a tuple of cosets, $e$ some polarizing edge and $\rho$ some $X$-loop based at $i(e)=v$ that is compatible with $D$. Then the edge group $\bbX_e$ acts on $\calG(D,\rho,e)$ via\[
        a \cdot g = i_e(a) g i_e(a)^{-1}, g \in \calG(D,\rho,e), a \in \bbX_e
      \] and this action induces a homomorphism $\alpha_{D,\rho,e}:\bbX_e \to \calE_{D,\rho,e}$ such that\[
      \mathrm{pow}(a\cdot g,g) = \mathrm{pow}(g,e)+\alpha{D,\rho,e}(a).  
      \]
    \end{prop}
    \begin{proof}
        By the definition of $e-$polarized normal forms, conjugation of $g$ by $i_e(a)=c_{\bar e} \in i_e(\bbX_e)$ doesn't change $DC(g,e)$. If $g$ is $ee_j$-prefix clean then we can migrate the initial $c_{\bar e}^{n_0}$ symbol through the $ee_j$ prefix and stay in prefix clean form. It's easy to verify that the new power vector is obtained from the old power vector by adding some $\epsilon_{D,\rho,e}(a)$. The other generator of $i_e(\bbX_e)$ is $t_v$, the generator of the centre of $\bbX_v$. The action by conjugation by $t_v$ on the fibre $F \leq \torus\alpha$ is precisely the multiple Dehn twist $\alpha$ which is given by the $\edges X$ substitution $e \mapsto e c_e^{n_{e,\alpha}}$. In particular conjugation by $t_v$ can be realized by an element of $\delta_1\autfo\bbX$ so Proposition \ref{prop:power-moves} implies the result for conjugation by powers of $t_v$, and the result follows.
    \end{proof}
    
    \begin{cor}[Action on embedding space]\label{cor:action-on-power-vectors}
        Let $A_D \leq \delta_1\autfo\bbX$ be a subgroup of automorphisms that preserve the cosets in $D$. Then $\bbX_e\times A_D$ acts on $\calG(D,\rho,e)$ via\[
           (a,\phi)\cdot g = i_e(a)(\phi\cdot_e g)i_e(a)^{-1} 
        \] and this action gives rise to a homomorphism\[
            \chi_{D,\rho,e}:\bbX_e\times A_D \to \calE_{D,\rho,e}
        \] such that
        \begin{equation}\label{eqn:power-action}
            \mathrm{pow}((a,\phi)\cdot g,e)=\mathrm{pow}(g,e)+\chi_{D,\rho,e}(a,\phi)
        \end{equation}
    \end{cor}

    \subsection{Solving the fiber and orientation preserving Mixed Whitehead problem}\label{sec:solving-MWHP}
    
        In this section, we describe and prove the correctness of our algorithm. Suppose we are given two tuples of tuples $\vec S=(S_1,\ldots,S_k)$ and $\vec T=(T_1,\ldots, T_k)$ of elements from $\torus\alpha$. The fibre and orientation preserving Mixed Whitehead problem is the decision problem of whether $\vec S$ and $\vec T$ in the same orbit under the action of $\delta\autfo\bbX\times(\torus\alpha^k)$ given by
        \begin{equation}\label{eqn:OG-action}
                (\phi,g_1,\ldots,g_k)\cdot (S_1,\ldots,S_k) = (g_1^{-1}\phi(T_1)g_1,\ldots,g_k^{-1}\phi(T_k)g_k).
        \end{equation}
        We first note that $\vec S$ and $\vec T$ are in the same orbit if and only if their augmentations $\vec S^+=(S_1^+,\ldots,S_k^+)$ and $\vec T^+=(T_1^+,\ldots, T_k^+)$ are in the same orbit. So, without loss of generality, we will assume that $\vec S$ and $\vec T$ are tuples of augmented tuples, if they aren't then we replace their entries by augmentations (which can be done algorithmically).
        
        By Lemma \ref{lem;reduc_delta_1} this problem could be reduced to a finite computable set of instances of similar orbit problems under $\delta_1\autfo\bbX\times(\torus\alpha^k)$. Specifically, $\vec S$ and $\vec T$ are in the same $\delta_1\autfo\bbX\times(\torus\alpha^k)$-orbit if and only if for one of the representatives $\nu_i$ of a coset of $\delta_1\autfo\bbX\backslash\delta\autfo\bbX$ we have that $\vec{\nu_i(S)}$ is in the same $\delta_1\autfo\bbX\times(\torus\alpha^k)$-orbit as $\vec T$.
        
        Now given some tuple $T$ of elements, we denote by $\calT$ the set of all the $\Bass(\bbX)$-conjugates of $T$ that are in centred short position (recall Section \ref{sec:centered-hyp}), if $T$ is hyperbolic or lineal, otherwise $\calT$ is the set containing the centring of $T$ and all its variations, should any exist (recall Section \ref{sec:centered-elliptic}.) With this notation, we define two sets of tuples of tuples\[
            \vec{\calS} = \bigcup_{i=1}^v\calS_{i1}\times\cdots\times\calS_{ik}
        \] where $\{\nu_1,\ldots,\nu_v\}$ is the computable set of representatives given by Lemma \ref{lem;reduc_delta_1} and $\vec{\nu_i(S)} = (S_{i1},\ldots,S_{ik}); i=1,\ldots,v$. We also define\[
        \vec\calT = \calT_1\times\cdots\times\calT_k
        \] where $\vec T = (T_1,\ldots,T_k)$.
        
        On the one hand, $\delta_1\autfo\bbX$ actually extends to a group of automorphisms $\Bass(\bbX)$. On the other hand, Lemma \ref{lem:conj-in-Bass} implies that we can extend our consideration to conjugation in the Bass group. Therefore our original $\vec S$ and $\vec T$ are in the same $\delta\autfo\bbX\times(\torus\alpha^k)$-orbit if and only if there exists at least one pair of tuples of tuples $\vec S' \in \vec\calS$ and  $\vec T' \in \vec\calT$ that are in the same $\delta_1\autfo\bbX\times(\Bass(\bbX)^k)$-orbit.
        
        In Section \ref{sec;action_delta_1}, we considered polarized actions which are not exactly the same as the action given in \eqref{eqn:OG-action}. However, since the Mixed Whitehead problem involves post-conjugation by group elements, this does not  
        change the problem. To be clear: our original $\vec S$ and $\vec T$ are in the same $\delta\autfo\bbX\times(\torus\alpha^k)$-orbit if and only if there exists at least one pair of tuples of tuples $\vec S' \in \vec\calS$ and  $\vec T' \in \vec\calT$ that are in the same $\delta_1\autfo\bbX\times(\Bass(\bbX)^k)$-orbit where the new action is given by:
        \begin{equation}\label{eqn:newschool-action}
            (\phi,g_1,\ldots,g_k)\cdot (S'_1,\ldots,S'_k) = \left(g_1^{-1}(\phi \cdot S'_1) g_1,\ldots,g_k^{-1} (\phi\cdot S'_k)g_k\right).
        \end{equation}
        Where the action is given by \eqref{eq:tupletuple-action}. This is because if $S'_j$ is hyperbolic or lineal, then $\phi\cdot S'_j = \gamma_{\bar e} \phi(S'_j)\gamma_{\bar e}^{-1}$.

        By Lemma \ref{lem:bonus}, prefix short forms are preserved by polarized actions. So suppose that $\vec S'$ and $\vec T'$ are in the same $\delta_1\autfo\bbX\times(\Bass(\bbX)^k)$ orbit, under the new action. Then there is some $\phi \in \delta_1\autfo\bbX$ such $\phi\cdot S'_i$ is conjugate in $\Bass(\bbX)$ to $T'_i$ for all $i=1,\ldots,k$. Now non-elliptic tuples in $\phi \cdot \vec S'$ may no longer be centred short position, but their leading entries are in prefix clean position. It follows that there are elements in $a_i\in i_{e_i}(\bbX_{i_(e_i)})$ such that the leading entry of\[
            a_i (\phi\cdot S'_i)a_i^{-1}
        \] is in short position for all appropriate indices. We'll now show why the entire tuple  $a_i (\phi\cdot S'_i)a_i^{-1}$ is in centred short position (and not only its leading entry). If that isn't the case, then we would need to further conjugate $a_i (\phi\cdot S'_i)a_i^{-1}$ by some element in the centralizer of the first entry to get to short position of the tuple. If $S'_i$ is lineal, doing so doesn't change any of the entries in $a_i (\phi\cdot S'_i)a_i^{-1}$, so $a_i (\phi\cdot S'_i)a_i^{-1}$ is already in short position. Otherwise, $S'_i$ was already chosen so that its second entry had shortest syllable length among all conjugates by the centralizer of the first entry since applying $\phi$ and conjugation by $a_i$ did not change syllable lengths, further conjugation by an element of the centralizer of the first entry of $a_i (\phi\cdot S'_i)a_i^{-1}$ will increase the syllable length of the second entry, so in this case as well $a_i (\phi\cdot S'_i)a_i^{-1}$ is already in short position.
        
        Now all the tuples in $\vec T'$ are also in centred short position. Thus either $a_i (\phi\cdot S'_i)a_i^{-1} = T'_i$, otherwise $a_i (\phi\cdot S'_i)a_i^{-1}$ is equal to one of the other finitely many short conjugates of $T'_i$. In particular there must be some other $\vec T'' \in \vec\calT$ such that $a_i (\phi\cdot S'_i)a_i^{-1} = T''_i$.
        
        If $S'_i$ is elliptic and $\phi\cdot S'_i$ is conjugate to $T'_i$ then either $\phi\cdot S'_i$ is conjugate to $T'_i$ by an element of $\bbX_{v_{S'_i}}$ or by some element of the form $\bbX_{v_{S'_i}}\ell$ where $\ell$ is a loop in $X$ based at $v_{S'_i}$. By our definition of variations and by Lemma \ref{lem:conj-elliptic}, there is some $\vec S''$ such that $\phi\cdot S''_i$ is conjugate to $T'_i$ by an element of $\bbX_{v_{S''_i}}$.
        
        Therefore our original instance of the fibre and orientation preserving Mixed Whitehead problem has a positive solution if and only if there exist $\vec S' \in \vec\calS$ and $\vec T' \in \vec\calT$, some $\phi \in \delta_1\autfo\bbX$ and elements $a_i$ where $a_i \in F_{v_i}$ if $S'_i$ is elliptic and $a_i \in i_{e_i}(\bbX_{e_i})$ if $S'_i$ is lineal or hyperbolic such that
        \begin{equation}\label{eqn:final}
            a_i (\phi\cdot S'_i) a_i^{-1}=T'_i.
        \end{equation}
        
        Note that by Proposition \ref{prop:power-edge-moves} conjugation by such elements $a_i$ for non-elliptic tuples leave invariant $DC(\phi\cdot S_i)$. We therefore finally apply Corollary \ref{cor;DC-tuples-of-tuples} to $\vec S'$ and $\vec T'$. This gives us an automorphism $\phi_1 \in \delta_1\autfo\bbX$ such that for elliptic tuples we already have $[\phi_1 \cdot S'_i]_{v_i} = [T'_i]_{v_i}$, so \eqref{eqn:final} has a solution for elliptic tuples.
        
        Without loss of generality the first $h$ entries $S'_1,\ldots,S'_h$ of $\vec S'$ are precisely the non-elliptic tuples, the same must be true for $\vec T'$. For a tuple of elements $T_{i}=(T_{i1},\ldots,T_{il_i})$ we define the embedding space\[
            \calE_{T_i} = \bigoplus_{j=1}^{l_i} \calE_{DC(T_{ij},e_i),\rho_{ij},e_i}
        \] where $\rho_{ij}$ is the $X$-path underlying $T_{ij}$. For the tuple of tuples $\vec T'$ we define the embedding space\[
            \calE(\vec T') = \bigoplus_{i=1}^h \calE_{T_i}.
        \]
        
        Now by Corollary \ref{cor;DC-tuples-of-tuples} we can decide if there is some  $\phi_1 \in \delta_1\autfo\bbX$ such that in addition $\vec{DC}(S'_i,e_i)=\vec{DC}(T'_i,e_i)$ in which case we would also have $\calE(\phi_1\cdot \vec S') = \calE(\vec T')$. 
        
        Consider now analogous mappings \[\mathrm{pow}(T_i) = \bigodot_{j=1}^{l_i} \mathrm{pow}(T_{ij},e_i) \textrm{~and~} \mathrm{pow}(\vec T') = \bigodot_{i=1}^h \mathrm{pow}(T_i),\] where $\odot$ denotes concatenation of tuples. By Corollary \ref{cor:action-on-power-vectors} there is an action of $A_{\vec T'}\times \bbX_{e_1} \times \cdots\times \bbX_{e_h}$,  where $A_{\vec T'}$ is as given in Corollary \ref{cor;DC-tuples-of-tuples}, obtained by assembling the actions on embedding spaces and a corresponding homomorphism\[
            \chi: A_{\vec T'}\times \bbX_{e_1} \times \cdots\times \bbX_{e_h} \to \calE(\vec T'),
        \]
        that satisfies the analogue of \eqref{eqn:power-action}.
        
        It finally follows that our instance of the Mixed Whitehead problem has a solution if and only if for the fixed $\phi_1 \in \delta_1\autfo\bbX$ that we computed above, there is some $(\phi_2,a_1,\ldots,a_n) \in A_{\vec T'}\times \bbX_{e_1} \times \cdots\times \bbX_{e_h}$ such that\[
            \mathrm{pow}(\vec T') = \mathrm{pow}(\phi_1\cdot \vec S')+\chi(\phi_2,a_1,\ldots,a_n).
        \] Since we have an explicit generating set of 
        $A_{\vec T'}\times \bbX_{e_1} \times \cdots\times \bbX_{e_h}$ and can compute $\chi$ the problem reduces to deciding whether given a generating set for a subgroup $\bk{h_1,\ldots,h_n}=H \leq \bbZ^n$ and two elements $\vec s$ and $\vec t$ whether there is $\vec h \in H$ such that $\vec t = \vec s + \vec h$. This amounts to elementary linear algebra over $\bbZ$ and is readily decidable.
        
        It therefore follows that the fibre and orientation preserving Mixed Whitehead problem is decidable in $\torus\alpha.$

\subsection{The proof of Theorem \ref{thm;gersten_hack}}
\label{sec;hack}

    We finally come back to  Theorem \ref{thm;gersten_hack}, for which we promised a proof.
    
    Let fix a finite basis $X$ of $F=F(X)$, then there is a one-to-one correspondence between subgroups of $F$ and \emph{core graphs with a base vertex $v$} (see \cite[Theorem 5.1]{kapovich_stallings_2002} or \cite{stallings_topology_1983}.) If $H \leq \freegp n$ is finitely generated, then the corresponding based core graph is also finite. \emph{An unbased core graph} is a directed graph with labels in $X$ which is folded (in the sense of \cite[Definition 2.3]{kapovich_stallings_2002}) without vertices of valence one. By \cite[Proposition 7.7]{kapovich_stallings_2002} there is a bijective correspondence between conjugacy classes of subgroups of $F$ and unbased core graphs. $\aut{F}$ acts on this set of graphs as follows: given an automorphism $\alpha$ and an unbased graph $\StalGraph H$, the image $\alpha(\StalGraph H)$ is computed using the following steps.
    \begin{enumerate}[(I)]
        \item \label{it:subdivide} For each edge labelled $y$    subdivide it into a path graph along which we read $\alpha(y)$.
         \item \label{it:fold} Perform a sequence of Stallings folds, until the graph is folded.
        \item \label{it:core} Perform a \emph{coring} operation, i.e. repeatedly contract edges adjacent to vertices of valence 1 to points until none such remain.
    \end{enumerate}
     Details are given in \cite[\S 9.2]{diao_grushko_2005}. Note that each of these operations can be realized as a sequence of continuous quotient maps between graphs sending vertices to vertices. In particular each vertex of $\StalGraph H)$ has a well-defined image in $\alpha'(\StalGraph H)$. A crucial fact is that the individual folds or edge contractions in steps \ref{it:fold} and \ref{it:core} can be done in any order and the end result will be the same. We will use Gersten's Theorem \ref{thm:gersten} presented earlier.

    As mentioned earlier, an actual subgroup $H \leq F$ is given by a \emph{core graph with a base vertex $\BasedStalGraph H v$} which is a directed folded graph with labels in $X$ all of whose vertices, except possibly the base vertex $v$, have valence one. The elements of $H$ are precisely the reduced words one reads of reduced loops starting and ending at the base vertex $v$.
    
    \subsubsection{Construction of an auxiliary $F(X,K)$ with its tuples of conjugacy classes of subgroups}
    
    For the remainder of this section we fix tuples $A=(A_1,\ldots,A_r)$ and $B=(B_1,\ldots,B_r)$ 
    of tuples of subgroups of $F=F(X)$ (i.e. each $A_i$ is a 
    tuple of subgroups).  We assume that $|A_i|=|B_i|$, and by convention, we assume that $|A_i|\geq 2$ for all $i$. If it happens that one is interested in $A_i$ being a singleton, one may take a pair of identical copies of the same subgroup so that this convention is satisfied. This duplication does not add any further constraint, hence is done without loss of generality.
    
    For each $i$, let $K_i$ be a  set of symbols  so that $|K_i|=|A_i|=|B_i|$.  Let $K=K_1 \sqcup \cdots \sqcup K_r$ and consider $F=F(X)$ as a free factor of $F(X,K)$.

    We list $K_i = \{\kappa_{i1},\ldots,\kappa_{in_i}\}$ as a set of symbols double indexed by positive integers such that $K_i \cap X = \emptyset$ and such that $i\neq j \Rightarrow K_i\cap K_j = \emptyset$. A \emph{$K_i$-star} is a directed labelled tree with a root vertex $v_i$ and for each symbol in $\kappa_{ij} \in K_i$ there is an edge outgoing from $v_i$ with label $\kappa_{ij}$. 
    
    We encode the tuples $A_i=(A_{i1},\ldots,A_{in_i})$ and $B_i=(B_{i1},\ldots,B_{in_i})$ of subgroups of $F$ using a \emph{tuple-graph}, denoted $\StalGraph{A_i}$ and $\StalGraph{B_i}$, respectively. They are constructed by taking a $K_i$-star, the disjoint union of the core graphs with a base vertex $\BasedStalGraph{A_{ij}}{u_{ij}}$ (respectively $\BasedStalGraph{B_{ij}}{u_{ij}}$), and identifying the vertex of the $K_i$-star with incoming edge labelled $\kappa_{ij}$ to the base vertex $u_{ij}$ (respectively $v_{ij}$). The key fact is that although $\StalGraph{A_i}$ is not a based graph, due to its specific structure, and edges with special labels $\kappa_{ij}$ it is possible to recover the core graphs with base vertex $\BasedStalGraph{A_{ij}}{u_{ij}}$ or $\BasedStalGraph{B_{ij}}{v_{ij}}$ and thus the ordered tuple of subgroups $A_i$ or $B_i$ respectively.
    
    A tuple-graph always has a distinguished vertex $v_i$, namely the only vertex with outgoing edges labelled by the elements $K_i$. Denote by $\dot A_i=\pi_1(\StalGraph{A_i},v_i) \leq F(X,K)$ the subgroup consisting of the elements that are represented as the words read along closed loops in $\StalGraph{A_i}$ based at $v_i$ and define $\dot{B_i}$ analogously. Note that the conjugacy classes $[\dot A_i]$ and $[\dot B_i]$ correspond exactly to the tuple graphs $\StalGraph{A_i}$ and $\StalGraph{B_i}$ respectively.
    
    Fix the following tuples of subgroups:
    \begin{eqnarray}
        \hat A &=& (\dot A_1, \ldots, \dot A_r, F(X), F(K_1),\ldots,F(K_r), R_1,\ldots R_r)\\
        \hat B &=& (\dot B_1, \ldots, \dot B_r, F(X), F(K_1),\ldots,F(K_r), R_1,\ldots R_r),\label{def;hatB}
    \end{eqnarray}
    where $R_i \leq F(K_i)$ is a rigid subgroup, in the sense that ${\rm Stab}_{\aut{F(K_i)}}(R_i)\leq \Inn{F(K_i)}$, for example $R_i = \bk{c_i}$ where $c_i$ is a C-test element as given in \cite{ivanov_Ctest_1998}. These are precisely the objects that Theorem \ref{thm:gersten} is about.
    
    \begin{lem}\label{lem;gersten-hack}
      Let $\beta \in \aut{F(X,K)}$ send $\hat A$ to $\hat B$ or fix $\hat B$. Then we have the following restrictions
      \begin{enumerate}
          \item $\beta|_{F(X)} = \ad{w}\circ \alpha$, where $\alpha \in \aut{F(X)}$ and $w \in F(X,K)$, and
          \item $\beta|_{F(K_i)} = \ad{w_i}$ for some $w_i \in F(X,K)$.
      \end{enumerate}
    \end{lem}
    \begin{proof}
      By hypothesis $\beta$ maps $F(X)$ to a conjugate $w^{-1}F(X)w$ and similarly $\beta(F(K_i))={w_i}^{-1}F(K_i)w_i$. This implies that the restriction $\ad{w^{-1}}\circ\beta|_{F(X)}$ is given by an automorphism $\alpha: F(X) \to F(X) \leq F(X,K).$ This proves the first point.
      
      For the second point, note that we also require that $\beta$ maps $R_i$ to a conjugate of $R_i$ since the restriction $\ad{w_i}\circ\beta|_{F(K_i)}$ is an automorphism of $\alpha_i$ of $F(K_i)$ and therefore $\ad{w_i}\circ\beta(R_i)=g_i^{-1}R_ig_i$ for some $g_i \in F(K_i)$. The hypothesis that $R_i$ is rigid implies that $\ad{w_i}\circ\beta|_{F(K_i)}$ actually restricts to an inner automorphism on $F(K_i)$. This proves the second point.
    \end{proof}

    \subsubsection{From a solution in $F(X)$ to a solution in $F(X,K)$}
    
    \begin{prop}\label{prop;gersten-hack-easy}
        If there is an $\alpha \in \aut{F(X)}$ such that $[\alpha(A_i)] =   [B_i]$
        then there is an automorphism $\beta \in \aut{F(X,K)}$ such that \[
        \beta(\hat A)=\hat B.
        \]
    \end{prop}
    \begin{proof}
      The $\alpha$ given in the hypothesis extends to an automorphism $\hat\alpha$ of $F(X,K)$ that fixes the elements of the free factor $F(K)$ elementwise. Let $w_i \in F(x)$ be the elements such that $w_i\alpha(A_i)w_i^{-1} = B_i$.
      
      The conjugacy class $[\dot A_i]$ corresponds to the (unbased) tuple graph $\StalGraph{A_i}=\StalGraph{\dot A_i}$. Denote by $v_{i1},\ldots,v_{in_i}$ the vertices that are the basepoints of the based subgraphs $\BasedStalGraph{A_{ij}}{v_{ij}} \subset \StalGraph{A_i}$ coming from the entries $A_{ij}$ of the tuple $A_i$. Every edge of $\StalGraph{A_i}$ with a label in $X$ lies in one of these subgraphs. Applying $\hat\alpha$ to $\StalGraph{A_i}$ is achieved by subdividing every edge with label $x \in X$ to an appropriately oriented path with label $\alpha(x)$, leaving the edges  with label $K_i$ unchanged, and finally folding and coring. The same result can be achieved by a tuple graph for $(\alpha(A_{i1}),\ldots,\alpha(A_{in_i}))$. This is not quite the target tuple we want, but we can remedy this by conjugating the special symbols the $K_i$ by the appropriate elements of $F(X)$. We achieve this with the automorphism $\beta' \in \aut{F(X,K)}$ that is defined on the basis $X\cup K$ as follows:\[
      \beta'(z) = \begin{cases}
        z & \textrm{if~} z=x\in X\\
        w_i z w_i^{-1} &\textrm{if~} z=\kappa_{ij} \in K_i
      \end{cases}.
      \] The effect of first applying $\beta'$ to each $\hat\alpha(\StalGraph{A_i})$ and then folding and coring is depicted in Figure \ref{fig;gersten-hack-easy}. Thus, taking $\beta = \beta'\circ\hat\alpha$, we have $[\beta(\dot A_i)]= [\dot B_i]$. Furthermore, it is immediate from the construction that $\beta$ maps $F(X)$ to itself and the restrictions $\beta|_{F(K_i)}(g) = \ad{w_i^{-1}}(g)$. So, in particular, $\beta$ maps the $F(K_i)$ to themselves and the subgroups $R_i$ to conjugates in $F(X,K)$.
    \end{proof}
        \begin{figure}[htb]
            \centering
            \includegraphics[width=\textwidth]{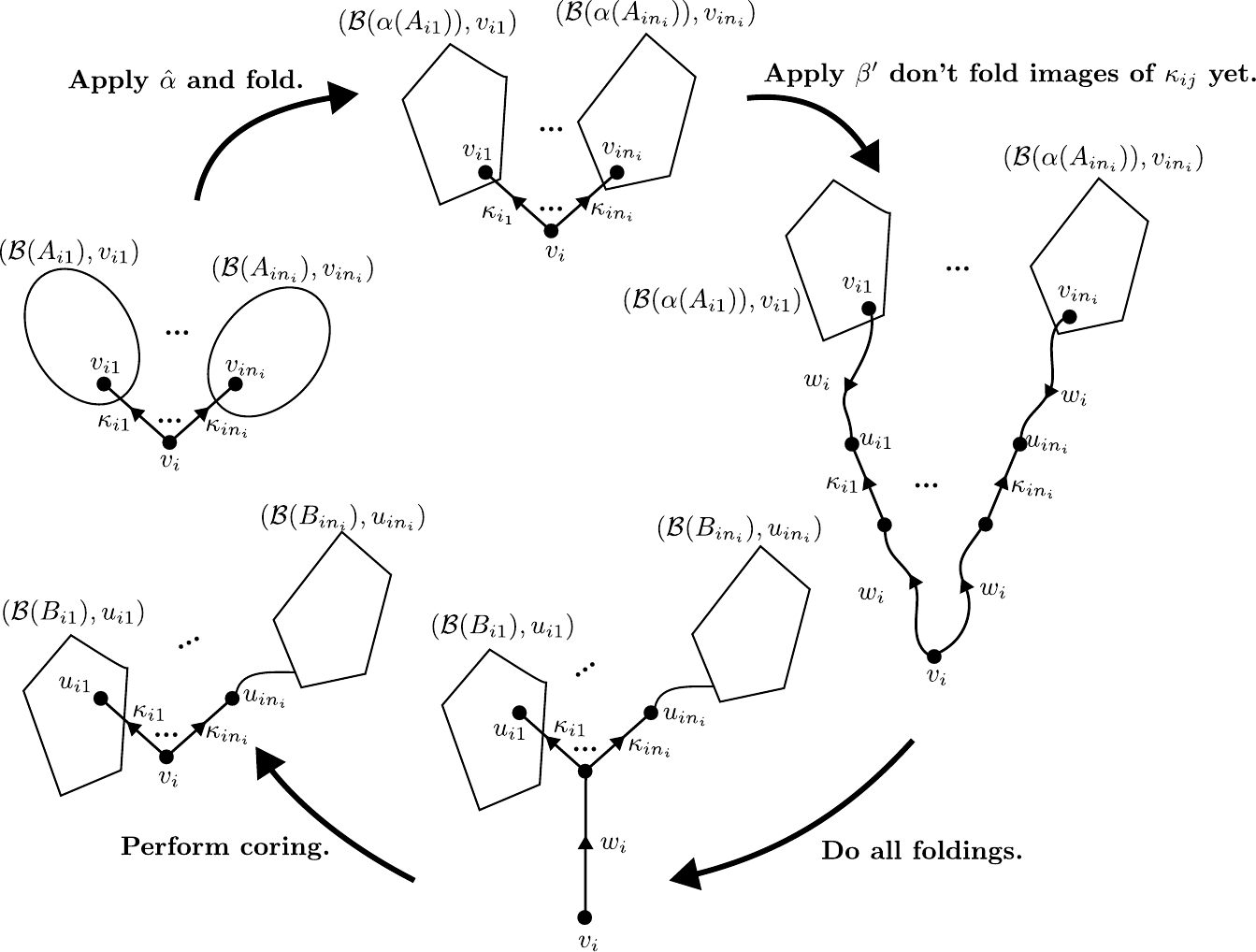}
            \caption{The effect of applying $\beta'$ to $\hat\alpha(\StalGraph{A_i})$, is shown in four steps. Each arrow is a continuous map, we use the same notation to designate a vertex and its image.}
            \label{fig;gersten-hack-easy}
        \end{figure}

    \subsubsection{From a solution in $F(X,K)$ to a solution in $F(X)$}
    
    Before stating and proving Proposition \ref{prop;rabbattre}, let us prove a lemma about foldings.

    \begin{lem}\label{lem;length-folding}
        Let $\calG$ be a directed graph with labels in $X\cup K$ that is constructed as follows. Take some $K_i$-star, whose vertices of valence 1 are denoted $u_1,\ldots,u_{n_i}$. Then take a disjoint union of directed graphs with basepoints $(\calG_1,v_1),\ldots,(\calG_{n_i},v_{n_i})$ with labels in $X$, all of which have non-trivial fundamental groups. Also, take $n_i$ copies of some path (a directed labelled graph homeomorphic to a line segment) $\gamma$ with labels in $X\cup K$ with initial vertex $u_\gamma$ and terminal vertex $v_\gamma$. Suppose furthermore that the word read along $\gamma$ is reduced. Finally from $\calG$ by joining each $\calG_i$ to the $K_i$-star by a copy of $\gamma$ by identifying $u_i$ with the copy of $u_\gamma$ and the copy of $v_\gamma$ with $v_i$. See Figure \ref{fig;length-folding}
        
        If $\gamma$ contains any edges with a label that is not in $X$, then $\calG$ cannot be brought to an unbased core graph that is the tuple graph of some tuple $(H_1,\ldots,H_{n_i})$ of subgroups of $F(X)$ by a sequence of folds or coring operations.
    \end{lem}
    \begin{proof}
      Let the word read along $\gamma$ as we read from $u_\gamma$ to $v_\gamma$ be written as \[z_1\cdots z_r,\] where $z_i \in (X\cup K)^{\pm 1}$. Suppose towards a contradiction that some $z_i \not\in X^{\pm 1}$ but that $\calG$ after folding and coring can be brought to the tuple graph of some tuple of subgroups of $F(X)$.
      
      Without loss of generality, we can assume that the last symbol $z_r$ read along $\gamma$ lies in $K$. Indeed we are supposing that there is some symbol labelling an edge of $\gamma$ that is in $K$. We can redefine the graphs $(\calG_1,v_1)$, by adjoining the maximal terminal segments of the copies of $\gamma$ only containing symbols from $X$ and redefining $\gamma$ so that it is the minimal initial segment containing all symbols in $K$. The resulting graph is exactly the same. One of the properties of a tuple graph is that the only edges with labels in $K$ are the ones contained in the $K_i$-star. By hypothesis, the paths that are copies of $\gamma$ contain edges with labels in $K$, there must therefore be some folding to identify all of these edges with the edges in the $K_i$-star.
      
      We start by performing folds within the graphs $\calG_i$. Since the terminal edge in $\gamma$ has a label in $K$, there is no possible folding between the edges in the $\calG_i$s and edges in terminal segments of copies of $\gamma$. We may then proceed to coring. At this point we can now assume that the only possible foldings occur between the initial edge of some copy of $\gamma$ and some edge of the $K_i$-star. 
      
      Without loss of generality, we may assume that $z_1 = \kappa_{in_i}^{-1}$ and this means that at this point there is at most one pair of edges that can be folded. We will now consider the folding process step-by-step. Denote by $\gamma_i$ the copy of $\gamma$ attached to $u_i$, then we first fold the first edge of $\gamma_{n_i}$ with the edge of the $K_i$-star labelled $\kappa_{in_i}$. If there are no other possible folds then the resulting graph is not a tuple graph as required, so we must continue folding. Without loss of generality $z_1=\kappa_{i1}$, and the terminal segment of $\gamma_{n_i}$ starts being partially identified with an initial segment of $\gamma_1$ this gives equalities 
      \begin{equation}\label{eq;gamma-fold}
        z_3=z_1, z_4=z_2,\ldots          
      \end{equation}
     for each edge of $\gamma_{n_i}$ that gets folded with an edge of $\gamma_1$. Note, however, that \eqref{eq;gamma-fold} implies that if $\gamma^+$ is the initial segment of $\gamma_{1}$ consisting of edges that get identified with edges of $\gamma_{n_i}$. If $\gamma^+$ has label\[
     z_1\cdots z_\ell
     \] then we have \[
     z_i = \begin{cases}
        {\kappa_{in_i}}^{-1} & \textrm{if $i$ is odd}\\
        \kappa_{i1} & \textrm{if $i$ is even.}
     \end{cases}\]
     for all $1\leq i\leq \ell+2$. Furthermore, by choice of $\ell$ we must have that $z_{\ell+3} \neq z_{\ell+1}$, which includes the possibility that $z_{\ell+3}$ does not exist (which would happen if the terminal segment of $\gamma_{n_i}$ completely folded into $\gamma_1$ and $|\gamma|=\ell+2$.)
     
     Suppose now that we have folded up to identifying the $\ell^{\mathrm{th}}$ edge in $\gamma_1$ (labelled by $z_\ell$) up with the $(\ell+2)^{\mathrm{th}}$ edge of $\gamma_{n_i}$. There are two possibilities:
     \begin{itemize}
         \item If $z_{\ell+3}$ exists, i.e. $|\gamma|\geq \ell+3$, then there are no more possible folding operations and the resulting graph cannot be a tuple graph (see Figure \ref{fig;length-folding}).
         \item If $z_{\ell+3}$ does not exist, i.e. the terminal segment of $\gamma_{n_i}$ has completely folded into $\gamma_1$, then we are in a situation where the graph obtained by folding and coring $\calG_{n_i}$ is attached to a vertex in $\gamma_1$ between edges with labels $z_\ell$ and $z_\ell+1$, both of which lie in $K^{\pm 1}$, and there is no more possible folding. The resulting graph cannot be a tuple graph either.
     \end{itemize}
     In all cases, we arrive at a contradiction so the result holds.
    \end{proof}
\begin{figure}[htb]
    \centering
    \includegraphics[width=0.66\textwidth]{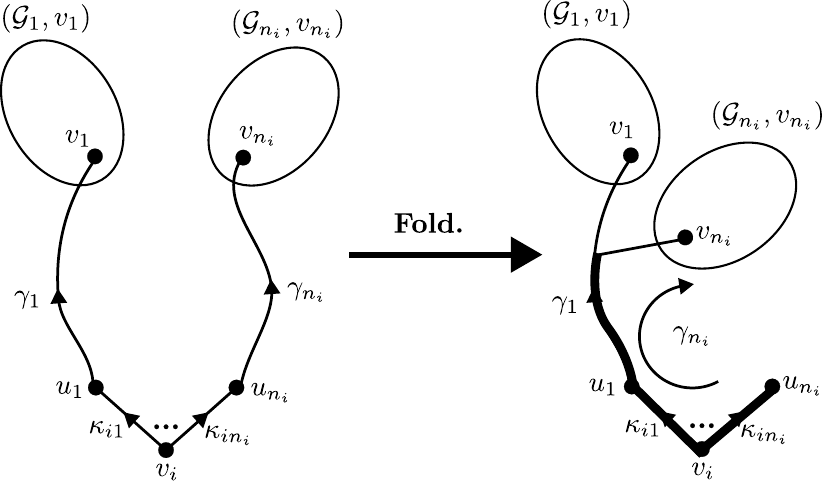}
    \caption{The possibly unfolded graph $\calG$ of Lemma \ref{lem;length-folding} and attempting to fold it down to a tuple graph. The thickened segment represents the edges that got folded with edges of $\gamma_{n_i}.$}
    \label{fig;length-folding}
\end{figure}
    
    We may now prove the main proposition of this subsection.
    
    \begin{prop}\label{prop;rabbattre}
     Let $A, B$ be two tuples of conjugacy classes of tuples of subgroups of the free group $F(X)$, and let $F(X,K)$ be the extension as above, and $\dot A_i, \dot B_i$ the corresponding groups of tuple graphs (they are subgroups of $F(X,K)$). 
     
    Assume that $\beta \in \aut{F(X,K)}$ is such that 
    \begin{itemize}
        \item $\beta(F(X)) = F(X)$,
        \item for all $i$, $[\beta(F(K_i))] = [F(K_i)]$ and $[\beta(R_i)]= [R_i]$,
        \item for all $i$, $[\beta(\dot A_i)] = [\dot B_i]$. 
        
    \end{itemize}
        Then $\beta|_{F(X)}: F(X) \to F(X)$ is an automorphism that sends the $F(X)$-conjugacy class (of tuples of subgroups) $[A_i]$ to $[B_i]$. 
    \end{prop}

    Observe that  $\beta|_{F(X)}$ is indeed an automorphism of $F(X)$ by assumption. 
    
    \begin{proof}
    Proving the proposition amounts to proving that the element that conjugates $\beta(\dot A_i)$ to 
    $\dot B_i$ can be taken in $F(X)$. 

    For each index $i$, there is a right coset $F(K_i)x_i$  such that $\beta(F(K_i))= F(K_i)^{x_i} $. 
    
    By the rigidity property of $R_i$, one has actually an element $ k_i \in F(K_i)$ such that, for all index $j$, $\beta(\kappa_{ij}) = \kappa_{ij}^{k_ix_i}$. Let us write $w_i=k_ix_i$.

    Applying Lemma \ref{lem;length-folding} to $\gamma = w_i$ and $\calG=\StalGraph{B_i}$, we obtain that    $w_i\in F(X)$.  
    
    It is therefore sufficient to prove that  $\beta(\dot A_i)= (\dot B_i)^{w_i}$. 
    This is the next lemma.

        \begin{lem}
             $\beta(\dot A_i)= (\dot B_i)^{w_i}$.
        \end{lem}
        \begin{proof}
        Knowing that $w_i\in F(X) $, we can consider the image of $\dot A_i$: it is generated by the groups $w_i^{-1}\kappa_{ij}w_i\beta(A_j) w_i^{-1}\kappa_{ij}^{-1} w_i$. We also know that $\beta(A_j) <F(X)$ since $\beta$ preserves $F(X)$.
        
        Let $g_{i,j,k}$ ($k$ ranging in some set $S_{i,j}$) be generators of $w_i^{-1}\kappa_{ij}w_i\beta(A_j) w_i^{-1}\kappa_{ij}^{-1} w_i$.
        
        All the considered generators $g_{i,j,k}$  are in $F(X)^{\kappa_{ij}^{-1} w_i}$ (for various $j$).  Foldings in the word $g_{i,j,k} (\neq 1)$ will never cancel the two letters 
        $\kappa_{ij}$ and $\kappa_{ij}^{-1}$ since all other letters are in $X$, and thus reveal a common prefix $w_i^{-1}$ and suffix $w_i$.
        
        Conjugating  $\beta(\dot A_i)$ by $w_i^{-1}$ produces a group generated by the $\kappa_{ij} C_{ij} \kappa_{ij}^{-1}$  (for $j$ ranging in the indices of the tuple $A_i$) with $C_{ij} <F(X)$.
        We must show that the generated group is $\dot B_i$.
        
         By assumption on $\beta$, each $\kappa_{i,j} C_{ij} \kappa_{i,j}^{-1}$  must be in a conjugate of $\dot B_i$ and their union generates this conjugate. Thus, let $y$ be such that for all $j$,  $ C_{ij}^{\kappa_{ij}^{-1}y^{-1}} <\dot B_i$, and the union (over $j$) generate $\dot B_i$. 
        
        Recall that by convention,  there are at least two indices $j$ in the tuple $A_i$. Write $y$ as a reduced word over $X\cup K$, and select $j$ such that $y\kappa_{ij}$ is reduced (there exists such $j$). Then considering reduced words $c\in C_{ij}$ (in $X$ only), the  words $y\kappa_{ij}c\kappa_{ij}^{-1}y^{-1}$ are reduced and are all in $\dot B_i$. Folding them on the graph $\StalGraph{B_i}$, it easily follows that $y\in \dot B_i$ (i.e. its folding finishes back to the vertex $v_i$, the only vertex on which the next edge labelled $\kappa_{ij}$ can fold). Hence  $ C_{ij}^{\kappa_{ij}^{-1}} <\dot B_i$, and the union over $j$ generate $\dot B_i$. In other words  $\beta(\dot A_i) =(\dot B_i)^{w_i}$, which we wanted to show.

        \end{proof}
        
    \end{proof}
   
   Propositions \ref{prop;gersten-hack-easy} and \ref{prop;rabbattre}  make the following an application of Gersten's theorem \ref{thm:gersten} (first point). 
   
   \begin{cor}\label{coro;new_ger_orb}
     Let $F$ be a free group, then given two tuples of conjugacy classes of tuples of finitely generated subgroups, one can decide whether an automorphism of $F$ sends the first to the second.
   \end{cor}
   
    \subsubsection{From generators of stabilizers in $\aut {F(X,K)}$ to generators of stabilizers in $\aut{F(X)}$}
    
    \newcommand{\SAFF}{{\rm Stab}_{\aut{F(X,K)}}(F(X))}

    \newcommand{\SAFXKB}{{\rm Stab}_{\aut{F(X,K)}}(\hat B)}
    

     \newcommand{\SAFB}{{\rm Stab}_{\aut{F(X)}}(B)}
     
    \newcommand{\SSABFX}{{\rm Stab}_{{\rm Stab}_{\aut{F(X,K)}}(\hat B)}(F(X))}
    
    
%
    
     Let $B$ be a tuple of conjugacy classes of tuples of subgroups of the free group $F(X)$, and let $F(X,K)$ and $\hat B$ be the extensions as above ($\hat B$ is a tuple of conjugacy classes of subgroups containing $F(X)$, whose list is in equation \ref{def;hatB}).
     
     We are interested in $\SAFXKB$ 
     (which we can compute by Theorem \ref{thm:gersten}), and its relation with  $\SAFB$. 
     
     As an auxiliary tool, we will consider \[\SSABFX
     <\SAFXKB.\]

    \begin{prop}\label{prop;rebalancer}
    
      Let $B$ be a tuple of conjugacy classes of tuples of subgroups of the free group $F(X)$, and let $F(X,K)$ and $\hat B$ be the extensions as above ($\hat B$ is a tuple of conjugacy classes of subgroups). 
      
      Let $\{\beta_j, \;  j\in J \}$ be a generating set of $\SAFXKB$, that contains the elements ${\rm ad}_x, x\in X$. 
      
      Then for all $j\in J$, there exists $\gamma_j \in F(X,K)$ such that 
      \[{\rm ad}_{\gamma_j} \circ \beta_j \in \SSABFX   < \SAFF, \]
      
      and $\{ {\rm ad}_{\gamma_j} \circ \beta_j,\; j\in J\}$ is a generating set of
       $\SSABFX$. 
       
    \end{prop}
    
    Observe that the condition that the given generating set contains the elements ${\rm ad}_x, x\in X$ is not restrictive since all these automorphisms preserve trivially $\hat B$, since the latter is a tuple of conjugacy classes.
    
    \begin{proof}
        By definition of $\hat B$, the automorphisms of $F(X,K)$ preserving $\hat B$ must preserve the conjugacy class of $F(X)$. 
        Therefore there exists $\gamma_j$ such that ${\rm ad}_{\gamma_j} \circ \beta_j \in \SAFF$. 
        We may take $\gamma_j = 1$ for all the $\beta_j$ that are of the form ${\rm ad}_x, x\in X$.

        Moreover, as we recalled, all inner automorphisms of $F(X,K)$ preserve trivially $\hat B$. It follows that 
        \[{\rm ad}_{\gamma_j} \circ \beta_j \in \SSABFX
        < \SAFF. 
        \]
        
        It remains to prove that the obtained automorphisms $\{ {\rm ad}_{\gamma_j} \circ \beta_j, \; j \in J \}$ generate the group 
        $\SSABFX$. 
        
        Let $\alpha_j =  {\rm ad}_{\gamma_j} \circ \beta_j$, to ease notations. 
        
        Consider an element  $w$ of $\SSABFX $.
        Since $\{\beta_j, \;  j\in J \}$ is a generating set of $\SAFXKB$,  
        we may write $w$ as \[w = \beta_{a_1}\beta_{a_2} \dots \beta_{a_s}\] with $a_i \in J$ for all $i$. Consider than $w' = \alpha_{a_1}\alpha_{a_2} \dots \alpha_{a_s} $. Substituting $\alpha_{a_i}$ by ${\rm ad}_{\gamma_{a_i}} \circ \beta_{a_i}$ and rearranging, it is elementary that there exists $\gamma$ such that $w' = {\rm ad}_\gamma \circ w$. 
        
        On one hand, $w'(F(S)) = F(S)$ because it is a composition of the  $\alpha_j$. On the other hand, $w(F(S)) = F(S)$ by assumption on $w$. It follows that $ {\rm ad}_\gamma$ preserves $F(S)$. However, since $F(S)$ is malnormal in $F(S,K)$, we can conclude that $\gamma \in F(S)$.  
        
        Since ${\rm ad}_x, x\in X$ are elements of the given generating set,  by our choice of $\gamma_j$, they are also elements of  $\{ \alpha_j, \; j\in J\}$. It is therefore possible to write 
         ${\rm ad}_\gamma$ as a word in the $\alpha_j$, and finally ${\rm ad}_{\gamma^{-1}}  \circ w' \in \langle \alpha_j, \; j\in J \rangle$.  Since ${\rm ad}_{\gamma^{-1}}  \circ w' = w$ we have proved that $w\in \langle \alpha_j \; j\in J \rangle$.
         
         We therefore have established the equality  \[\SSABFX 
         =  \langle \alpha_j \; j\in J \rangle.\] 
        
    \end{proof}
    
    Observe now the following.
    \begin{lem} \label{lem;back_in_FX}
    If  $w\in \SSABFX$, 
    then \[w|_{F(X)} \in \SAFB.\]

    If $w_0 \in \SAFB$, 
    then there exists 
     $w\in \SSABFX$,  
     for which $w|_{F(X)} = w_0$.
     
     \end{lem}        
     
     \begin{proof}
       The first part is a rewriting of the definitions, and for the second part, extend $w_0$ as the identity on $K$.
     \end{proof}
     
     \begin{cor}\label{coro;new_ger_stab}
       Given $B$ as in  Proposition \ref{prop;rebalancer}, a generating set of the group  $\SAFB$ 
       is computable.
     \end{cor}
     
     \begin{proof}
     Construct $F(X,K)$ and $\hat B$, compute, by Gersten's theorem \ref{thm:gersten} a generating set of  $ \SAFB$,   
     apply Proposition \ref{prop;rebalancer}, this provides, by Lemma \ref{lem;back_in_FX}  a generating set of  $\SAFB$. 
     \end{proof}

    By Corollary \ref{coro;new_ger_orb}, and Corollary \ref{coro;new_ger_stab}, we obtain Theorem \ref{thm;gersten_hack}, as promised.

    \section{Effective Minkowskian congruences}\label{sec;EMP}

            If $G$ is a group, $P$ is a characteristic subgroup of $G$ if it is invariant by automorphisms of $G$. We then write $P\charsgp G$, and it induces a homomorphism $\out{G}\to \out{G/P}$.  

          If $P_0$ is a finite index subgroup of a finitely generated group $G$,  by taking the intersection of all the (finitely many) subgroups of $G$ of the same index as $P_0$, one obtains  an associated characteristic finite index subgroup. 

          \begin{defn}\label{defn:speparating-congruences}
            Let $P \charsgp G$ be a finite index characteristic subgroup. We
            say \emph{$P$-congruences separate the torsion of $\out G$} 
            if every finite order outer automorphism
            $[\alpha] \in \out G$ has non-trivial image $[\bar\alpha]$ through the
            natural map\[ \out G \to \out{G/P}.
                  \]
                  
                  In that case, the finite quotient $G/P$ is a \define{Minkowskian congruence.}
                  
            A class of groups is effectively Minkowskian if each group admits a finite Minkowskian congruence 
            that is computable from a presentation.
          \end{defn}
        

        The proposed  terminology comes from a classical theorem of Minkowski: if $p\geq 3$, the kernel of the quotient $GL_n(\mathbb{Z}) \to GL_n(\mathbb{Z}/p\mathbb{Z})$ is torsion-free. In other words $\mathbb{Z}^n/(3\mathbb{Z})^n$ is  a Minkowskian congruence for  $\mathbb{Z}^n$, and free abelian groups form an effectively Minkowskian class.

            The following is a useful remark. 
                
            \begin{lem}
                If $P \charsgp G$    
                and $H\charsgp P$ is  a
                characteristic subgroup of $P$, then
                $H \charsgp G$ and  if   $P$-congruences separate the torsion in $\out G$, then $H$-congruences separate the torsion in $ \out G$. 
            \end{lem}
            \begin{proof}
                Any automorphism of $G$ induces an automorphism of $P$, hence preserves $H$.  
                There is a natural map $\out{ G/H} \to \out {G/P}$ which commutes with $\out G\to \out {G/P}$.  Any  finite order outer automorphism of $G$ survives in $\out {G/P}$, hence also in $\out {G/H}$. 
            \end{proof}

            The aim of this section is to prove the following.
        
            \begin{thm}\label{theo;SPTS_M}
                Finitely generated subgroups of piecewise trivial suspensions of free groups, are effectively Minkowskian.
            \end{thm}

            Recall that we have two levels of decomposition for such a group, from top to bottom: 
            \begin{itemize} 
                \item the Grushko decomposition in free product, 
                \item for each freely indecomposable free factor,  the canonical decomposition given by Corollary \ref{coro;hybrid_canonical_for_sbgp}: 
                an  acylindrical  bipartite  graph-of-groups with  edge groups either cyclic or isomorphic to $\bbZ^2$,  and white vertex groups that are either finitely generated free 
                or direct product of finitely generated free groups with an infinite cyclic group (a canonical trivially partially suspended decomposition). 
            \end{itemize}
        
            By Proposition \ref{prop;construct-canonical} such decompositions can be found effectively. In order to prove Theorem \ref{theo;SPTS_M} we will show the two following theorems.
            
            \begin{thm}\label{theo;EMP_free_product}
              A free product of  freely indecomposable, residually finite,  effectively Minkowskian, torsion-free groups  and free groups is effectively Minkowskian.
            \end{thm}

            \begin{thm}\label{theo;TPSC_Mink}
                Any freely indecomposable group that has a  canonical trivially partially suspended  decomposition, as defined above, is effectively Minkowskian. 
            \end{thm}

            In the free product case, we will use Hensel and Kielak's version of the Nielsen realization theorem for free products: a finite order automorphism of a free product is realized as a graph-of-group automorphism that takes advantage of a symmetry of a certain graph of groups and induces finite order outer automorphisms of the fixed vertex groups.
            
            In the one-ended case (Theorem \ref{theo;TPSC_Mink}), this is more involved.  We will first construct a clean covering of the given graph of groups, inspired by \cite{wise_subgroup_2000}. The cover will be taken to be characteristic so that automorphisms of the ambient group $\Gamma$ induce automorphisms of the finite-index subgroup $\Xi$  associated with the cover.
            
            Then we will take a quotient of $\Xi$, which we will call a vertex filling, to obtain a virtually free group $V$, which is a graph of finite groups over the same underlying graph, in which sufficiently many witnesses of the non-triviality of the finite order outer automorphisms will have survived. 
          
            Finally, we will take a quotient of $V$ to a finite group $Q_0$, in which all the conjugacy classes of finite order elements remain disjoint. By studying the images of the witnesses in the induced quotient $Q$ of $\Gamma$, we will finally establish that $Q$ is in fact a  Minkowskian congruence, thus giving us the desired result.
    
        \subsection{Effective Minkowskian property for free products}

            In this section, we prove Theorem \ref{theo;EMP_free_product}. Let $G_1, \dots, G_k$ be non-trivial, torsion-free  freely indecomposable groups that are effectively Minkowskian and residually finite. 
            The group $L_r$ is a free group of rank $r$. We consider the free product $\Gamma = G_1 *G_2 \dots G_k*L_r$. 
            
            \begin{prop}\label{prop;fm_good_quotients}
            There exists a finite list of finite quotients of $\Gamma$ such that,     
              for each automorphism $\alpha$ of $\Gamma$ that is non-trivial in $\out \Gamma$ and  has finite order in $\out \Gamma$, there is $\alpha'$ conjugate to $\alpha$ in $\out \Gamma$, a finite quotient $\Gamma \to \bar \Gamma$ in the list,  whose kernel is preserved by $\alpha'$, and on which $\alpha'$ descends as an automorphism $\bar {\alpha'} \in \aut{\bar \Gamma}$ whose outer class is non-trivial (in $\out{\bar \Gamma}$). 
            \end{prop}
            
            We will distinguish several cases according to the properties of $\alpha$. First, observe that $\alpha$ must permute (and can possibly preserve) the conjugacy classes of the groups $G_i$. For each $i$, let $G_i \to \overline{G_i}^{M}$ be a Minkowskian finite quotient. 
            
            \begin{lem}\label{lem;induce}
                If there is an index $i\leq k$, and $\gamma_i \in \Gamma$ such that $\alpha(G_i)^{\gamma_i} = G_i$, and if $ {\rm ad}_{\gamma_i} \circ \alpha|_{G_i}: G_i \to G_i$ defines a non-trivial outer automorphism of $G_i$, then the quotient \[\Gamma \to G_i \to \overline{G_i}^{M}\] satisfies the conclusion of Proposition \ref{prop;fm_good_quotients} for $\alpha$.
            \end{lem}
            \begin{proof}
                Observe that $ {\rm ad}_{\gamma_i} \circ \alpha$ is the same outer class as $\alpha$, hence has finite order in $\out \Gamma$. It follows that in its restriction to $G_i$ (that it preserves) it has finite order up to conjugation by an element normalizing $G_i$, thus in $\out{G_i}$. By assumption, it is non-trivial. Applying the Minkowski property for $G_i$, it survives in the quotient $\overline{G_i}^{M}$.
            \end{proof}
            
            In the next lemma, consider for each $i$ a non-trivial finite quotient by a characteristic subgroup $G_i \to \bar{G_i}^c$. This exists by residual finiteness, and one chooses them so that $ \bar{G_i}^c \simeq \bar{G_j}^c$ if and only if $G_i\simeq G_j$.  
            
            Construct the free product quotient of $\Gamma$: \[ \Gamma \to \overline{\Gamma}^{vf} \; = \;\bar{G_1}^c*\bar{G_2}^c*\dots \bar{G_k}^c*L_r.\]
            
            Observe that $\overline{\Gamma}^{vf}$ is virtually free. Make the further quotient to a finite group \[\overline{\Gamma}^{vf} \to \overline{\Gamma}^{dp}   \; = \; \bar{G_1}^c\times \bar{G_2}^c \times \dots \bar{G_k}^c.\]

            \begin{lem}\label{lem;permute} 
                Let $\alpha$ be an automorphism of $\Gamma$ with an index $i$  for which  for which $[\alpha(G_i)] \neq [G_i]$. Then $\alpha$ descends to the quotient  $\overline{\Gamma}^{dp}$ as a non-trivial outer automorphism.

            \end{lem}
            \begin{proof}
                First, it is clear that all automorphisms of $\Gamma$ descend to $\overline{\Gamma}^{dp}$.
                Let $g_i \in G_i$ that survives in $ \bar{G_i}^c$. Then $\bar \alpha$ sends $\bar g_i$ to a non-trivial element of another direct factor, hence a non-conjugate element; this proves the claim. 
            \end{proof}

            \begin{conv} \label{conv}
                From now on,  $\alpha$ is an automorphism of $\Gamma$ that is non-trivial in $\out \Gamma$, and of finite order in $\out \Gamma$, and such that, for all indices $i$, there exists $\gamma_i \in \Gamma$ for which ${\rm ad}_{\gamma_i} \circ \alpha|_{G_i}: G_i \to G_i$ is the identity of $G_i$. 
            \end{conv}
            
            To make sense of such $\alpha$ we need the following statement.
            
            \begin{lem}\label{lem;fm_conv}
              There are only finitely many conjugacy classes of elements as given in Convention \ref{conv} in $\out \Gamma$.
              
              Moreover, for each such $\alpha$, there exists $L'_r<\Gamma$ and $\gamma\in \Gamma$  such that\begin{itemize} \item $\Gamma$ splits as a free product $G_1*\dots G_k*L'_r$, \item   $\ad{\gamma} \circ \alpha (L'_r) = L'_r$
              \item $\ad{\gamma} \circ \alpha|_{L'_r}$ is non-trivial and of finite order in $\out{L'_r}$. 
              \end{itemize}
            \end{lem}
            
            \begin{proof}
               By \cite[Corollary 6.1]{hensel_kielak_nielsen_2018} all such outer automorphisms must fix a barycentre of a simplex in the relative outer-space of the free product. In other words,  the group   $\Gamma$ splits as a graph of groups with trivial edge groups and infinite vertex groups that are conjugate to $G_i$, in such a way that $\alpha$ induces an automorphism of graph of groups.
                
                This allows us to write an element  $\ad{\gamma} \circ \alpha$ of the outer class of  $\alpha$, as a graph-of-group automorphism of some fixed free-splitting graph of groups $\bbX$. Let it be  $\ad{\gamma} \circ \alpha = (\alpha_X, \alpha_v, \alpha_e, \gamma_e)$. By assumption on $\alpha$, one may take $\alpha_v = Id$ for all vertices and $\alpha_e = Id$ for all edges as well.  Also, the graph automorphism $\alpha_X$ must induce the trivial permutation on the non-free vertices of the graph (i.e. those that are attached to a non-trivial group).

                Write  $\ad{\gamma} \circ \alpha = \alpha'$ to ease notation.

                It follows that $(\alpha')^s = (\alpha_X^s, (Id_v)_v, (Id_e)_e, (\gamma_e^s)_e )$. Since $\gamma_e$ has infinite order,\footnote{The result would hold with little modification if each free factor of $\Gamma$ had only finitely many torsion elements  -- in particular for the free product of finite groups} 
                if $[\alpha^m] =[Id]$, then $\gamma_e = 1$. 
            
                Thus $\ad{\gamma} \circ \alpha = (\alpha_X, Id_v, Id_e, 1)$. Let $L'_r$ be the free factor of $\pi_1(\bbX, v_0)$ defined by the graph $X$. We see that $L'_r$ is preserved by $\alpha'$, and that $\alpha'$ induces a finite order automorphism of $L'_r$, which is non-trivial in $\out{L'_r}$ (otherwise $\alpha'$ would be  trivial in $\out \Gamma$.
            
                It follows that there cannot be more conjugacy classes of automorphisms $\alpha$ (in $\out \Gamma$) than the number of orbits of cells in the relative outer space multiplied by the maximal order of graph automorphism groups of thus appearing graphs.
            \end{proof}
            
            We may thus proceed to analyze the automorphisms $\alpha$ satisfying Convention \ref{conv} one by one.
            
            \begin{lem}\label{lem;one-by-one_conv}
              If $\alpha$ is an automorphism satisfying Convention \ref{conv}, then there exists a finite quotient of $\Gamma$ in which $\alpha$ descends as a non-trivial outer automorphism.  
            \end{lem} 
            
            \begin{proof}
              Taking the notations of the previous lemma, it suffices to notice that there is a finite quotient of $L'_r$ on which $\alpha'$ descends as a non-trivial outer automorphism. This is the Minkowskian property for free groups, proved in \cite[Corollary 3.2]{dahmani_touikan_reducing_2021}
            \end{proof}

             We thus showed Lemma \ref{lem;one-by-one_conv}  that all automorphisms satisfying the convention \ref{conv} are separated in a finite quotient, and Lemma \ref{lem;fm_conv}, that there are only finitely many conjugacy classes of them. We also showed (Lemma \ref{lem;induce},  \ref{lem;permute}) that there are finitely many finite quotients separating all finite order outer automorphisms that do not satisfy the convention.
        
                 Proposition \ref{prop;fm_good_quotients} is thus established.
        
            It is therefore sufficient to prove the following, in order to obtain Theorem \ref{theo;EMP_free_product}. Let $\alpha \in \out \Gamma$ and let $K\leq \Gamma$ be a finite index subgroup. If $\alpha([K])=[K]$ and $\alpha$ descends to a non-trivial outer automorphism of $\Gamma/K$ then we say \define{$K$ separates $\alpha$}.
            
            \begin{lem}
              If $K_1$ separates $\alpha_1$ and $K_2$ separates $\alpha_2$, then there is a subgroup $K$ that separates both $\alpha_1$ and $\alpha_2$.
            \end{lem}
            \begin{proof}
                 We will consider $K$ to be the intersection of all images of $K_1 \cap K_2$ by automorphisms (it is a finite intersection of subgroups sharing the same index in $\Gamma$). To show that $K$ is as required, we need examine to how $\alpha_1$ survives in $\Gamma/K$. First it induces an automorphism of $\Gamma/K$ since the quotient is characteristic. Second, $\Gamma/K_1$ is a quotient of $\Gamma/K$ on which $\alpha_1$ descends non-trivially non-trivially as an outer automorphism, thus it cannot be a trivial outer automorphism in $\Gamma/K$.
             \end{proof}
             
        \subsection{Effective Minkowskian property for one-ended groups with canonical splittings.}
     
            In this section, we consider the  Minkowskian property for groups for which a canonical decomposition as a graph-of-groups $\bbX$ is given. In that case, all outer automorphisms can be expressed as elements of $\delta \aut \bbX$, thus we clarify a possible congruence strategy to separate the torsion in  $\delta \aut \bbX$.
     
            \subsubsection{From graphs-of-groups to vertex groups}
     
     The following Proposition isolates two origins of torsion in  $\delta \aut \bbX$, either witnessed by permuting vertices or by restricting to a non-trivial finite order outer automorphism of a vertex group.
     
     \begin{prop}\label{prop;2cases}
       Let $\Gamma$ be the fundamental group of a   
       graph of group $\bbX$, and let $\phi = (\phi_X, (\phi_v)_v, (\phi_e)_e, (\gamma_e)_e) \in \delta \aut \bbX$ be non-trivial and of finite order in $\out \Gamma$. Suppose furthermore that for each non-abelian vertex group $\bbX_v$, the quotient by its centre is torsion-free.

       Then if $\phi_X = Id$, for at least  one vertex $v$, $\phi_v$ defines a non-trivial, finite order element in $\out{\bbX_v}$.
     \end{prop}

     \begin{proof}
        Assume the contrary: all $\phi_v$ must be of the form ${\rm ad}_{g_v}$. 
        We choose a preferred vertex $v_0$, with non-abelian vertex group,  and identify $\Gamma$ with $\pi_1(\bbX, v_0)$. After composition with an inner automorphism of   $\Gamma$, we may assume that  $g_{v_0} =1$, that is $\phi_{v_0} = Id$.

        Since $\phi_X $ is the identity, in the Bass group each edge $e$ is sent, by $\phi$,  on $\gamma_{\bar e}^{-1} e \gamma_e$. 
        
         If we choose $\tau $ a maximal subtree of $X$, one may change the expression of $\phi$ without changing its image in $\out \Gamma$, so that, for each edge $e$ in $\tau$ pointing outward from  $v_0$, the element $\gamma_e$ is trivial. To see this,  enumerate the edges in $\tau$ respectively to the distance to $v_0$ and modify iteratively as follows:  for a given $e$,  change   $\gamma_e$ by $1$, change    $g_{\tau(e)}$ by $ g_{\tau(e)}\gamma_e^{-1}$, and all  $\gamma_{\bar e'}$ for $e'$ exiting $\tau(e)$ by $ \gamma_{\bar e'} \gamma_e^{-1} $ (this was already done for $e' = \bar e$ !). It is an immediate computation to see that the Bass diagrams are satisfied, and the automorphism on $\Gamma= \pi_1(\bbX, v_0)$ is the same. 
         (For that, note that, in the Bass group, the element $ehe'$ is sent on $\gamma_{\bar e}^{-1} e \gamma_e g_{\tau(e)}^{-1}\,  h\,  g_{\tau(e)}   \, \gamma_e^{-1} \gamma_e \,  \gamma_{\bar e'}^{-1} e' \gamma_{e'} $.)

        If all $\phi_v$ are the identity, then, since all the $\gamma_e$ are trivial or have infinite order, it follows that $\phi$ is trivial or has infinite order.

        Consider $v_1$ closest to $v_0$ (in $\tau$) such that $\phi_{v_1}$ is not the identity. It is ${\rm ad}_{g_{v_1}}$ for some $g_{v_1}$ non-central in $\bbX_{v_1}$ (in particular $\bbX_{v_1}$ is non-abelian). By assumption, no power of $g_{v_1}$ is central. Let $p=e_1\dots e_k$ be a path from $v_0$ to $v_1$ in $\tau$, and consider \[\gamma(g) = g_0 e_1 \dots e_k g \bar e_k \dots \bar e_1 g'_0\] for $g\in \bbX_{v_1}$, and $g_0, g'_0$ not in adjacent edge groups and not inverse of each other (which is possible by choice of $v_0$).  
        
        Using that $\gamma_{e_i}=1$, one can compute that  the image of $\gamma(g)$ by $\phi^n$ is 
        
         \[\phi^n(\gamma(g)) \; = \;  g_0  \, x_{1,n} \, e_1 \, x_{2, n} \, e_2 \dots x_{k,n} e_k   
         \, g^{g_{v_1}^n}  
         \, \bar e_k \,  x_{k,n}^{-1} \, \bar e_{k-1} \dots   x_{1,n}^{-1} \,  g'_0  \]
        
        where the $x_{i, n} =( \gamma_{\bar{e_i}})^{-n}$ are in the vertex group $\bbX_{\tau(e_{i-1})}$.

        For all $n$, $ g_{v_1}^{n}$ is, by assumption, non-central in its vertex group, therefore  we may find  
        $g$ outside the centralizer of $g_{v_1}^{n}$,  
        so that 
        $  
        g^{g_{v_1}^n} 
        $ is different from $g$.  It follows that, since the expression of $\phi^n(\gamma(g))$ is in normal form,
        that $\phi^n$ is non-trivial. Since it is trivial on $\bbX_{v_0}$, but non-trivial on $\bbX_{v_1}$ it is non-trivial in $\out \Gamma$. This is in contradiction with the assumption that $\phi$ has finite order.
        \end{proof}
        
        In order to separate the torsion in the outer automorphism group, we will need particular quotients of our graphs of groups. The first step is to quotient the vertex groups to reach a virtually free group, which we achieve in Proposition \ref{prop:vertex-fill}. This will be achieved in two steps, first producing a clean cover of the graph of groups (Proposition \ref{prop;clean}), and second, filling appropriately the vertex groups. We will then proceed to use this quotient to separate conjugacy classes (Proposition \ref{prop;clean-filled}).

            \subsubsection{Clean covers and vertex fillings} 

    Recall that an element of a free group
    $\bbF$ is said to be \define{primitive} if  belongs to some basis
    of $\bbF$, or equivalently if it generates a cyclic free factor in some free product decomposition.

    We  consider graphs of groups in which vertex groups are either free groups or direct products of free groups with $\mathbb{Z}$, and edge groups are either cyclic or isomorphic to $\mathbb{Z}^2$ and maximal, as in Corollary \ref{coro;hybrid_canonical_for_sbgp}. For short we will call them trivially partially suspended.

    For such a graph-of-groups $\bbX$, we denote by $\bbX_v = \bbF_v \times \langle t_v \rangle$, with $\bbF_v$ free, and $t_v$ either of infinite order or trivial.

    Following \cite{wise_subgroup_2000}, we say that  a  graph-of-groups in this class is  \emph{clean}  if, for all edges $e$, the intersection $\bbF_{\tau(e)}\cap \tau_e(\bbX_e)$ is a cyclic free factor in $\bbF_{\tau(e)}$.

    \begin{prop}\label{prop;clean}
      For any trivially partially suspended  graph-of-groups (as defined above), there exists a computable characteristic finite index subgroup $\Gamma_0 $ of $\Gamma= \pi_1(\bbX, v_0)$ such that the graph of groups decomposition $\bbY$ of $\Gamma_0$ given by its action on the Bass-Serre tree dual to $\bbX$ is trivially partially suspended,  and clean.
    \end{prop}
    We will follow Wise's argument for graphs of free groups, in \cite{wise_subgroup_2000}.  Before proving the proposition, we first need a few lemmas.
    
\begin{lem}\label{lem;plenty_primitives}
  Let $c_1,\ldots,c_k$ be a collection of elements of a free group
  $\bbF$, then there is a computable finite index normal subgroup $K \leq \bbF$
  such that each for any $f \in \bbF$ the
  intersection $  K \cap \langle c_i^f \rangle$   is a free factor of $K$.
  \end{lem}
  \begin{proof}
    By the Marshall Hall Theorem \cite{hall_subgroups_1949} for
    $i=1,\ldots,k$ there are finite index subgroups $K_i \leq \bbF$
    such that $K_i\cap \bk{c_i}$ is primitive. Since $[\bbF:K_i]<
    \infty$, the intersection\[
       K=\bigcap_{f \in \bbF, \\ i \in \{1, \dots k\}} K_i^f
    \] only involves finitely many terms so it is a finite index normal
    subgroup. 
    
    For each $i$ and $f$, the group $\langle c_i^f\rangle \cap K_i^f$ is a free factor in $K_i^f$, hence maximal-elliptic in some Bass-Serre  $K_i^f$-tree associated to a free product of cyclic groups decomposing $K_i^f$. Therefore   $\langle c_i^f\rangle \cap K$ is maximal-elliptic in this tree and is a cyclic free factor of $K$. 
  \end{proof}

  Let $K_v \fidx\leq \bbF_v$ be a finite
  index subgroup and let $e \in \edges X$ such that $\tau(e)=v$. In our notations, 
  $\tau_e(\bbX_e) \cap F_n = \bk{c_e}$. Analogously to
  \cite{wise_subgroup_2000}, we call each  $K_v$-conjugacy class of a  non-trivial intersection
  with a conjugate $K_v \cap \bk{c_e^f}$ an \emph{$e$-elevation} we
  denote the set of $e$-elevations as
  \[\elev{e}{K_v} = \{ \bk{\gamma} \mid\exists f \in F_v, \bk{\gamma} = K_v
    \cap \bk{c_e^f}\}.\] Given an $e$-elevation $\bk{\gamma} \in \elev
  e{K_v}$ we define its \emph{length} to be the number $\ell_\bk\gamma$
  such that \[
    \gamma^f = c_e^{\ell_\bk\gamma}
  \] for some $f \in F_v$. An easy finite degree covering space
  argument gives the following:
  \begin{lem}[Index and lengths]\label{lem:deg-length}
    The sum of the lengths of elevations of an edge equals the
    index. That is to say, if $\tau(e) = v$
    then\[ \sum_{[\bk{\gamma}]_{K_v}\in \elev e{K_v}} \ell_\bk\gamma = [\bbF_v:K_v].
      \]
  \end{lem}

        We may now prove Proposition \ref{prop;clean}
        
        \begin{proof}        
        For each $v$, by Lemma \ref{lem;plenty_primitives}, there exists a $P_v$ a normal subgroup of finite index in $\bbF_v$ such that every elevation of adjacent edge group is a cyclic free factor in $P_v$.  Let $l_\gamma$ the length of the elevation $\bk\gamma$ of an edge group, and $l_0$ the least common multiple of all the $l_\gamma$, ranging over the images of all edge groups in all vertex groups. We apply Wise's omnipotence Theorem \cite[Theorem 3.5]{wise_subgroup_2000} to the collection of conjugacy classes of the generators $\gamma$ of elevations $\bk\gamma$: there exists a constant $L_0$ such that for each tuple of elevations of edge groups, there exists a finite quotient of $P_v$ in which the generator $\gamma$ has order $L_0\times l_0/l_\gamma$. Let $K_v\leq P_v$ be the kernel of this finite quotient. By choice of constants, viewing every $K_v \leq \freegp v$, we have that all elevations have length $L_0\times l_0 =L$.
        
        Now $\bbX_v$ either has an infinite centre but isn't cyclic, in which case it is isomorphic to $\freegp x \times \bk{t_v}$, or $\bbX_v$ is a subgroup of a free group. In all cases, $t_v$ is a generator of the centre of $\bbX_v$ if $\bbX_v$ isn't cyclic and we allow $t_v=1$ if $\bbX_v$ is a subgroup of a free group.  We consider the subgroup $H_v= \bk{K_v,t_v^L} \leq \bbX_v $. It has finite index in $\bbX_v$, and its non-trivial intersections with each conjugate of an edge group $\tau_e(\bbX_e)^g$, for $g\in \bbX_v$, is $L \tau_e(\bbX_e)^g $ (in the notation of abelian groups, as are the  $\tau_e(\bbX_e)^g$). In particular, each intersection is a characteristic subgroup of $\tau_e(\bbX_e)^g $.
        
        We will now construct the graph of groups $\bbY$ with the required properties by constructing a covering space of $\calX$ a standard graph of spaces underlying $\bbX$ as in \cite{scott_topological_1979}. Recall that $X$ is the graph underlying $\bbX$.
        For each vertex $v \in \verts X$ let  $E_v = \frac{[F_v:K_v]}{L}$, and for each edge $e \in \edges X$ we have, by Lemma \ref{lem:deg-length}, that if $i(e)=v$ and $\tau(e)=i(\bar e)=w$ then
        \[|\elev{e}{K_v}| = E_v \textrm{~and~}|\elev{\bar e}{K_w}| =E_w\]
    
     For each vertex $v$, let  
    \[d_v = \frac{\mathrm{lcm}\left(E_w, w\in \verts X\right)}{E_v},\]
    so that for any $u,v \in \verts X$ we have \[ d_uE_u=d_vE_v.
    \] 
    
    We thus define a graph $Y$ by taking as a set of vertices $d_v$ copies of the vertex $v\in \verts X$, for each $v$. We call them lifts of $v$.
    
    For each pair $(u,v)$ related by an edge $\{e, \bar e\}$ in $X$ we thus define, as edges of $Y$,    $d_uE_u$ copies of the edge $\{e, \bar e\}$, attached to the vertices that are lifts of $u$ and $v$ so that there are $E_u$ edges attached to each of the lifts of $u$, and $E_v$ attached to each of the lifts of $v$. Since $d_uE_u=d_vE_v$ is it possible to pair the edges.
    
    Each lift of a vertex $u$ is endowed with a copy of the group $H_u$. Each lift of an edge $e$ is endowed with the group  $L \tau_e(\bbX_e)$. The attaching maps are then natural, considering the elevations of the edges groups of $\bbX$ in our set of vertex groups. 
    
    This thus defines a graph of groups $\bbY$. It is easy to see that there is a natural homomorphism $\pi_1(\bbY, u_0) \to \pi_1(\bbX, u)$ for $u_0$ a lift of $u$, and that it is injective. A classical consideration of graph of spaces associated to $\bbY$, covering a graph of spaces associated to $\bbX$ with $d_{u_0}E_{u_0}$-sheets,    reveals that the image has finite index (see also \cite{scott_topological_1979,wise_subgroup_2000}).
    Finally, intersecting the conjugates of the image of  $\pi_1(\bbY, u_0)$ gives a characteristic finite index subgroup of $\pi_1(\bbX,u)$ that is still clean, and trivially partially suspended.

  \end{proof}

        \begin{prop}[Virtually free vertex fillings]\label{prop:vertex-fill}

        Let $\bbX$ be a trivially partially suspended graph-of-groups, with conventions of notation as above.
    
        Assume furthermore that $\bbX$ is clean, and for each edge $e$, let  $c_e$ be a generator of  $\tau_e(\bbX_e) \cap \bbF_e$ (hence a primitive element).

        Let, for each $v$, $N_v$ be a finite index subgroup of $\bbF_v$. 
        
        Then, for each vertex $v$, there is a computable characteristic finite quotient  $p_v: \bbX_v \to Q_v$ for which $\ker (p_v)\cap \bbF_v  $ is in $N_v$, and such that:
        for each edge $e$, the maps $t: \bbX_e \to \bbX_{\tau(e)} \to Q_{\tau(e)}$ and $i_e: \bbX_e \to \bbX_{i(e)} \to Q_{i(e)}$ both factorize through a finite quotient $\bbX_e \to Q_e$ that is isomorphic to the images in $Q_{\tau(e)}$ and $Q_{i(e)}$.

        In particular, there is a (computable) graph-of-group homomorphism $\bbX \to \bbX_Q$, obtained by substituting the vertex and edge groups by $Q_v, Q_e$,  for which $\pi_1(\bbX_Q, v_0)$ is virtually free.
      
        \end{prop}

        The argument presented below is taken from the proof of
        \cite[Lemma 4.2]{cotton-barratt_conjugacy_2012}.
        \begin{proof}
            First, after possibly taking the intersection in $\bbF_v$ of the conjugates of $N_v$, we may assume that each $N_v$ is normal.

            Consider the free product\[
            \bbF_* = \bigast_{w \in \verts X} \bbF_w
            \] which is itself a free group.  Consider  $\ker ( \bbF_* \to \bbF_v \to \bbF_v/ N_n)$ which is a finite index subgroup, take the intersection $K_0$ of all of them for $v$ ranging in $\verts X$, and   take a characteristic finite index subgroup of $K_0$. It intersects each $\bbF_v$ inside $N_v$.

            We claim that there is $N$ such that, for each vertex $v$, and for each incident edge group $e$, $\tau_e(\bbX_e) \cap K_0$ has index index  $N$ in   $\tau_e(\bbX_e) \cap \bbF_v$.

            We now note that by hypothesis on $\Xi=\pi_1(\bbX)$, 
            for any $e \in \edges X$ the generator $c_e$ of 
            $t (\bbX_e )\cap \bbF_{\tau(e)}$ 
            is a primitive element of $\bbF_{\tau(e)}$, hence of $\bbF_*$ as well. It, therefore, follows that for
            any pair $e,f \in \edges Y$ there is an automorphism
            $\phi \in \Aut{\bbF_*}$ such that $\phi(c_e) =c_f$.  Since  
            $K_0\charsgp \bbF$ is characteristic, the order of  $c_e$ and of $c_f$ in the quotient $\bbF_*/K_0$ are the same:  there exists some
            $N \in \bbZ$ such that for each $e \in \edges X$
                \[K\cap \bk{c_e} = c_e^N.\] 
    
           Recall that according to our notations, for each vertex $v$, $\bbX_v = \bbF_v \times \langle t_v \rangle$, with  $t_v$ of infinite order if $\bbX_v$ is not free, and with $t_v=1$ if $\bbX_v$ is free.   
   
            Setting $K_v = K\cap \bbF_v \triangleleft \bbF_v$, and $Q_v = \bbX_v / (K_v \times \langle t_v^N \rangle)$, and $Q_e = \bbX_e/ N\bbX_e$ (as abelian groups) produces the desired graph of groups $\bbX_Q$, with all the natural attaching maps. 
    
        \end{proof}
    
            \subsubsection{Separation in virtually free vertex fillings}
    
    Show that certain conjugacy classes of $\Gamma$ are separated in a preferred, characteristic,   virtually free vertex filling  of $\bbX$.  
    For that, we begin by expressing the combined effect of clean covers and vertex filling.
    
    Proposition \ref{prop;clean-filled_0} below is of interest and close to what we need. However, we will need a refinement that is harder to parse, but of a similar spirit, that implies this proposition. We state the easier result in order to indicate the direction we are taking.

     \begin{prop}\label{prop;clean-filled_0}
      Let $\Gamma$ be the fundamental group of a  trivially partially suspended  graph of groups $\bbX$. For any pair $(\gamma_1, \gamma_2)$ of non-conjugate elements of $\Gamma$ that are in vertex groups, but not conjugate to images of edge groups,
      there exists a normal (characteristic) finite index subgroup $\Xi$ of $\Gamma$,  
      $m \geq 1$  such that $\gamma_1^m, \gamma_2^m$ are in $\Xi$, and
      there exists a virtually free quotient $V$ of $\Xi$ in which the   elements  $\gamma_1^m$ and $(\gamma_2^m) \in \Xi$ map on non-conjugate elements of $V$.
      
    \end{prop}
    
    As mentioned this statement, however archetypal, is not sufficient for our needs. Recall that prior to the proof of Proposition \ref{prop:iso-problem-for-suspensions} we defined the \define{peripheral structure} of a vertex group $\bbX_v$ of a graph of groups $\bbX$ to be the collection in $\bbX_v$ of conjugates of the images of the incident edge groups. We will define a \define{peripheral automorphism} to be an automorphism of $\bbX_v$ that preserves the peripheral structure. We will need a single quotient that separates multiple pairs simultaneously. For that, we need the following definition.
    
    \begin{defn}
      In a graph of group $\bbX$, equipped with a characteristic
      quotient $\pi_1 (\bbX, v_0)\to Q$, a pair of elements $g_1, g_2$ belonging to vertex groups $\bbX_{v_1}, \bbX_{v_2}$ but not conjugate into the images of the incident edge groups, is said to be \emph{strongly  $Q$-conjugacy-separated}  if
      
      \begin{itemize}
          \item if $v_1= v_2$,  then any of the pairs $(g'_1, g'_2)$ that are images of $(g_1, g_2)$ by a peripheral automorphism of the vertex group  $\bbX_{v_1}$,  maps to a pair $(\bar g'_1, \bar g'_2)$ of non-conjugate elements in $Q$,
          \item if $v_1\neq v_2$,  then any of their images by a pair of peripheral automorphisms of $\bbX_{v_1}, \bbX_{v_2}$, maps to a pair of non-conjugate elements in $Q$.
      \end{itemize}
    \end{defn}

    Here is the result we want to establish.
    
     \begin{prop}\label{prop;clean-filled}
      Let $\Gamma$ be the fundamental group of a  partially suspended  graph of groups $\bbX$. For any pair $(\gamma_1, \gamma_2)$ of non-conjugate elements of $\Gamma$ that are in vertex groups, but not conjugate to images of edge groups,
      there exists a computable characteristic finite index subgroup $\Xi$ of $\Gamma$,  
      $m \geq 1$  such that $\gamma_1^m, \gamma_2^m$ are in $\Xi$, and
      there exists a virtually free quotient $ \Xi\to V$ of $\Xi$ in which, for all $\gamma \in \Gamma$, the  
      pair $(\gamma_1^m, (\gamma_2^m)^\gamma)$ of elements of  $ \Xi$ is strongly $V$-conjugacy-separated.

    \end{prop}
    \begin{proof}
    
        Let $\Xi$ be a characteristic (normal) finite index subgroup corresponding to a clean cover $\bbY$ of $\bbX$, as provided by Proposition \ref{prop;clean}.

       Consider  $\gamma_1, \gamma_2$ as in the statement.  Since  $\Xi$ is of finite index,  there exists an integer $m\geq 1$ such that  $\gamma_1^m, \gamma_2^m$ are in $\Xi$.  Note that the conjugacy class of $\gamma_2^m$ in $\Gamma$ is contained in $\Xi$ (by normality), and breaks into finitely many $\Xi$-conjugacy classes, given by $(\gamma_2^m)^\mu \in \Xi$, for $\mu$ ranging over a set $M$ of representatives of right cosets of $\Xi$ in $\Gamma$. Since  $\gamma_1, \gamma_2$ are not conjugate in $\Gamma$, none of the $(\gamma_2^m)^\mu $ is conjugate to $\gamma_1$ in $\Xi$.
       
       We are looking for a virtually free vertex filling of $\bbY$ that  strongly separates, for each $\mu\in M$,   the conjugacy classes of  $\gamma_1^m$ and of  $ (\gamma_2^m)^\mu$, and any of their automorphic images.  Since there are only finitely many conjugacy classes $ (\gamma_2^m)^\mu$, we may argue independently for each of them, in order to find finite index subgroups in the vertex groups that suitably define a vertex filling, and then take their finite intersection in order to obtain a vertex filling suitable for each of them.
       This reduces our study to   the following statement, in which the notation is simplified.
       
       \begin{lem}
         Let $\bbY$ be a clean  trivially partially suspended graph of groups, and $g_0$ an element of a vertex group $\bbY_{v_0}$, and 
      let   $g_1$  be an element of a vertex group $\bbY_{v_1}$, such that neither $g_i$ is conjugate to the other, neither to an adjacent edge group.

         Then there exists a computable virtually free vertex filling of $\bbY$ that strongly separates the pair $(g_0, g_1)$. 
       \end{lem}
       \begin{proof}

    We distinguish whether $g_1$ is in the same vertex group as $g_0$ or not.

       Assume first that they are in the same vertex group $\bbY_v$.
       Virtually free groups are conjugacy separable, in particular, one can separate the conjugacy class of all the elements of the vertex groups in finite quotients \cite{dyer_separating_1979}\footnote{Stebe \cite{stebe_residual_1970}, Remeslennikov \cite{remeslennikov_groups_1972}, and Wehrfritz \cite{wehrfritz} proved with different methods the conjugacy separability of infinite order elements, while Dyer  \cite{dyer_separating_1979} treated the case of finite order elements in virtually free groups, which is the case we are using.}. 
       Therefore,  there exists a normal subgroup $N_v \triangleleft \bbY_v$ such that their image in $\bbY_v / N_n$
       are still non-conjugate.  It can be computed by enumeration.  We choose $N_v$ characteristic by taking the intersection of its images by automorphisms of $\bbY_v$. Since it is an intersection of given index subgroups, it is finite index. Now, all automorphisms of $\bbY_v$ descend to the quotient $\bbY_v\to \bbY_v/N_v$. Therefore, since $g_0, g_1$ are conjugacy separated in one quotient, all the $\psi(g_0), \psi(g_1) $, for $\psi$ an automorphism of $\bbY_v$,  are conjugacy separated in $\bbY_v/N_v$.   
       Applying Proposition \ref{prop:vertex-fill}, one gets a virtually free quotient of $\pi_1(\bbY, v_0)$ in which $\bbY_v$ maps on a group that still quotient further on  $\bbY_v / N_n$. The pair $(g_0, g_1)$ has thus been strongly separated in a single virtually free vertex filling.

       Assume that  $g_0, g_1$ are not conjugate in the same  vertex groups of $\bbY$ (but are respectively in vertex groups $\bbY_{v_0}, \bbY_{v_1}$ of $\bbY$). By assumption, they are not conjugate in the edge groups adjacent to $v_0$ and $v_1$.

       By a result of Wilson and Zalesskii  \cite[Prop 2.5]{wilson_conjugacy_1998},  finitely generated subgroups of virtually free groups are conjugacy distinguished: given two (or finitely many) of them that are non-conjugate, there is a finite quotient in which they remain non-conjugate. It can be computed by enumeration.     
       
       Applying this result to the two direct products $\bbY_{v_0}$, $\bbY_{v_1}$,  their adjacent edge groups and the group generated by the elements $g_0, g_i$,  there exist finite quotients of $\bbY_{v_0}$ and of $\bbY_{v_1}$, in which $g_0$ is not conjugate in  the images of any of the adjacent edge groups, and similarly for $\bbY_{v_i}$. Let $N'_{v_0}, N'_{v_1}$ the respective kernels of these quotients, and $N_{v_0}, N_{v_1}$ the intersection of all their images by peripheral automorphisms.   We can now apply  Proposition \ref{prop:vertex-fill}, to get a virtually free quotient of $\pi_1(\bbY, v_0)$, that has a graph of group decomposition with finite vertex groups,  in which the images of $g_0$ and $g_i$ are in two distinct vertex groups, and in no adjacent edge group. This ensures that they are not conjugate in this virtually free quotient. A similar argument using the property that automorphisms of $\bbY_{v_i}$ commute with the quotient $\bbY_i\to \bbY_i/N_{v_i}$ ensures that the same is true for all the pairs obtained as images by such automorphisms of vertex groups.
       
       Hence, the pair $(g_0, g_1)$ is strongly separated in the obtained virtually free vertex filling.

       \end{proof}
    As already mentioned, the Lemma ensures the Proposition. 
       \end{proof}
       
       We end with the following corollary.
       Recall that by the Nielsen Realization Theorem, each finitely generated free group, and immediately each direct product of a finitely generated free group with a cyclic group has only finitely many conjugacy classes of outer automorphisms of finite order.
       
       By a result of Grossman \cite{grossman_residual_1974}, (see also Lubotzky \cite{lubotzky_normal_1980}), pointwise-inner automorphisms of a free group are inner.  It follows that pointwise inner automorphisms of direct products of free groups with a cyclic group are inner.    Hence   for each non-trivial outer automorphism of a group that is free or a direct product of free and cyclic, there exists a conjugacy class of elements  (actually infinitely many) that is sent on another by the automorphism.
       Call them witnesses of non-triviality.

       \begin{cor}\label{cor:all_survivors_in_VF}
         Let $\Gamma$ be the fundamental group of a  trivially partially suspended graph of groups $\bbX$, let $\bbY$ be a characteristic clean cover of $\bbX$, $\Xi$ its fundamental group, and $M$ a set of right coset representatives of $\Xi$ in $\Gamma$. 
    
        For each white vertex group $\bbY_v$ of $\bbY$, let $\calS^0_v$ consist of a finite set of witnesses for a representative of each conjugacy class of non-trivial, finite order automorphism, that are in addition,  not conjugate in any adjacent edge group. We choose $\calS_v^0$ in such a way that if $\bbY_v$ is isomorphic to $\bbY_w$ by an isomorphism preserving the peripheral structure of the adjacent edge groups, then there is such an isomorphism sending $\calS_v^0$ to $\calS_w^0$.
        
        For each white vertex $v$ in $Y$, let  $\calS_v $ be the set of images of $\calS_v^0$ by the group of peripheral  
        automorphisms of $\bbY_v$ (for the edge groups structure). 
        
         Let $\calS$ be the union of the $\calS_v$.
         
         Then there exists a computable characteristic vertex filling virtually free group in which the elements of $\calS$, and their conjugates by elements of $M$ map on non-trivial torsion elements, and such that whenever two of them are non-conjugate in $\pi_1(\bbY, v_0)$, they are not conjugate in the virtually free quotient.
         
       \end{cor}
      
      \begin{proof}
      There are only finitely many orbits (by automorphisms of vertex groups) of elements in $\calS$, therefore one may apply Proposition \ref{prop;clean-filled} for each pair of those, and obtain a vertex filling suitable for this pair.  Take the intersection of the kernels of these finitely many fillings. It is of finite index in each vertex group. Take the intersection of all images by automorphisms: it is still of finite index in each vertex group since it is an intersection of same index subgroups,  and it is now characteristic. Take the quotient by this intersection. It quotients further on each of the vertex fillings. 
      Therefore, it strongly separate every pair, and is still a vertex filling by finite index subgroups in the vertex groups. Hence it is suitable for the conclusion.
     \end{proof}

            \subsubsection{Separation of many conjugacy classes}

    The following observation does not depend on the previous context, but since it will be applied to our previous construction, we take matching notations. The argument 
    was inspired by \cite[Lemmas 2.7, 3.1]{cotton-barratt_conjugacy_2012}, but the arguments of the previous section let us avoid needing to deploy profinite machinery.

  Let $\Gamma$ be a finitely generated group, $\Xi <\Gamma$ a characteristic subgroup, $\Xi \to V$ a characteristic virtually free quotient of $\Xi$, and let $\calC_V$ be the (finite) set of conjugacy classes of elements of finite order in $V$.
 
 If $\xi \in \Xi$, we denote by $[\xi]_V$ the $V$-conjugacy class of its image in the quotient $V$.  We will be interested in elements that map on finite order elements in $V$, thus defining this way an element in $\calC_V$. 
 
 Let $M$ be a set of right coset representatives of $\Xi$ in $\Gamma$.

\begin{prop}\label{prop;all_survivors_in_finite}

  Let $\Gamma, \Xi \triangleleft \Gamma$, with    $\Xi \to V$, and $\calC_V$, and $M$ be as above.

Then there exists  a computable finite quotient $q: \Gamma \to Q$ of $\Gamma$, with kernel in $\Xi$, whose restriction to $\Xi$ factorizes through $V$,  such that whenever $\xi_1, \xi_2 \in \Xi$ are such that,  for all $\mu \in M$,   $[\xi_1]_V, [ (\xi_2)^\mu]_V $ are in $\calC_V$ and are different, one has  $[q(\xi_1)]_Q \neq  [q(\xi_2)]_Q$.
\end{prop}

The situation is illustrated in the diagram:
\[
\begin{array}{ccccccc}
    \Gamma  &   & \stackrel{q}{\longrightarrow} &  & Q & {\scriptstyle \hbox{{\footnotesize question: }} [q(\xi_1)]_Q \neq  [q(\xi_2)]_Q ?} \\
        \triangledown    & & & & \vee \\
        \Xi & \to & V & \to &Q_0 & {\scriptstyle \hbox{{\footnotesize by construction: }} [q(\xi_1)]_{Q_0} \neq  [q({\xi_2}^\mu)]_{Q_0}}\\
        {\scriptstyle \xi_1, \xi_2}  & & {\scriptstyle [\xi_1]_V\neq [\xi_2^\mu]_V } &  
\end{array}
\]

\begin{proof}

Recall that $V$ is conjugacy separated, in particular, its finite order conjugacy classes are separated in finite quotients \cite{dyer_separating_1979}.  We may thus find a finite quotient $Q_0$ of $V$ in which all finite  finite order conjugacy classes of $V$ map on different conjugacy classes. This defines a quotient map $\Xi \to Q_0$ with kernel of finite index in $\Xi$. Up to reducing the kernel (thus keeping the separation property), we may as well choose it so that the kernel in $\Xi$ is characteristic and finite index. It is therefore normal in $\Gamma$, and defines a finite quotient $q:\Gamma \onto Q$ with same kernel (hence in $\Xi$).

Assume that $\xi_1, \xi_2$ in $\Xi$ are as in the statement. To prove that the $Q$-conjugacy classes of their images are distinct, we proceed by contradiction, assuming $[q(\xi_1)]_Q= [q(\xi_2)]_Q$.   

Reformulating this equality: there is $\mu \in M$ and $\zeta \in \Xi$ such that $q(\xi_1)=q(\xi_2^{  \mu\zeta})$. Lifting this in $\Gamma$, we have that there exists  
$\zeta'\in \ker q$ such that $\xi_1=\xi_2^{ \mu\zeta} \zeta'$. We rewrite this as $\xi_1\in (\xi_2^{ \mu})^{\zeta}\cdot \ker q$. 

Recall that $\xi_1, \xi_2, \zeta$, and  $\xi_2^\mu$  are all in $\Xi$.   Since $\ker q  = \ker (\Xi\to Q_0) <\Xi$, 
this relation descends in $Q_0$ (through $V$) as 
$[q(\xi_1)]_{Q_0}=[q(\xi_2^{ \mu})]_{Q_0}$. However, this is contrary to our assumption on $\xi_1, \xi_2$ and the fact that $\Xi \to Q_0$ separates all conjugacy classes in $\calC_v$.


\end{proof} 

            \subsubsection{Minkowskian congruence}

            \begin{prop}
              If $\Gamma$ is the fundamental group of a trivially partially suspended graph of group $\bbX$, then, for a characteristic clean covering $\bbY$ of $\bbX$, with fundamental group $\Xi \triangleleft \Gamma$, a virtually free vertex filling $\Xi \to V$ of $\Xi$ satisfying Corollary \ref{cor:all_survivors_in_VF}, and a characteristic finite quotient $Q$ of $\Gamma$ satisfying Proposition \ref{prop;all_survivors_in_finite}, then for each element $\phi$ of $\delta \aut \bbX$ that is non-trivial, and of finite order in $\out \Gamma$, the induced automorphism of $Q$ by $\phi$ is non-trivial in $\out Q$.  
            \end{prop}
            
                \begin{proof}
                Let $\bbX, \Gamma,\Xi$  and $\phi$ be in the statement of the proposition. $\phi$  induces an automorphism of $\bbY$ as well, since it is characteristic.
                
                First, we check that $\phi|_\Xi$ has finite order in $\out \Xi$.  A power $\phi'$ of $\phi$ is inner in $\Gamma$, hence is $\ad{g}$ for some $g\in \Gamma$. Therefore, $(\phi')^m= \ad{g^m}$, and for  $m= [\Gamma; \Xi]$, $g^m\in \Xi$. 
                In particular a power of $\phi|_\Xi$ is inner in $\Xi$.
                
                We next check that $\phi|_\Xi$ is non-trivial in $\out \Xi$:
                if $\phi$ itself is inner in $\Xi$, then using uniqueness of roots, one can easily check that the same conjugator makes it inner in $\Gamma$.

                Denote by $\phi= (\phi_Y, (\phi_v), (\phi_e), \gamma_e)$ the obtained automorphism of $\bbY$.
                By Proposition \ref{prop;2cases},  there are two cases for $\phi$ in $\delta \aut \bbY$. 
            
                The first case is when $\phi_Y$ is non-trivial. Then let $v$ be a white vertex of $Y$ such that $\phi_Y(v)= w \neq v$. Then $\bbY_v, \bbY_w$ with their adjacent edges peripheral structures are isomorphic.

                If $g_v \in \calS_v^0$ (from Corollary \ref{cor:all_survivors_in_VF}) then it is not conjugate to $\phi_v(g_v)$, which by construction is also in $\calS$. By that corollary, both are then sent on finite order non-conjugate elements of the quotient $V$ of the corollary. The Proposition \ref{prop;all_survivors_in_finite} then ensures that they are sent to non-conjugate elements of the finite quotient $Q$ of $\Gamma$, therefore $\phi$ induces on $Q$ an automorphism that does not preserve a certain conjugacy class. It is therefore non-trivial.

                If $\phi_Y$ is trivial, there must be some $v$, such that $\phi_v$ is non-trivial and of finite order in $\out{\bbY_v}$. By the construction in Corollary \ref{cor:all_survivors_in_VF},  there exists an element in $\calS_v$ that is sent on a non-conjugate element. Again, by that Corollary, it is sent on finite order non-conjugate elements of the quotient $V$ of the corollary. The Proposition \ref{prop;all_survivors_in_finite} again allows us to conclude in the same way. 
                \end{proof}
            
                We just proved Theorem \ref{theo;TPSC_Mink} by the proposition, and therefore we proved Theorem \ref{theo;SPTS_M} as well.




{\small
\bibliographystyle{alpha} \bibliography{unipotent_linear_suspensions}

\noindent {\sc Fran\c{c}ois Dahmani, Institut Fourier, Universit\'e Grenoble Alpes, 38058 Grenoble cedex 9, France.}\\
e-mail. {\tt francois.dahmani@univ-grenoble-alpes.fr} \\  \url{https://www-fourier.univ-grenoble-alpes.fr/~dahmani}\\

\noindent {\sc Nicholas Touikan, Department of Mathematics and Statistics, University of New Brunswick, Fredericton, N.B.,  Canada, E3B 5A3}\\
e-mail. {\tt ntouikan@unb.ca}\\
\url{https://ntouikan.ext.unb.ca}

}


\end{document}